\documentclass{article}
\usepackage[utf8]{inputenc}
\usepackage[left=3cm,right=3cm,top=3cm,bottom=3cm]{geometry}
\usepackage{graphicx}
\usepackage{enumitem}
\usepackage{color}
\usepackage{amsmath,amsfonts,amssymb,amsthm,epsfig,epstopdf,titling,url,array}
\usepackage{mathrsfs}
\usepackage{caption}
\usepackage{subcaption}
\usepackage{tikz-cd}
\usepackage{MnSymbol}

\DeclareMathAlphabet{\mathpzc}{OT1}{pzc}{m}{it}

\DeclareMathOperator{\dom}{Dom}
\DeclareMathOperator{\fix}{Fix}
\DeclareMathOperator{\per}{Per}
\DeclareMathOperator{\supp}{Supp}
\DeclareMathOperator{\diam}{diam}
\DeclareMathOperator{\leb}{Leb}
\newcommand{\Mane}{Mañé }

\def\XXint#1#2#3{{\setbox0=\hbox{$#1{#2#3}{\int}$ }
\vcenter{\hbox{$#2#3$ }}\kern-.6\wd0}}

\newtheorem{theo}{Theorem}[section]
\newtheorem{prop}[theo]{Proposition}
\newtheorem{lem}[theo]{Lemma}
\newtheorem{cor}[theo]{Corollary}

\theoremstyle{definition}
\newtheorem{defi}[theo]{Definition}

\newtheorem{rem}[theo]{Remark}

\theoremstyle{remark}

\numberwithin{equation}{section}

\title{Representation of Global Viscosity Solutions for Tonelli Hamiltonians}
\author{CHARFI Skander}
\date{}

\begin{document}

\maketitle

\begin{abstract} 
	We consider the Lax-Oleinik operator $\mathcal{T}$ associated with the non-stationary Hamilton-Jacobi equation $\partial_tu + H(t,x,\partial_xu) = \alpha_0$ for a Tonelli Hamiltonian $H$ and its \Mane critical value $\alpha_0$. It is known from the work of A. Fathi and J.N. Mather \cite{MR1792479} that the convergence of this semigroup fails in the non-autonomous framework.   
	
	In this context, we study the action of $\mathcal{T}$ on its non-wandering set $\Omega(\mathcal{T})$. First, we show that $\mathcal{T}$ acts as an isometry on this set, and then we characterize $\Omega(\mathcal{T})$ as the set of global viscosity solutions of the Hamilton-Jacobi equation, i.e. solutions that are defined for all real times.\\  
	
	Next, we introduce a generalized Peierls barrier $\underline{k}$ and a set of generalized static classes $\underline{\mathbb{M}}$ within the Mather set. Using these, we represent elements $u$ of $\Omega(\mathcal{T})$ as
	\begin{equation*}
		u(x) = \inf_{y \in \underline{\mathbb{M}}} \{ u(y) + \underline{k}(y,x) \}
	\end{equation*}
	
	We apply this representation formula to prove Fathi's convergence theorem for autonomous systems and provide a representation formula for $n$-periodic viscosity solutions. Additionally, we establish that the dynamics of non-wandering viscosity solutions are governed by the Lagrangian flow on the Mather set. Specifically, we show that if the Mather set consists solely of $N$-periodic orbits for some integer $N$, then all non-wandering viscosity solutions are $N$-periodic. Furthermore, we show that if the restriction of the Lagrangian flow to the Mather set is uniformly recurrent for a time sequence $p_n$, then all non-wandering viscosity solutions are uniformly recurrent for the same time sequence $p_n$.
\end{abstract}

\tableofcontents

\newpage

\section{Introduction}

\subsection{Overview}

\par Let $M$ be a connected closed manifold. Denote by $TM$ its tangent bundle and by $T^*M$ its cotangent bundle with their respective projections $\pi_{TM} : TM \to M$ and $\pi_{T^*M} : T^*M \to M$. Unless it is ambiguous, these projections will be simply denoted by $\pi$. Denote by $\mathcal{C}(M,\mathbb{R})$ the set of continuous scalar maps on $M$ endowed with its usual infinite norm $\Vert.\Vert_\infty$ given by $||u||_\infty = \sup_{x \in M} |u(x)|$. The set $\mathbb{T}^1 = \mathbb{R} / \mathbb{Z}$ refers to the circle.\\

Let $H(t,x,p) : \mathbb{T}^1 \times T^*M \to \mathbb{R}$ be a $1$-time periodic Tonelli Hamiltonian (See Definition \ref{TonelliDef}) and let $\alpha_0$ be the critical \Mane value associated to $H$ (See Definition \ref{LODef}). We consider the \textit{Hamilton-Jacobi equation}

\begin{equation} \label{HJalpha}
	\partial_tu + H(t,x,\partial_xu) = \alpha_0
\end{equation} 

\par This equation has a well posed Cauchy problem in the viscosity sense (see \cite{MR667669}), \cite{MR0690039}). And A.Fathi \cite{MR1451248} developed the weak-KAM theory which states that in the Tonelli framework, these solutions are generated by the Lax-Oleinik operator $\mathcal{T}^{s,t}$ defined on $\mathcal{C}(M,\mathbb{R})$ (see Definition \ref{LODef}). More specifically, for any real time $s \in \mathbb{R}$ and any $u_0 \in \mathcal{C}(M,\mathbb{R})$, there exists a unique viscosity solution $u(t,x):[s,+\infty) \times M \to \mathbb{R}$ of (\ref{HJ}) with $u(s,\cdot) = u_0$ defined by $u(t,x) = \mathcal{T}^{s,t}u_0(x)$.\\

The weak-KAM theory, initially developed in the autonomous framework by A.Fathi, focused on weak-KAM solutions, which are fixed points of the operator $\mathcal{T}^{0,1}$. These solutions correspond to viscosity solutions that are one-periodic in time. Furthermore, Fathi explicited examples of weak-KAM solutions using the Peierls barrier $h^\infty(x,\cdot)$ (see Definition \ref{pPeierlsDefi}) and established various properties of weak-KAM solutions, such as their calibration on the Mather set (see Definitions \ref{CalibDefi} and \ref{MatherDefi}) and domination by the barrier $h^\infty$ on the Mather set (see Definition \ref{DominationDef}). These were foundational for subsequent results on weak-KAM solutions, including their $C^{1,1}$-regularity when restricted to the Aubry set, uniqueness on the Mather set, and representation through the Peierls barrier (see \cite{fathi2008weak}).\\

The representation formula provides a description of any weak-KAM solutions $u$ using the Peierls barrier $h^\infty$ as follows 
\begin{equation} \label{IntroRepresentation}
	u(x) = \inf_{y \in \mathbb{M}} \{ \psi(y) + h^\infty(y,\cdot) \}
\end{equation}
where $\mathbb{M}$ is a subset of the Mather set called the set of static classes (See Definition \ref{StaticClassesDefi}) and the map $\psi : \mathcal{M} \to \mathbb{R}$ is dominated by the Peierls barrier $h^\infty$. This formula was later extended for non-autonomous Tonelli Hamiltonians by G. Contreras, R. Iturriaga, and H. S\'anchez-Morgado in \cite{CIS}.\\

In the non-autonomous framework, P. Bernard and J.-M. Roquejoffre \cite{MR2041603} demonstrated that the calibration on the Mather set applies to non-wandering viscosity solutions of the Hamilton-Jacobi equation \ref{HJalpha}. This crucial result implies that non-wandering viscosity solutions share similar properties with weak-KAM solutions, such as uniqueness on the Mather set (see Theorem \ref{CalibNW}) and $C^{1,1}$-regularity when restricted to the Aubry set.\\

As a result, it is reasonable to consider the set $\Omega(\mathcal{T}^{0,1})$ of non-wandering solutions in the non-autonomous framework, rather than focusing solely on the set $\fix (\mathcal{T}^{0,1})$ of weak-KAM solutions. This paper aims to study the action of the Lax-Oleinik operator $\mathcal{T}^{0,1}$ on $\Omega(\mathcal{T}^{0,1})$, which reveals to be the set of all global viscosity solutions, and to generalize the representation formula using a generalized Peierls barrier.

\subsection{Main Results}

We set $\mathcal{T} = \mathcal{T}^{0,1}$ and $\mathcal{T}^\tau = \mathcal{T}^{0,\tau}$.\\

In the first part, we study the action of the Lax-Oleinik operator on its non-wandering set, $\Omega(\mathcal{T})$. We begin by noting that $\mathcal{T}$ is non-expanding on $\mathcal{C}(M, \mathbb{R})$, which has several implications for the behavior of viscosity solutions. One key implication is that for a viscosity solutions, being non-wandering is equivalent to recurrent (see Proposition \ref{LONW}). Additionally, we establish that $\mathcal{T}$ acts as an isometry on the non-wandering set $\Omega(\mathcal{T})$, as demonstrated by the following proposition 

\begin{prop} \label{Isometry} 
	\begin{enumerate}
		\item The restriction of $\mathcal{T}$ to the non-wandering set $\Omega(\mathcal{T})$ is an isometric bijection, i.e $\mathcal{T}$ is invertible and for all $v$ and $w$ in $\Omega(\mathcal{T})$, $\Vert \mathcal{T} v - \mathcal{T} w \Vert_\infty = \Vert v -  w \Vert_\infty$.
	
		\item More generally, if we denote by $\Omega_{\tau}(\mathcal{T}) = \mathcal{T}^\tau \big( \Omega(\mathcal{T}) \big)$, then for all times $s<t$, the operator $\mathcal{T}^{s,t} : \Omega_s(\mathcal{T}) \to \Omega_t(\mathcal{T})$ is an isometric bijection. We denote its inverse map by $\mathcal{T}^{t,s}$.
	
		\item For all times $s$, $t$ and $\tau$ in $\mathbb{R}$, $\mathcal{T}^{s,t} = \mathcal{T}^{\tau,t} \circ \mathcal{T}^{s,\tau}$.
	\end{enumerate}
\end{prop}

This property, demonstrated in Subsection \ref{SectionIsometry}, arises from the fact that the Lax-Oleinik operator $\mathcal{T}$ is both surjective and non-expanding on the $\omega$-limit sets of scalar maps, which are compact. It follows from a classical result that any surjective, non-expanding map on a compact set is a bijective isometry. It is a classical result that surjective, non-expanding maps on compact sets are bijective isometries.\\    

From this proposition, we establish that all non-wandering elements of $\Omega(\mathcal{T})$ can be associated with global viscosity solutions of the form $u(t,x) = \mathcal{T}^{0,t}u(x) : \mathbb{R} \times M \to \mathbb{R}$, where the initial data is $u(0,\cdot) = u$. This reveals to be a characterization of non-wandering viscosity solutions. Indeed, if we denote by $\mathcal{V}(\mathcal{T})$ the set of the initial datum of global viscosity solutions i.e. defined for all real times, then we have the following result.

\begin{theo} \label{Glob=NW}
	A viscosity solution of the Hamilton-Jacobi equation \eqref{HJalpha} is global if and only if it is non-wandering. In other words, the following equality holds
	\begin{equation}
		\Omega(\mathcal{T}) = \mathcal{V}(\mathcal{T})
	\end{equation}
\end{theo}
	
The idea of the proof is to first show that non-wandering viscosity solutions correspond to bounded global viscosity solutions. Then, we use the fact —pointed out to us by M. Zavidovique— that... 

\begin{prop} \label{Global=Bdd}
	For a Tonellu Hamiltonian, All the global viscosity solutions $u : \mathbb{R} \times M \to \mathbb{R}$ of the Hamilton-Jacobi equation \eqref{HJalpha} are bounded.
\end{prop}

This tells us that the study of non-wandering viscosity solutions is equivalent to the study of global viscosity solutions.\\

In the second part, we will focus on establishing a generalized representation formula on $\Omega(\mathcal{T})$. To achieve this, we first introduce the concept of domination, which is essential for developing representation formulas.
\begin{defi} \label{DominationDef}
	Let $X$ be a set and let $f: X \times X \to \mathbb{R}$ be a map. A map $\psi : X \to \mathbb{R}$ is said to be \textit{$f$-dominated on $X$} if for all $x$ and $y$ in $X$, we have
	\begin{equation}
		\psi(y) - \psi(x) \leq f(x, y)
	\end{equation}
	We denote by $\dom(X, f)$ the set of $f$-dominated maps on $X$.
\end{defi}

The goal is to identify suitable tame maps $f$ and sets $X$ where domination holds. Additionally, a uniqueness theorem must be valid on the set $X$, which means that if $u$ and $v$ are two elements of $\Omega(\mathcal{T})$ whose restrictions to $X$ are equal, then $u$ and $v$ themselves must be equal.\\

In the representation of weak-KAM solutions by Contreras, Iturriaga, and S\'anchez-Morgado \cite{CIS}, the set $X$ is chosen to be $\mathbb{M}$, which consists of the static classes mentioned in \eqref{IntroRepresentation}, and the map $f$ corresponds to the Peierls barrier $h^\infty$.\\

We will introduce a generalized Peierls barrier $\underline{k} : \mathcal{M}_0 \times M \to \mathbb{R}$, where $\mathcal{M}_0$ refers to the restriction of the Mather set at time $t=0$ (see Definition \ref{MatherDefi} for $\mathcal{M}_0$ and Definition \ref{kDefi} for the barrier $\underline{k}$).\\

Besides, we define 
\begin{itemize}[label= -]
	\item A pseudometric $\underline{d} : \mathcal{M}_0 \times \mathcal{M}_0 \to \mathbb{R}$ by
	\begin{equation}
		\underline{d}(x,y) = \underline{k}(x,y) +  \underline{k}(y,x)
	\end{equation}
	\item An equivalence relation $\sim$ on $\mathcal{M}_0$ by
	\begin{equation}
		x \sim y \Longleftrightarrow \underline{d}(x,y) = 0
	\end{equation}
	\item And the \textit{generalized static classes} as the equivalence classes of the equivalence relation $\sim$. 
\end{itemize}
We denote by $\underline{\mathbb{M}}$ the set of generalized static classes and assume that every element of $\underline{\mathbb{M}}$ is represented by an element of $\mathcal{M}_0$ so that we have the inclusion $\underline{\mathbb{M}} \subset \mathcal{M}_0$. Then, we obtain the following result.

\begin{theo} \label{kRepresentationTheorem}
	We have the following bijection
	\begin{equation} \label{kRepresentationBijection}
		\begin{matrix}
			\Psi_{\underline{k}}: \dom(\underline{\mathbb{M}}, \underline{k}) & \longrightarrow & \Omega(\mathcal{T}) \\
			\psi & \longmapsto & \inf\limits_{y \in \underline{\mathbb{M}}} \{ \psi(y) + \underline{k}(y, \cdot ) \}
		\end{matrix}
	\end{equation}
	with its inverse being the restriction map
	\begin{equation} \label{kInverseMap}
		\begin{matrix}
			\Phi_{\underline{k}} : \Omega(\mathcal{T}) & \longrightarrow & \dom(\underline{\mathbb{M}}, \underline{k}) \\
			v & \longmapsto & v_{| \underline{\mathbb{M}}} 
		\end{matrix}
	\end{equation}     
\end{theo}

Finally, we will explore various applications of this representation formula. It is important to note that these applications may not be direct implications of the theorem but rather adaptations of its proofs and underlying ideas.\\

First, we will consider the autonomous framework and demonstrate Fathi's convergence theorem.

\begin{cor} \label{FathiTh} (Fathi's Convergence Theorem \cite{MR1650261})
	Let $H : T^*M \to \mathbb{R}$ be an autonomous Tonelli Hamiltonian. For any initial data $u \in \mathcal{C}(M,\mathbb{R})$, the viscosity solution $u(t,x) = \mathcal{T}^tu(x)$ converges at $+\infty$ to a weak-KAM solution $v : M \to \mathbb{R}$ of 
	\begin{equation} \label{HJAutonomous}
		H(x,d_xu) = \alpha_0
	\end{equation} 
\end{cor}

Another application involves representing the set $\per_n(\mathcal{T})$ of $n$-periodic viscosity solutions for a fixed integer period $n$, where $n$ does not necessarily have to be the minimal period. In this context, we will define the $n$-Peierls barrier $h^{n\infty}$, as first introduced by A. Fathi and J. N. Mather in \cite{MR1792479} (see Definition \ref{StaticClassesDefi}).

Similarly to the general case, we define

\begin{itemize}[label=-]
	\item A pseudometrice $d_n : \mathcal{M}_0 \times \mathcal{M}_0 \to \mathbb{R}$ by
	\begin{equation}
		d_n(x,y) = h^{n\infty}(x,y) +  h^{n\infty}(y,x)
	\end{equation}
	\item An equivalence relation $\sim_n$ on $\mathcal{M}_0$ by
	\begin{equation}
		x \sim y \Longleftrightarrow d_n(x,y) = 0
	\end{equation} 
	which equivalence classes are called $n$-static classes and are represented by a subset $\mathbb{M}_n$ of the Mather set $\mathcal{M}_0$.
\end{itemize}

In this context, the set $\mathbb{M}_n$ serves as a uniqueness set for $n$-periodic viscosity solutions. Specifically, if two $n$-periodic viscosity solutions $u$ and $v$ in $\per_n(\mathcal{T})$ coincide on $\mathbb{M}_n$, then they coincide everywhere. Furthermore, we establish that every $n$-periodic element of $\per_n(\mathcal{T})$ is $h^{n\infty}$-dominated on $\mathbb{M}_n$. Consequently, we obtain the following result:

\begin{theo} \label{PeriodicRepresentationTheorem}
	We have the following bijection
	\begin{equation} \label{PeriodicRepresentationBijection}
		\begin{matrix}
			\Psi_n: \dom(\mathbb{M}_n, h^{n\infty}) & \longrightarrow & \per_n(\mathcal{T}) \\
			\psi & \longmapsto & \inf\limits_{y \in \mathbb{M}_n} \{ \psi(y) + h^{n\infty}(y, \cdot ) \}
		\end{matrix}
	\end{equation}
	with its inverse being the restriction map
	\begin{equation} \label{PeriodicInverseMap}
		\begin{matrix}
			\Phi_n : \per_n(\mathcal{T}) & \longrightarrow & \dom(\mathbb{M}_n, h^{n\infty}) \\
			v & \longmapsto & v_{| \mathbb{M}_n} 
		\end{matrix}
	\end{equation}     
\end{theo}

Another implication of the representation Theorem \ref{kRepresentationTheorem} is that the dynamics on the non-wandering set $\Omega(\mathcal{T})$ can be controlled by the dynamics of the Lagrangian flow $\phi_L$ (see Section \ref{SubsectionTonelli}) on the Mather set. Based on this, we can establish the following two results:

\begin{cor} \label{PeriodicMatherOmega} 
	If there exists a positive integer $N \geq 1$ such that $\phi^N_{L|\mathcal{M}_0} = Id_{\mathcal{M}_0}$, then $\Omega(\mathcal{T}) = \per_N(\mathcal{T})$.
\end{cor}

\begin{cor} \label{UniformlyRecurrentCor}
	If there exists an increasing sequence of positive integers $\underline{p}=(p_n)_{n \geq 0}$ such that $\phi^{-p_n}_{L|\tilde{\mathcal{M}}_0}$ uniformly converges to the identity, then the elements $v$ of $\Omega(\mathcal{T})$ are $\underline{p}$-recurrent i.e $\lim_n \mathcal{T}^{p_n}v = v$ with a uniform convergence on $v$.
\end{cor}

\subsection{Structure of the Article}

In Section \ref{SectionTools}, we introduce the preliminary tools from weak-KAM theory and Aubry-Mather theory. To ensure thoroughness, we provide proofs for all key results, though we omit extensive details, thereby offering an introduction to weak-KAM theory in the non-autonomous setting. Section \ref{SectionLONW} examines the action of the Lax-Oleinik operator $\mathcal{T}$ on its non-wandering set $\Omega(\mathcal{T})$ and presents a proof of the characterization of non-wandering viscosity solutions stated in Theorem \ref{Glob=NW}. Section \ref{SectionUniqueness} provides a proof of the uniqueness theorem for non-wandering viscosity solutions on the Mather set $\mathcal{M}_0$, following \cite{MR2041603}. Here, the calibration on the Mather set is demonstrated using a classical approach, inspired by A. Fathi’s method in the autonomous case. In Section \ref{SectionRepresentation}, we define generalized Peierls barriers, discuss their properties, and then present the proof of the main result, Theorem \ref{kRepresentationTheorem}. Finally, Section \ref{SectionApplication} is dedicated to exploring various examples and proving the corollaries stated in the introduction.

\section{Preliminary Tools and Properties} \label{SectionTools}

This section introduces the primary tools utilized in the construction process, drawn from the Weak-KAM and the Aubry-Mather theories. To ensure comprehensiveness, we will demonstrate the majority of the properties while omitting lengthy and intricate proofs.\\

These properties were initially established in the autonomous setting. And some were subsequently studied in the discrete case, which encompasses the non-autonomous setting (See \cite{MR2387248}, \cite{MR2753949}, \cite{MR2874895}. A more specific exposition of the non-autonomous weak-KAM theory has been conducted in \cite{MR2393423} within the the pseudographs point of view. In this section, we provide a more classical presentation, similar to what has been written by A.Fathi in \cite{fathi2008weak}.

\subsection{Tonelli Hamiltonians and the Lax-Oleinik Operator} \label{SubsectionTonelli}

We define the notion of Tonelli Hamiltonians, initially introduced in \cite{MR1109661}. These are Hamiltonians that are convex and superlinear in the fibres, serving as the tame framework for the weak-KAM theory.

\begin{defi} \label{TonelliDef}
	A $1$-time-periodic Hamiltonian $H(t,x,p) : \mathbb{T}^1 \times T^*M \to \mathbb{R}$ is said \textit{Tonelli} if it satisfies the following classical hypotheses :
	\begin{itemize}
		\item Regularity: $H$ is $\mathcal{C}^2$.
		\item Strict convexity: $\partial_{pp}H(t,x,p)>0$ for all $(t,x,p) \in \mathbb{T}^1 \times T^*M$.
		\item Superlinearity: $H(t,x,p)/|p| \to \infty$ as $|p| \to \infty$ for each $(t,x) \in \mathbb{T}^1 \times M$.
		\item Completeness: The Hamiltonian vector field $X_H(t,x,p) = (\partial_pH(t,x,p),-\partial_xH(t,x,p))$ and hence its flow $\phi_H^{t,s}$ is complete in the sense that the flow curves are defined for all times $t \in \mathbb{R}$.
	\end{itemize}
\end{defi}
 
Under these assumptions, one can associate to $H(t,x,p)$ a time-periodic Tonelli \textit{Lagrangian} $L(t,x,v): \mathbb{T}^1 \times TM \to \mathbb{R}$ given by
\begin{equation} \label{Lag}
	L(t,x,v) = \max_{p \in T_x^*M} \{p(v) - H(t,x,p)\}
\end{equation}
which symmetrically gives 
\begin{equation}\label{Ham} 
	H(t,x,p) = \max_{v \in T_xM} \{p(v) - L(t,x,v)\}
\end{equation}

\par The \textit{Euler-Lagrange flow} also named \textit{Lagrangian flow} $\phi_L^{s,t}$ is the conjugate to the Hamiltonian flow $\phi_H^{s,t}$ by the \textit{Legendre map}\footnote{The Tonelli assumptions imply that for all time $t \in \mathbb{R}$, the Legendre map $\mathcal{L}$ is a diffeomorphism between $\{t\} \times TM$ and $\{t\} \times T^*M$ (see \cite{fathi2008weak} for details).} $\mathcal{L}(t,x,v) = \big(t,x,\partial_vL(t,x,v)\big)$.  We adopt the notation $\phi_L^t$ for $\phi_L^{0,t}$.\\

If $0\leq s \leq t$ are two real times and $x$ and $y$ are two points of $M$, we define the following quantities
\begin{itemize}
	\item For all absolutely continuous curve $\gamma :[s,t] \to M$, the \textit{action} of $\gamma$ is
	\begin{equation}
		A_L(\gamma) = \int_s^t L(\tau,\gamma(\tau), \dot{\gamma}(\tau)) \; d\tau
	\end{equation}
	\item the \textit{potential} between $(s,x)$ and $(t,y)$ is
	\begin{equation} \label{Potential0}
		h_0^{s,t} (x,y) = \inf \left\{ A_L(\gamma) \; \left| \;
		\begin{matrix}
			\gamma : & [s,t] \to M \\
			& s \mapsto x \\
			& t \mapsto y
		\end{matrix} \right.	\right\}
	\end{equation}
	where the infimum is taken over such absolutely continuous curves $\gamma$.
\end{itemize}

\begin{rem}
	The demanded absolute continuity of the curves enables to define this integral. All the curves and all the infimum over the curves that may be considered in this paper will be in the family of absolutely continuous curves. We will refrain from recalling it every time.
\end{rem}

We now introduce the Lax-Oleinik operator mentioned in the introduction and used to generate the viscosity solutions of the Hamilton-Jacobi equation (\ref{HJalpha}).

\begin{defi} \label{LODef}
	Fix two times $s<t$.
	\begin{enumerate}
		\item The \textit{Lax-Oleinik operator} $\mathcal{T}_0^{s,t}: \mathcal{C}(M,\mathbb{R}) \to \mathcal{C}(M,\mathbb{R})$ is defined as
		\begin{equation} \label{LO}
			\begin{split}
				\mathcal{T}_0^{s,t} u_0(x) & 
				= \inf_{
				\begin{matrix}
					\gamma : [s,t] \rightarrow M \\
					t \mapsto x 
				\end{matrix}
				} \left\{  u_0(\gamma(s)) + \int_s^t L(\tau,\gamma(\tau), \dot{\gamma}(\tau)) \;d\tau \right\} \\
				& = \inf_{y \in M} \left\{  u_0(y) + h_0^{s,t}(y,x) \right\}\\
			\end{split}
		\end{equation}
		We adopt the notations $\mathcal{T}_0^t$ for $\mathcal{T}_0^{0,t}$ and $\mathcal{T}_0$ for $\mathcal{T}_0^{0,1}$.
		\item The \textit{Mañé critical value} $\alpha_0$ is defined as 
		\begin{equation}\label{ManeCritValue}
			\alpha_0 = - \inf_\mu \left\{\int_{\mathbb{T}^1 \times TM}L \;d\mu \right\}
		\end{equation}
		where the infimum is taken over compact supported Borel probability measures $\mu$ invariant by the Euler-Lagrangian flow corresponding to $L$.\\
		\item The \textit{full Lax-Oleinik operator} $\mathcal{T}^{s,t}: \mathcal{C}(M,\mathbb{R}) \to \mathcal{C}(M,\mathbb{R})$ is defined as
		\begin{equation}
			\mathcal{T}^{s,t} u_0(x) = \mathcal{T}^{s,t} u_0(x) + \alpha_0.(t-s)
		\end{equation}
		We adopt the notations $\mathcal{T}^t$ for $\mathcal{T}^{0,t}$ and $\mathcal{T}$ for $\mathcal{T}^{0,1}$.
	\end{enumerate}
\end{defi}

One can verify that for all $s<t<\tau$, $\mathcal{T}^{t,\tau} \circ \mathcal{T}^{s,t} = \mathcal{T}^{s, \tau} $. Additionally, Since $L$ is time-periodic, $\mathcal{T}^{t+1} = \mathcal{T}^t \circ \mathcal{T}$. Hence $\{\mathcal{T}^n \}_{n\in \mathbb{N}}$ is a discrete semi-group acting on $\mathcal{C}(M,\mathbb{R})$, called the Lax-Oleinik semi-group. These properties are also verified by $\mathcal{T}_0$. The main focus of this paper is the asymptotic behaviour of $\mathcal{T}$.\\

\begin{defi}
	\begin{enumerate}
		\item To all scalar map $u \in \mathcal{C}(M,\mathbb{R})$ and all time $s \in \mathbb{R}$, we associate the \textit{viscosity solution of the Hamilton-Jacobi equation \eqref{HJalpha}} $u : [s,+\infty) \times M \to \mathbb{R}$ defined by $u(t,x) = \mathcal{T}^{s,t}u(x)$.
		\item If $\mathcal{T}u = u$, then $u$ is called a \textit{weak-KAM solution} of the Hamilton-Jacobi equation (\ref{HJalpha}). In other words, weak-KAM solutions are the initial data of one-time periodic viscosity solutions.
		\item We denote by 
		\begin{enumerate}
			\item $\fix (\mathcal{T})$ the set of fixed points of the operator $\mathcal{T}$, namely the set of weak-KAM solutions.
			\item $\per (\mathcal{T})$ the set of periodic points of the operator $\mathcal{T}$.
			\item For all $n \geq 1$, $\per_n (\mathcal{T})$ the set of $n$-periodic points of the operator $\mathcal{T}$. The integer $n$ need not to be the minimal period of elements of $\per_n (\mathcal{T})$.
		\end{enumerate}
		
	\end{enumerate}  
\end{defi}

\begin{rem} \label{RemarqueIntro}
	\begin{enumerate}
		\item By abuse of language, we will often refer to the initial condition $u(x)$ as a viscosity solution while referring in reality to the corresponding viscosity solution $u(t,x) = \mathcal{T}^{t}u(x)$.
		
		\item The Lax-Oleinik operator $\mathcal{T}_0^{s,t}$ is the generator of the viscosity solutions of the Hamilton-Jacobi equation
		\begin{equation} \label{HJ}
			\partial_t u + H(t,x,d_xu) =0
		\end{equation} 
		It is easy to verify that a map $u \in \mathcal{C}(\mathbb{R} \times M, \mathbb{R})$ is a viscosity solution of (\ref{HJalpha}) if and only if $u+ \alpha_0 t$ is a viscosity solution of (\ref{HJ}).  
		\item Replacing $H$ by $H- \alpha_0$ and $L$ by $L+ \alpha_0$, it is always possible to assume that the critical \Mane value $\alpha_0$ is null. In this case, $\mathcal{T} = \mathcal{T}_0$.
	\end{enumerate}
\end{rem}

\subsection{Existence of Minimizing Curves}

We give classical results on minimizing curves that stem directly from the theory of variational calculus. For certain standard statements, we will avoid providing tedious proofs that can be found in the autonomous framework in \cite{fathi2008weak} and \cite{MR1720372}.\\

\begin{defi} \label{MinimizingDef}
	A \textit{minimizing curve} $\gamma : I \to M$ defined on an interval $I$ is a curve such that for all $s<t$ in $I$,
	\begin{equation} \label{MinimizingFormula}
		\int_s^t L(\tau,\gamma(\tau), \dot{\gamma}(\tau)) \; d\tau = h_0^{s,t}( \gamma(s), \gamma(t))
	\end{equation}
\end{defi}

\begin{prop}\label{MinimizingProp}
	\begin{enumerate}
		\item If for some times $s<t$, the curve $\gamma$ verifies (\ref{MinimizingFormula}), then it is minimizing on $[s,t]$.
		\item A minimizing curve $\gamma$ is as regular as the Lagrangian $L$ and it follows the Lagrangian flow $\phi_L$ i.e for all time $\tau \in [s,t]$, $(\gamma(\tau), \dot{\gamma}(\tau)) = \phi_L^{s,\tau}(\gamma(0), \dot{\gamma}(0))$. These curves verify the Euler-Lagrange equation
		\begin{equation}
			\partial_x L(\tau,\gamma(\tau), \dot{\gamma}(\tau)) = \frac{d}{d\tau} \big(\partial_v L(\tau,\gamma(\tau), \dot{\gamma}(\tau))\big)
		\end{equation}
	\end{enumerate}
\end{prop}

\begin{rem} \label{MinimizingRem}
	The first property is due to the fact that the quantity
	\begin{equation*}
		\int_{s'}^{t'} L(\tau,\gamma(\tau), \dot{\gamma}(\tau)) - h_0^{s',t'}( \gamma(s'), \gamma(t'))  \; d\tau \geq 0
	\end{equation*}
	is non-negative and is increasing with the size of the interval $[s',t']$.
\end{rem}

A main brick in the study of Tonelli Lagrangians is the existence of minimizing curves.

\begin{theo}(Tonelli's Theorem) \label{TonelliTheorem}
	Let $L : \mathbb{T}^1 \times TM \to \mathbb{R}$ be a Tonelli Lagrangian. Let $s<t$ be two real times and $x$ and $y$ be two points of $M$. Then, there exists a minimizing curve $\gamma : [s,t] \to M$ linking $x$ to $y$.
\end{theo}

A first application of this theorem to the Lax-Oleinik operator $\mathcal{T}$ gives the following 
\begin{cor} \label{TonelliLaxOleinik}
	Let $u$ be a scalar map in $\mathcal{C}(M,\mathbb{R})$. For all times $s<t$ and for all point $x$ of $M$, there exists a minimizing curve $\gamma : [s,t] \to M$ such that 
	\begin{equation}
		\gamma(t) = x \quad \text{and} \quad \mathcal{T}_0^{s,t}u(x) = u(\gamma(0)) + A_L(\gamma)
	\end{equation}
\end{cor}

\begin{proof}
	We use the potential $h^{s,t}_0(\cdot,x)$ introduced in (\ref{Potential0}). Recall that
	\begin{equation*}
		\mathcal{T}_0^{s,t}u(x) = \inf_{y \in M} \{u(y) + h_0^{s,t}(y,x) \}
	\end{equation*}
	We will see in Proposition \ref{Regularity} that the map $y \mapsto h_0^{s,t}(y,x)$ is continuous, and so is $y \mapsto u(y)+ h_0^{s,t}(y,x)$ on the compact manifold $M$. Thus, it admits a minimizing point $y$. We get
	\begin{equation*}
		\mathcal{T}_0^{s,t}u(x) = u(y) + h_0^{s,t}(y,x)
	\end{equation*}
	An application of Tonelli's theorem to $h_0^{s,t}(y,x)$ completes the proof.
\end{proof}

\subsection{A Priori Compactness and Regularity}

Another fundamental theorem is the A priori compactness property of minimizing curves for Tonelli Lagrangians. It has been first proven by John N. Mather in \cite{MR1109661}.

\begin{theo}(A Priori Compactness) \label{APrioriCompactness}
	Let $L : \mathbb{T}^1 \times TM \to \mathbb{R}$ be a Tonelli Lagrangian and fix a small positive $\varepsilon >0$. Then, there exists a compact subset $K_\varepsilon$ of $TM$ such that every minimizing curve $\gamma : [s,t] \to M$ with $t-s \geq \varepsilon$ verifies $(\gamma(\tau) , \dot{\gamma}(\tau)) \in K_\varepsilon$.
\end{theo}

\begin{cor} \label{MinimizingC1Compactness}
	For fixed times $s<t$, if $\gamma_n : [s,t] \to M$ is a sequence of minimizing curves, then it admits a subsequence that $C^1$-converges to a minimizing curve $\gamma:[s,t] \to M$.
\end{cor}  

\begin{proof}
	By A priori compactness, we can extract a subsequence $(\gamma_{k_n}(s), \dot{\gamma}_{k_n}(s))$ that converges to $(x,v) \in TM$. Since the curves $\gamma_n$ are minimizing, we have for all $\tau \in [s,t]$, $(\gamma_n(\tau), \dot{\gamma}_n(\tau)) = \phi_L^{s,\tau}(\gamma_n(s), \dot{\gamma}_n(s))$. Set $\gamma(\tau) = \pi \circ \phi_L^{s,\tau}(x,v)$ so that $(\gamma(\tau), \dot{\gamma}(\tau)) = \phi_L^{s,\tau}(x,v)$. By continuity of the Lagrangian flow $\phi_L$, we deduce the $C^1$-convergence of the curves $\gamma_{k_n}$ to $\gamma$ on the time interval $[s,t]$. By continuity of $h_0$ obtained from proposition \ref{Regularity}, we obtain
	\begin{align*}
		h_0^{s,t}(\gamma(s),\gamma(t)) = \lim_n h_0^{s,t}(\gamma_{k_n}(s),\gamma_{k_n}(t)) = \lim_n \int_s^t L(\tau, \gamma_{k_n}(\tau), \dot{\gamma}_{k_n}(\tau)) \; d\tau = \int_s^t L(\tau, \gamma(\tau), \dot{\gamma}(\tau)) \; d\tau
	\end{align*}
	where the last limit is due to the $C^1$-convergence. We conclude the the curve $\gamma$ is minimizing.
\end{proof}

A consequence of the A Priori Compactness is the regularization property of the Lax-Oleinik operator $\mathcal{T}$. 

\begin{prop} \label{Regularity}
	For all positive $\varepsilon >0$, there exists a positive constant $\kappa_\varepsilon>0$ such that for all times $s < t$ with $t-s \geq \varepsilon$, we have
	\begin{enumerate}
		\item The potential $h_0^{s,t} : M \times M \to \mathbb{R}$ is $\kappa_\varepsilon$-Lipschitz. Moreover, we can take  $\kappa_\varepsilon$ so that the time dependent potential $h_0 : \{0 \leq s \leq t- \varepsilon\}\times M \times M \to \mathbb{R}$ is still $\kappa_\varepsilon$-Lipschitz. 
		\item For all initial data $u \in \mathcal{C}(M,\mathbb{R})$, the maps $\mathcal{T}_0^{s,t}u$  and $\mathcal{T}^{s,t}u : M \to \mathbb{R}$ is $\kappa_\varepsilon$-Lipschitz on the set $\{0 \leq s \leq t- \varepsilon\}\times M$.
	\end{enumerate}
\end{prop}

Consequently, we get a regularity result on viscosity solutions.
 
\begin{cor} \label{Equicontinuity}
	For all viscosity solution $u: [s, + \infty) \times M \to \mathbb{R}$ and $v : \mathbb{R} \times M \to \mathbb{R}$ of the Hamilton-Jacobi equation (\ref{HJalpha}), the families $(u(t,\cdot))_{t \geq s+1}$ and $(v(t,\cdot))_{t \in \mathbb{R}}$ are $\kappa_1$-equilipschitz.
\end{cor}

\subsection{Calibrated Curves}

Calibrated curves represent a type of minimizing curves that are well adapted to a given viscosity solution in the following sense.

\begin{defi} \label{CalibDefi}
	Let $u(t,x)$ be a viscosity solution of (\ref{HJalpha}). A curve $\gamma : I \subset \mathbb{R} \to M$ defined on a real interval $I$ is said \textit{calibrated by $u$} or \textit{$u$-calibrated} if for all times $s<t$ of $I$, we have
	\begin{equation} \label{CalibrationEquation}
		\begin{split}
			u(t, \gamma(t)) &= u(s, \gamma(s)) + \int_s^t L(\tau, \gamma(\tau), \dot{\gamma}(\tau)) \; d\tau + \alpha_0.(t-s) \\
			&= u(s, \gamma(s)) + h_0^{s,t}(\gamma(s), \gamma(t)) + \alpha_0.(t-s) \\
			&= u(s, \gamma(s)) + h^{s,t}(\gamma(s), \gamma(t)) 
		\end{split}
	\end{equation}
	where the potential $h$ will be defined in Subsection \ref{SectionPotential}.
\end{defi}

\begin{rem}\label{CalibEL}
	\begin{enumerate}
		\item One observes from (\ref{CalibrationEquation}) that calibrated curves $\gamma$ realize the infimum in the definition (\ref{LO}) of the Lax-Oleinik operator. This means that all calibrated curves are minimizing and do follow the Lagrangian flow $\phi_L$.
		\item Same as for minimizing curves in Remark \ref{MinimizingRem}, if a curve $\gamma$ verifies (\ref{CalibrationEquation}) for some times $s<t$, then it verifies it for all $s \leq s'<t' \leq t$ and $\gamma$ is calibrated by $u$ on the interval $[s,t]$.
		\item For viscosity solutions of the translated Hamilton-Jacobi equation (\ref{HJ}), the good equation of calibration is the following
		\begin{equation}
			\begin{split}
				u(t, \gamma(t)) &= u(s, \gamma(s)) + \int_s^t L(\tau, \gamma(\tau), \dot{\gamma}(\tau)) \; d\tau \\
				&= u(s, \gamma(s)) + h_0^{s,t}(\gamma(s), \gamma(t))
			\end{split}
		\end{equation}
	\end{enumerate}
\end{rem}

\begin{prop} \label{CalibExist}
	Let $u(t,x): \mathbb{R} \times M \to \mathbb{R}$ be a viscosity solution of (\ref{HJalpha}). For all points $x$ of $M$ and for all times $t \in \mathbb{R}$, $u$ admits a calibrated curve $\gamma_x : (- \infty , t] \to M$ with $\gamma(t) = x$.
\end{prop}

\begin{proof}
	Fix a point $x$ in $M$ and a time $t$ in $\mathbb{R}$. We know that for all $s<t$, $u(t, \cdot) = \mathcal{T}^{s,t}u(s,\cdot)$. Then, applying Corollary \ref{TonelliLaxOleinik}, we obtain curves $\gamma_n : [-n,t] \to M$ with $\gamma_n(t) = x$, for large $n$, such that 
	\begin{equation*} 
		u(t, \gamma_n(t)) = u(s, \gamma_n(-n)) + \int_{-n}^t \Big( L(\tau, \gamma_n(\tau), \dot{\gamma_n}(\tau)) + \alpha_0 \Big) \; d\tau 
	\end{equation*}
The Remark \ref{CalibEL} points out that these curves $\gamma_n$ are $u$-calibrated .\\
	The A priori compactness Theorem \ref{APrioriCompactness} says that the curves $(\gamma_n, \dot{\gamma}_n)$ have their images in a same compact set $K$ of $TM$. Thus, for all compact interval $C$ of $(-\infty, t]$ and for large enough integer $n_0$, the family $(\gamma_{n|C})_{n \geq n_0}$ is relatively compact in the $C^1$-topology. Therefore, a diagonal arguments provides us with a curve $\gamma : (-\infty, t] \to M$ with $\gamma(t) =x$ such that the sequence $\gamma_n$ $C^1$-converges, up to extraction, to the curve $\gamma$ on all compact subsets of $(-\infty, t]$.\\
	Let $s \leq t$ be a real time. For $n$ large enough, we have from the calibration of $\gamma_n$ that
	\begin{equation*} 
		u(t, \gamma_n(t)) = u(s, \gamma_n(s)) + \int_s^t \Big( L(\tau, \gamma_n(\tau), \dot{\gamma_n}(\tau))+\alpha_0 \Big) \; d\tau
	\end{equation*}
	And taking the limit on $n$, we get the calibration equation for $\gamma$.
	\begin{equation*}
		u(t, \gamma(t)) = u(s, \gamma(s)) + \int_s^t \Big( L(\tau, \gamma(\tau), \dot{\gamma}(\tau)) +\alpha_0 \Big) \;d\tau
	\end{equation*}
\end{proof}

One of the powerful results of the weak-KAM theory is the theorem of regularity on calibrated curves proved by A.Fathi in the autonomous case (see \cite{fathi2008weak}). For the sake of completeness, we provide a proof in the non-autonomous case following \cite{MR3674224}. But first, we introduce a tool derived from convex analysis.

\begin{prop}(Fenchel's inequality)
	For all $x$ in $M$ and all $(v,p) \in T_xM \times T^*_xM$
	\begin{equation} \label{Fenchel}
		p(v) \leq H(t,x,p) + L(t,x,v)
	\end{equation}
	with equality if and only if $p = \partial_v L (x,v)$ if and only if $v= \partial_p H(t,x,p)$.
\end{prop}
 
\begin{rem}
	Note that if $\phi^t_H(x)=(x(t),p(t)) $ is a curve that follows the Hamiltonian flow, then the Hamiltonian equations results in the equalities
	\begin{equation}
		\dot{x}(t) = \partial_p H\big(t,x(t),p(t)\big) \quad \text{and} \quad p(t) = \partial_v L\big(t,x(t),\dot{x}(t) \big)
	\end{equation}
\end{rem}

\begin{theo} \label{CalibRegularity}
	Let $u$ be a viscosity solution of (\ref{HJ}). If the curve $\gamma : I \to M$ is calibrated by $u$, then for all time $t$ in the interior of $I$, $u$ is differentiable at $(t, \gamma(t))$ with differential
	\begin{equation}
		\partial_tu(t,\gamma(t))= \alpha_0-H(t,\gamma(t),d_xu(t,\gamma(t))) \quad \text{and} \quad  d_xu(t,\gamma(t)) = \partial_vL \big(t,\gamma(t), \dot{\gamma}(t) \big)
	\end{equation}
\end{theo}

\begin{proof}
	\textit{Differentiability.} Let $\gamma : I \to M$ be a curve calibrated by $u$ and fix $(t,x) = (t,\gamma(t))\in I \times M$. We will bound $u$ in a neighbourhood of $(t,x)$ by two $C^1$ maps that coincide with it at this point. Consider a chart $B_0 \subset M$ around $x$ and let $(s,y) \in \mathbb{R} \times B_0$ be close enough to $(t,x)$. Fix two reference times $t^+<t<t^-$ in $I$ and $x^\pm = \gamma(t^\pm)$. From the calibration
	\begin{equation} \label{Calibdiff}
		u(t,x) = u(t^\pm,x^\pm) + \int_{t^\pm}^t \Big( L\big(\tau,\gamma(\tau),\dot{\gamma}(\tau) \big) + \alpha_0 \Big) \; d\tau =: \psi^\pm(t,x)
	\end{equation}
	and from the definition of viscosity solutions, we have
	\begin{equation}
		u(s,y) \leq u(t^+,x^+) + \int_{t^+}^s \Big( L\big(\tau,\gamma^+_{(s,y)}(\tau),\dot{\gamma}^+_{(s,y)}(\tau) \big) + \alpha_0 \Big) \; d\tau =: \psi^+(s,y)
	\end{equation}
	and 
	\begin{equation}
		u(s,y) \geq u(t^-,x^-) - \int_s^{t^-} \Big( L\big(\tau,\gamma^-_{(s,y)}(\tau),\dot{\gamma}^-_{(s,y)}(\tau) \big) + \alpha_0 \Big) \; d\tau =: \psi^-(s,y)
	\end{equation}
	where, in the chart $\mathbb{R} \times B_0$,
	\begin{equation}
		\gamma^\pm_{(s,y)}(\tau) = \gamma(\tau) + \frac{\tau - t^\pm}{s - t^\pm} (y-\gamma(s))
	\end{equation}		
	are smooth families of curves linking $(t^\pm,x^\pm)$ to $(s,y)$ and such that $\gamma^\pm_{(t,x)}= \gamma$.\\
	It is easy to see that $\psi^\pm$ are $C^1$. Moreover, $\psi^- \leq u \leq \psi^+$ with equalities at $(t,x)$. Then $u$ is differentiable at $(t,x)$.\\
	
	\textit{Evaluation of the Differential.} We differentiate (\ref{Calibdiff}) with respect to time $t$ without forgetting that $x = \gamma(t)$, to get
	\begin{equation}
		\partial_tu(t,\gamma(t)) + d_x u(t,\gamma(t)) = L\big(t,\gamma(t), \dot{\gamma}(t) \big) + \alpha_0
	\end{equation}
	And by Fenchel inequality (\ref{Fenchel}) for $x = \gamma(t)$, $v = \dot{\gamma}(t)$ and $p = d_x u(t,\gamma(t))$, we have
	\begin{equation} \label{Calibdiff2}
		0 = \partial_tu(t,\gamma(t)) + d_x u(t,\gamma(t)). \dot{\gamma}(t) - L\big(t,\gamma(t), \dot{\gamma}(t) \big) - \alpha_0 \leq \partial_tu(t,x) + H(t,x,d_x u(t,x)) - \alpha_0
	\end{equation}
	We show the inverse inequality $\partial_tu(t,x) + H(t,x,d_x u(t,x)) - \alpha_0 \leq 0$. Let $v$ be any element of $T_xM$ and let $\sigma : [t,t+1] : \to M$ be a curve such that $\sigma(0) = x$ and $\dot{\sigma}(0) = v$. Since $u$ is a viscosity solution, we have that for all $s \in [t,t+1]$,
	\begin{align*}
		\frac{u(s,\sigma(s)) - u(t,x)}{s-t} \leq \frac{1}{s-t}\int_t^s \Big( L(\tau, \sigma(\tau), \dot{\sigma}(\tau) + \alpha_0 \Big) \; d\tau
	\end{align*}
	Letting $s$ tend to $t$, we deduce that
	\begin{align*}
		\partial_tu(t,x) + d_x u(t,\sigma(t)). v \leq L(t,x,v) + \alpha_0
	\end{align*}
	Taking the supremum on $v \in T_xM$, we infer from the relation between $L$ and $H$ expressed in (\ref{Ham}) the desired inequality $\partial_tu(t,x) + H(t,x,d_x u(t,x)) - \alpha_0\leq 0$. Therefore, there is equality everywhere in (\ref{Calibdiff2}), and in particular in the Fenchel inequality. Hence 
	\begin{equation*}
		\partial_tu(t,x) + H(t,x,d_x u(t,x)) = \alpha_0 \quad \text{and} \quad d_x u(t,\gamma(t)) = \partial_v L \big(t,\gamma(t), \dot{\gamma}(t) \big)
	\end{equation*}
\end{proof}

\subsection{The Mather Set} \label{MatherSection}

In this section, we present a concept derived from the Aubry-Mather theory. The Mather set, introduced by John N. Mather, is associated with the study of action-minimizing measures in Lagrangian and Hamiltonian systems. Some particular subsets specific to the non-autonomous case are defined as follows.

\begin{defi} \label{MatherDefi}
	\begin{enumerate}
		\item A measure $\mu$ on $\mathbb{T}^1 \times TM$ is \textit{a minimizing measure} if it is a Borel probability measure, invariant by the Euler-Lagrange flow $\phi_L$ and it satisfies
		\begin{equation}
			\int_{\mathbb{T}^1 \times TM} L \; d\mu = - \alpha_0
		\end{equation} 
		where $\alpha_0$ is the Mañé critical value defined in (\ref{ManeCritValue}) (and assumed to be null).
		\item The \textit{Mather set} $\tilde{\mathcal{M}}$ is defined by
		\begin{equation} \label{MatherDefFormula}
			\tilde{\mathcal{M}} = \bigcup_\mu \supp(\mu) \subset \mathbb{T}^1 \times TM
		\end{equation} 
		where the union is on minimizing measures $\mu$.
		\item The \textit{projected Mather set} $\mathcal{M}$ is the projection of $\tilde{\mathcal{M}}$ to $\mathbb{T}^1 \times M$.
		\item The \textit{time-zero Mather set} $\tilde{\mathcal{M}}_0$ and its projected counterpart $\mathcal{M}_0$ are the intersections
		\begin{equation}
			\tilde{\mathcal{M}}_0 := \tilde{\mathcal{M}} \cap \big( \{0\} \times TM \big) \quad \text{and} \quad \mathcal{M}_0 := \mathcal{M} \cap \big( \{0\} \times M \big)
		\end{equation}
		seen respectively as subsets of $TM$ and $M$.
		\item We define the \textit{recurrent Mather set} $\tilde{\mathcal{M}}_0^R \subset TM$ by 
		\begin{equation}
			\tilde{\mathcal{M}}_0^R = \{\tilde{x} \in \tilde{\mathcal{M}}_0 \; | \; \tilde{x} \text{ is recurrent under the map } \phi_L^{-1} \}
		\end{equation}
		and we denote its projection on $M$ by $\mathcal{M}_0^R$.
	\end{enumerate}
\end{defi}

\begin{rem} \label{MatherInv}
	\begin{enumerate}
		\item More explicitly, the invariant measures $\mu$ featured in the definition are invariant by the maps $\Phi_L^\tau$, for all time $\tau>0$, given by
		\begin{equation} \label{InvarianceFlow}
			\begin{split}
				\Phi_L^\tau : \mathbb{T}^1 \times TM &\longrightarrow \mathbb{T}^1 \times TM \\
				(t,x,v) & \longmapsto (t + \tau, \phi_L^{t,t+\tau}(x,v))
			\end{split}
		\end{equation}   
		\item We easily see from these definitions that the different Mather sets are invariant by either the Euler-Lagrange flow or its time-one map. 
	\end{enumerate}
\end{rem}

The next proposition has been proved by John N.Mather in the non-autonomous case. See proposition 4 of \cite{MR1109661}
\begin{prop} \label{MatherNonempty}
	The Mather set $\tilde{\mathcal{M}}$ is compact and non-empty and the recurrent Mather set $\tilde{\mathcal{M}}_0^R$ is dense in $\tilde{\mathcal{M}}_0$
\end{prop}

\subsection{The Potential $h$}  \label{SectionPotential}

For all times $s<t$ we define the potential $h^{s,t}: M \times M \to \mathbb{R}$ by
\begin{equation} \label{Potential}
	\begin{split}
		h^{s,t}(x,y)& = (t-s).\alpha_0 + h_0^{s,t}(x,y)\\  
		&= (t-s).\alpha_0 + \inf \left\{ A_L(\gamma) \; \left| \;
			\begin{matrix}
				\gamma : & [s,t] \to M \\
				& s \mapsto x \\
				& t \mapsto y
			\end{matrix} \right.	\right\}
	\end{split}
\end{equation}
And we adopt the notation $h^t$ for $h^{0,t}$.

\begin{rem} \label{PeierlsRemark}
	\begin{enumerate}
		\item The Tonelli Theorem \ref{TonelliTheorem} states that the infimum in the definition above is always achieved by a minimizing curve.
		\item When the \Mane critical value $\alpha_0$ is null, we get $h^{s,t} = h_0^{s,t}$.
		\item We deduce from Proposition \ref{Regularity} the Lipschitz regularity of the potential $h^{s,t}$.
	\end{enumerate}
\end{rem}

\begin{prop} \label{hprop}
	\begin{enumerate}
		\item (Triangular Inequality) For all real time $s<\tau<t$, and for all points $x$, $y$ and $z$ in $M$, we have the triangular inequality
			\begin{equation}\label{TriangIneg}
				h^{s,t}(x,z) \leq h^{s,\tau}(x,y)+h^{\tau,t}(y,z)
			\end{equation}
		\item \label{hvisc} For all $x \in M$, $h^t(x,\cdot)$ is a viscosity solution of the Hamilton-Jacobi equation (\ref{HJalpha}) i.e. for all times $0<s<t$, $\mathcal{T}^{s,t} h^s(x,\cdot) = h^t(x,\cdot)$.\\
	\end{enumerate}
\end{prop}

\begin{proof}
	\textit{1.} Let $\gamma_1: [s,\tau] \to M$ be a curve linking $x$ to $y$ and $\gamma_2: [\tau,t] \to M$ be a curve linking $y$ to $z$. And let $\gamma : [s,t] \to M$ be their concatenated curve linking $x$ to $z$. Then, we have
	\begin{align*}
		h^{s,t}(x,y) \leq (t-s).\alpha_0 + A_L(\gamma) = (\tau-s).\alpha_0 + A_L(\gamma_1) + (t-\tau).\alpha_0 + A_L(\gamma_2)
	\end{align*}
	and taking the infinimum on the curves $\gamma_1$ and $\gamma_2$, we get the wanted inequality.\\
	
	\textit{2.} For $s<t \in \mathbb{R}$ and $x,y \in M$,
	\begin{align*}
		\mathcal{T}^{s,t}h^s(x,y) & = \inf_{
	\begin{matrix}
		\gamma_1 : [s,t] \rightarrow M \\
		t \mapsto y 
	\end{matrix}
	} \left\{  h^s\big(x,\gamma_1(s)\big) + A_L(\gamma_1) + (t-s).\alpha_0 \right\} \\
	& = \inf_{
	\begin{matrix}
		\gamma_1 : [s,t] \rightarrow M \\
		t \mapsto y 
	\end{matrix}
	}   \inf_{
	\begin{matrix}
		\gamma_2 : [0,s] \rightarrow M \\
		0 \mapsto x \\
		s \mapsto \gamma_1(s)
	\end{matrix}
	} \left\{  s.\alpha_0 +  A_L(\gamma_2) +  A_L(\gamma_1) + (t-s).\alpha_0 \right\} \\
		& = \inf_{
	\begin{matrix}
		\gamma : [0,t] \rightarrow M \\
		0 \mapsto x\\
		t \mapsto y 
	\end{matrix}
	} \left\{ t.\alpha_0+  A_L(\gamma) \right\} = h^t(x,y)
	\end{align*}
\end{proof}

We will show that this potential $h$ is bounded for $t-s > \varepsilon$ for some fixed $\varepsilon$. To achieve this, we need to consider its lower and upper bounds.

\begin{defi}
	We define the maps $m$ and $M : M \times M \to \mathbb{R}$ as 
	\begin{equation} \label{PotentialMinMaxDef} 
		m(x,y) = \inf_{n \geq 1} h^n(x,y) \quad \text{and} \quad M(x,y) = \sup_{n\geq 1} h^{n}(x,y)
	\end{equation}

	The associated time-dependant maps are defined as 
	\begin{equation}
		m^{t}(x,y)= m(t,x,y) = \inf_{n \geq 1} h^{n+t}(x,y) \quad \text{and} \quad M^t(x,y)= M(t,x,y) = \sup_{n\geq 1} h^{t+n}(x,y)
	\end{equation}
	
	And more generally, for every two times $s<t$, we define
	\begin{equation} \label{PeierlsGeneral}
		m^{s,t}(x,y) = \inf_{n \geq 1} h^{s,n+t}(x,y) \quad \text{and} \quad M^{s,t}(x,y) = \sup_{n \geq 1} h^{s,t+n}(x,y)
	\end{equation}
\end{defi}

\begin{rem}
	There are some subtleties that justify taking the infimum over $n \geq 1$ while excluding $n=0$. However, for the purpose of the following proposition, the chosen definition is sufficient.
\end{rem}

\begin{prop} \label{hfinite}
	For all $(t,x,y) \in [0, +\infty) \times M \times M$, the maps $m(t,x,y)$ and $M(t,x,y)$ are finite, and at $t=0$ we have the uniform bound
	\begin{equation}
		\max ( \Vert m \Vert_\infty , \Vert M \Vert_\infty ) \leq 2\kappa_1.\diam (M)
	\end{equation}
	where $\kappa_1$ is the Lipschitz constant introduced in Proposition \ref{Regularity}.
\end{prop}

\begin{proof}
	The proof of this proposition is divided into two main steps. The first step is to find points $x_n$ and $y_n$ such that $\frac{1}{n}h^n(x_n,y_n)$ converges to a finite limit using the ergodic decomposition and the A Priori Compactness Theorem \ref{APrioriCompactness}. The second step is to estimate the oscillation of $h$ with respect to this established finite limit using the Lipschitz regularity of the potential \ref{Regularity}.\\

	\textit{Step 1.} Let $\mu$ be a minimizing measure. We infer from proposition \ref{MatherNonempty} that $supp(\mu)$ is compact. We will need the ergodic decomposition of invariant measures. For that purpose, we introduce some definitions following \cite{MR0889254}.
	\begin{itemize}[label= -]
		\item Let $\Sigma^0$ be the set of points $(t,x,v) \in \supp(\mu)$ such that for every continuous map $\theta : \mathbb{T}^1 \times TM \to \mathbb{R}$, the limit 
		\begin{align*}
			\lim_{n\to +\infty} \frac{1}{n} \int_0^n \theta(t + \tau, \phi_L^{t,t+\tau}(x,v)) \;d\tau
		\end{align*}
		exists and is finite. 
		\item For all points $(t,x,v)$ of $\Sigma^0$, we define the invariant Borel measure $\mu_{(t,x,v)}$ defined with the Riesz's representation theorem as
		\begin{equation}
			\begin{matrix}
				\mu_{(t,x,v)} : &\mathcal{C}( \mathbb{T}^1 \times TM, \mathbb{R}) & \longrightarrow & \mathbb{R} \\
				& \theta & \longmapsto & \lim\limits_{n\to +\infty} \frac{1}{n} \int_0^n \theta(t + \tau, \phi_L^{t,t+\tau}(x,v)) \;d\tau
			\end{matrix}
		\end{equation}
		
		\item Let $\Sigma$ be the set of points $(t,x,v)$ of $\Sigma^0$ such that $\mu_{(t,x,v)}$ is ergodic and $(t,x,v) \in \supp (\mu_{(t,x,v)})$. 
	\end{itemize}
	Then, we have the following theorem
	\begin{theo}[Ergodic Decomposition of Invariant Measures, \cite{MR0889254}]
		The set $\Sigma$ is of full measure i.e $\mu(\Sigma) =1$, and for all map $\theta \in \mathcal{C}( \mathbb{T}^1 \times TM, \mathbb{R})$ we have
		\begin{equation}
			\int_{\mathbb{T}^1 \times TM} \left( \int_{\mathbb{T}^1 \times TM} \theta \; d\mu_{(t,x,v)} \right) d\mu = \int_{\mathbb{T}^1 \times TM} \theta \;  d\mu 
		\end{equation}
	\end{theo}
	We apply the theorem with $\theta$ being the Lagrangian $L$. Recall that the measure $\mu$ is chosen to be minimizing, hence we get
	\begin{equation} \label{ErgDecomMinimizing}
		-\alpha_0 = \int_{\mathbb{T}^1 \times TM} L\;  d\mu = \int_{\mathbb{T}^1 \times TM} \left( \int_{\mathbb{T}^1 \times TM} L \; d\mu_{(t,x,v)} \right) d\mu 
	\end{equation}
	Moreover, we know from the definition \eqref{ManeCritValue} of the \Mane critical value $\alpha_0$ that for all $(t,x,v) \in \Sigma^0$,
	\begin{equation} \label{ErgDecomMinimizing2}
		-\alpha_0 \leq  \int_{\mathbb{T}^1 \times TM} L \; d\mu_{(t,x,v)} 
	\end{equation}
	Hence, we deduce from \eqref{ErgDecomMinimizing} that for $\mu$-almost all $(t,x,v) \in \Sigma^0$, the measure $\mu_{(t,x,v)}$ is minimizing and we have equality in \eqref{ErgDecomMinimizing2}.
	
	Let $(t,x',v')$ be such a point and consider $(0,x,v) = (\Phi_L^t)^{-1}(t,x',v')$ where the map $\Phi_L$ has been defined in \eqref{InvarianceFlow}. Let us show that $\mu_{(t,x',v')} = \mu_{(0,x,v)}$. For all integer $n \geq 0$ and scalar map $\theta \in \mathcal{C}( \mathbb{T}^1 \times TM, \mathbb{R})$, we have
	\begin{align*}
		\frac{1}{n} \int_0^n \theta(\tau, \phi_L^\tau(x,v)) \; d\tau &= \frac{1}{n} \int_{-t}^{n-t} \theta(t + \tau, \phi_L^{t+\tau}(x,v)) \; d\tau \\
		&= \frac{1}{n} \int_{-t}^{n-t} \theta(t + \tau, \phi_L^{t,t+\tau}(x',v')) \; d\tau \\
		&= \frac{1}{n} \Big[ \int_{-t}^{0} \theta(t + \tau, \phi_L^{t,t+\tau}(x',v')) \; d\tau + \int_{0}^{\lfloor n-t \rfloor} \theta(t + \tau, \phi_L^{t,t+\tau}(x',v')) \; d\tau \\& + \int_{\lfloor n-t \rfloor}^{n-t} \theta(t + \tau, \phi_L^{t,t+\tau}(x',v')) \; d\tau \Big] \\
	\end{align*}
	where $\lfloor \cdot \rfloor$ stands for the floor map, and with 
	\begin{align*}
		 \left| \frac{1}{n} \int_{-t}^{0} \theta(t + \tau, \phi_L^{t,t+\tau}(x',v')) \; d\tau \right| &\leq \frac{|t|}{n} \left\Vert \theta_{|\supp(\mu)} \right\Vert_\infty \longrightarrow 0 \quad \text{as } n \to +\infty \\
		 \left| \frac{1}{n} \int_{\lfloor n-t \rfloor}^{n-t} \theta(t + \tau, \phi_L^{t,t+\tau}(x',v')) \; d\tau \right| &\leq \frac{1}{n} \left\Vert \theta_{|\supp(\mu)} \right\Vert_\infty \longrightarrow 0 \quad \text{as } n \to +\infty 
	\end{align*}
	Thus, we deduce that
	\begin{align*}
		\int_{\mathbb{T}^1 \times TM} \theta \; d\mu_{(0,x,v)} &= \lim_{n\to +\infty} \frac{1}{n} \int_0^n \theta(\tau, \phi_L^\tau(x,v)) \; d\tau \\
		&= \lim_{n\to +\infty} \frac{1}{n} \int_{0}^{\lfloor n-t \rfloor} \theta(t + \tau, \phi_L^{t,t+\tau}(x',v')) \; d\tau \\
		&= \lim_{n\to +\infty} \frac{1}{\lfloor n-t \rfloor} \int_{0}^{\lfloor n-t \rfloor} \theta(t + \tau, \phi_L^{t,t+\tau}(x',v')) \; d\tau \\
		&= \int_{\mathbb{T}^1 \times TM} \theta \; d\mu_{(t,x',v')}
	\end{align*}
	Therefore, the measure $\mu_{(0,x,v)}$ is also minimizing and we obtain
	\begin{align} \label{PeierlsFiniteDem1}
		\lim_{n\to +\infty} \frac{1}{n} \int_0^n L(\tau, \phi_L^\tau(x,v)) \; d\tau =\int_{\mathbb{T}^1 \times TM} L \; d\mu_{(0,x,v)} = - \alpha_0
	\end{align}

	Let $x_n = \pi \circ \phi_L^{n}(x,v)$. We claim that the limit above implies 
	\begin{equation} \label{PeierlsFiniteDem2}
		\lim_n \frac{1}{n} h_0^{n}(x,x_n) = - \alpha_0
	\end{equation}
	In fact, if $\gamma_n : [0,n] \to M$ is a curve realizing $h_0^{n}(x,x_n)$, we consider the measure $\nu_n$ represented by $\theta \mapsto \frac{1}{n} \int_0^n \theta(\tau,\gamma(\tau),\dot{\gamma}(\tau)) d\tau$ which is supported on $(\gamma_n, \dot{\gamma}_n)$. Let $\nu_{k_n}$ be any subsequence of $\nu_n$. By the A priori compactness Theorem \ref{APrioriCompactness}, a subsequence of $\nu_{k_n}$ weakly converges to an invariant, compactly supported Borel measure $\nu$. We obtain a subsequence $k'_n$ of $k_n$ such that 
	\begin{align*}
		\lim_n \frac{1}{k'_n} h_0^{k'_n}(x,x_{k'_n}) =  \int_{\mathbb{T}^1 \times TM} L \; d\nu \geq -\alpha_0
	\end{align*}
	where the final inequality comes from definition of the critical \Mane value $\alpha_0$. However, we know from (\ref{PeierlsFiniteDem1}) that
	\begin{align*}
		\lim_n \frac{1}{k'_n} h_0^{k'_n}(x,x_{k'_n}) \leq \lim_n \frac{1}{k'_n} \int_0^{k'_n} L(\tau, \phi_L^\tau(x,v)) \; d\tau = - \alpha_0
	\end{align*}
	We get a double inequality, and thus equality of the limits. Moreover, this procedure shows that $-\alpha_0$ is the only limit value of the sequence $\frac{1}{n} h_0^{n}(x,x_n)$, which proves (\ref{PeierlsFiniteDem2}).\\

	\textit{Step 2.} Now let $M_n = \max\limits_{M \times M} h_0^{n}$ and $m_n = \min\limits_{M \times M} h_0^{n}$. We have for all integers $n,m \geq 1$ and points $y$ and $z$ in $M$, 
	\begin{align*}
		h_0^{m+n}(y,z) \leq h_0^{m}(y,y) + h_0^{n}(y,z) \leq M_m + h_0^{n}(y,z)
	\end{align*}
	And taking the maximum on $M \times M$, we obtain the subadditive inequality
	\begin{align*}
		M_{n+m} \leq M_m + M_n
	\end{align*}
	and the sequence $M_n$ is subadditive. A classical consequence of this is that there exist $\beta \in \mathbb{R} \cup \{-\infty\}$ such that
	\begin{align} \label{PeierlsFiniteDem3}
		\lim_n \frac{M_n}{n} = \inf_{n \geq 0} \frac{M_n}{n} = \beta
	\end{align}
	Similarly, the sequence $-m_n$ is subadditive and there exist $\beta' \in \mathbb{R} \cup \{+\infty\}$ such that
	\begin{align} \label{PeierlsFiniteDem4}
		\lim_n \frac{m_n}{n} = \sup_{n \geq 0} \frac{m_n}{n} = \beta'
	\end{align}
	Additionally, we know from the regularity Proposition \ref{Regularity} on the potential $h_0$ that for all $n \geq 1$, 
	\begin{equation} \label{PeierlsFiniteDem5}
		0 \leq M_n - m_n \leq \kappa_1. 2\diam M
	\end{equation}
	where $\diam M := \max \{ d(x,y) ,  (x,y) \in M \times M \}$ is the diameter of $M$. This uniform bound immediately implies that $\beta = \beta' \in \mathbb{R}$. Considering the points $x$ and $x_n$ of the established limit (\ref{PeierlsFiniteDem2}), we get
	\begin{align*}
		\frac{m_n}{n} \leq \frac{1}{n} h^n(x,x_n) \leq \frac{M_n}{n}
	\end{align*}
	and taking the limit on $n$, we infer the equality $\beta = \beta' = - \alpha_0$.
	
	By gathering (\ref{PeierlsFiniteDem3}), (\ref{PeierlsFiniteDem4}) and (\ref{PeierlsFiniteDem5}), we deduce that
	\begin{equation*}
		- \kappa_1. 2\diam M - \alpha_0.n \leq m_n \leq -\alpha_0.n \leq M_n \leq   \kappa_1. 2\diam M - \alpha_0.n
	\end{equation*}
	leading to
	\begin{equation*}
		- \kappa_1. 2\diam M  \leq \min\limits_{M \times M} h^n = m_n + \alpha_0.n  \leq 0 \leq \max\limits_{M \times M} h^n = M_n + \alpha_0.n \leq   \kappa_1. 2\diam M 
	\end{equation*}
	Therefore, we obtain the desired bounding on the maps $m$ and $M$. The time continuity of $h$ extends the (uniform) finiteness to $m(t,x,y)$ and $M(t,x,y)$.
\end{proof}
 
\begin{rem}
	Note from the previous proof that the \Mane critical value $\alpha_0$ is finite.
\end{rem}

\section{Action of the Lax-Oleinik Operator $\mathcal{T}$ on its Non-Wandering Set $\Omega(\mathcal{T})$} \label{SectionLONW}

In this section, we prove the non-expansiveness of the Lax-Oleinik operator $\mathcal{T}$ and examine its implications for viscosity solutions and their stability. Additionally, we explore the consequences on the non-wandering set $\Omega(\mathcal{T})$ and provide a characterization of non-wandering viscosity solutions as global viscosity solutions.

\subsection{Non-Expansiveness and Boundedness of $\mathcal{T}$}

We begin with a proof of the well-known non-expansiveness of the Lax-Oleinik operator $\mathcal{T}$.

\begin{prop} \label{Contracting}
	For all time $t>0$, The time $t$ Lax-Oleinik operators $\mathcal{T}_0^t$ and $\mathcal{T}^t$ are non-expanding, \textit{i.e}
	\begin{equation} \label{NonExpansivenessFormula}
		\forall u,v \in \mathcal{C}(M,\mathbb{R}), \quad \Vert\mathcal{T}^tu - \mathcal{T}^tv\Vert_\infty = \Vert \mathcal{T}_0^tu - \mathcal{T}_0^tv\Vert_\infty \leq  \Vert u-v\Vert_\infty
	\end{equation} 
	In particular, the Lax-Oleinik operators $\mathcal{T}_0$ and $\mathcal{T}$ are non-expanding.
\end{prop}

\begin{proof}
	First, note that
	\begin{align*}
		\mathcal{T}^tu - \mathcal{T}^tv = \mathcal{T}_0^tu + \alpha_0.t- \mathcal{T}_0^tv - \alpha_0.t = \mathcal{T}_0^tu - \mathcal{T}_0^tv 
	\end{align*}
	We prove the non-expansiveness of $\mathcal{T}_0^t$. Let $x \in M$. Corollary \ref{TonelliLaxOleinik} provides with a minimizing curve $\sigma : [0,t] \to M$ with $\sigma(t) = x$ realizing the infimum in the definition (\ref{LO}) of $\mathcal{T}_0^t v(x)$. Hence we get
	\begin{align*}
		\mathcal{T}_0^tu(x) - \mathcal{T}_0^tv(x) &  = \inf_{
	\begin{matrix}
		\gamma : [0,t] \rightarrow M \\
		t \mapsto x 
	\end{matrix}
	} \left\{  u(\gamma(0)) + \int_0^t L(s,\gamma(s), \dot{\gamma}(s))\;ds \right\}
	- \left( v(\sigma(0)) +\int_0^t L(s,\sigma(s), \dot{\sigma}(s))\;ds \right) \\
	& \leq \left( u(\sigma(0)) +\int_0^t L(s,\sigma(s), \dot{\sigma}(s))\;ds \right) - \left( v(\sigma(0)) +\int_0^t L(s,\sigma(s), \dot{\sigma}(s))\;ds \right) \\
	& \leq u(\sigma(0)) - v(\sigma(0)) \leq \Vert u-v \Vert_\infty
	\end{align*}
	Since $u$ and $v$ play symmetric roles, we get the lacking inequality. 
\end{proof}

As a direct consequence, we obtain the forward-time boundedness of all viscosity solutions.

\begin{cor} \label{ViscBounded}
	For all $(s,u) \in \mathbb{R} \times \mathcal{C}(M, \mathbb{R})$, the family $\big( u(t,x) = \mathcal{T}^{s,t}u \big)_{t \geq s}$ is uniformly bounded in $\mathcal{C}(M, \mathbb{R})$.
\end{cor}
\begin{proof}
	Fix an element $(s,u) \in \mathbb{R} \times \mathcal{C}(M, \mathbb{R})$. Let $v$ be a weak-KAM solution which existence will be proved in Corollary \ref{ExistenceKAMF}. By the non-expansiveness Proposition \ref{Contracting}, we get for all $t \geq s$
	\begin{align*}
		|| \mathcal{T}^{s,t}u - \mathcal{T}^{s,t} \mathcal{T}^{s}v ||_\infty \leq || u-\mathcal{T}^{s}v ||_\infty
	\end{align*}
	And we know from the definition of weak-KAM solutions that
	\begin{align*}
		\mathcal{T}^{s,t} \mathcal{T}^{s}v = \mathcal{T}^{t} v = \mathcal{T}^{\lfloor t \rfloor, t} \mathcal{T}^{\lfloor t \rfloor} v = \mathcal{T}^{\lfloor t \rfloor, t} v = \mathcal{T}^{t- \lfloor t \rfloor} v
	\end{align*}
	Hence, we obtain from the continuity in time of $\mathcal{T}^{\tau}v$ that
	\begin{align*}
		|| \mathcal{T}^{s,t}u ||_\infty & \leq || u-\mathcal{T}^{s}v ||_\infty + ||\mathcal{T}^{s,t} \mathcal{T}^{s}v||_\infty = || u-\mathcal{T}^{s}v ||_\infty + ||\mathcal{T}^{t- \lfloor t \rfloor} v ||_\infty \\
		& \leq || u-\mathcal{T}^{s}v ||_\infty + \sup_{\tau \in [0,1]} ||\mathcal{T}^{\tau}v ||_\infty < +\infty
	\end{align*}
\end{proof}

\begin{rem} \label{ManeCharacterization}
	This boundedness result offers a characterization of the \Mane critical value $\alpha_0$. Indeed, $\alpha_0$ is the unique real constant $c$ such that there exists (or for all) scalar map $u \in \mathcal{C}(M, \mathbb{R})$, the family $\big( u_c(t,x) = \mathcal{T}_0^{s,t}u +c.(t-s) \big)_{t \geq s}$ is uniformly bounded.
\end{rem}

\subsection{Stability of Viscosity Solutions}

We present stability results for viscosity solutions of a Tonelli Hamiltonian. The stability of limits given by Proposition \ref{ViscosityLim} remains valid in more general frameworks. For a more detailed exposition on stability results for viscosity solutions, see \cite{MR1613876}. However, in the Tonelli framework, the Lax-Oleinik operator $\mathcal{T}$ allows for the stability of both the infimum and the liminf.

\begin{prop} \label{ViscosityInf}
	Let $(v_i : [s,t] \times M \to \mathbb{R})_{i \in I}$ be a family of viscosity solutions and let $u(t,x) = \inf_{i \in I} \{v_i(t,x) \}$. Then $u$ is a viscosity solution.
\end{prop} 
\begin{proof}
	Fix two real times $s \leq s'<t' \leq t$. Then, for all $x \in M$ the following computation holds
	\begin{align*}
		\mathcal{T}^{s',t'}u(s',x) & = \inf_{y \in M} \big\{ u(s', y ) + h^{s',t'}(y,x) \big\} \\
		&= \inf_{y \in M} \big\{ \inf_{i \in I} \{ v_n(s', y ) \} + h^{s',t'}(y,x) \big\} \\
		&= \inf_{i \in I} \; \inf_{y \in M} \Big\{ v_n(s', y ) + h^{s',t'}(y,x) \Big\} \\
		&= \inf_{i \in I}  \big\{ \mathcal{T}^{s',t'}v_n(s',x) \big\} = \inf_{n \geq 0}  \{ v_n(t',x) \} =  u(t',x) 
	\end{align*}
\end{proof}

\begin{prop} \label{ViscosityLim}
	Let $v_n$ be a sequence of scalar maps in $\mathcal{C}(M, \mathbb{R})$ that converges to $v$. Then, for all times $s \in \mathbb{R}$, the associated viscosity solutions $v_n : [s,+\infty) \times M \to \mathbb{R}$ defined by $v_n(t,x) = \mathcal{T}^{s,t}v_n(x)$ converge to the viscosity solution $v(t,x) = \mathcal{T}^{s,t}v(x)$ in the $C^0$ topology.
\end{prop}
\begin{proof}
	For all times $s<t$, we have by non-expansiveness of $\mathcal{T}^{s,t}$ that
	\begin{equation}
		\Vert v_n(t,\cdot) - v(t,\cdot) \Vert_\infty =  \Vert \mathcal{T}^{s,t}v_n - \mathcal{T}^{s,t}v \Vert_\infty \leq  \Vert v_n - v \Vert_\infty \longrightarrow 0 \quad \text{as } n \to +\infty
	\end{equation}
\end{proof}

\begin{prop} \label{ViscosityLiminf}
	Let $u \in \mathcal{C}(M, \mathbb{R})$ be a scalar map with associated viscosity solution $u(t,x) : [0, +\infty) \times M \to \mathbb{R}$ and let $(k_n)_n$ be an increasing sequence of integers. We define the map $v : \mathbb{R} \times M \to \overline{\mathbb{R}}$ defined by
	\begin{equation} \label{globalformula}
		v(t,x) = \liminf_n u(t+k_n,x)
	\end{equation}
	If $v$ is everywhere finite, then it is a viscosity solution of the Hamilton-Jacobi equation.
\end{prop}

\begin{proof}
	We start by proving that it is a sub-solution i.e. for all times $s \leq t$, $\mathcal{T}^{s,t}v(s, \cdot ) \leq v(t, \cdot)$. Fix $(t,x)$ in $\mathbb{R} \times M$ and  subsequence $k_{n_i}$ of $k_n$ such that $\mathcal{T}^{t+k_{n_i}} u(x)$ converges to $v(t,x)$. For every integer $i$, we get from Tonelli's Theorem \ref{TonelliTheorem} a minimizing curve $\gamma_i : [s,t] \to M$ which realizes $\mathcal{T}^{s,t} (\mathcal{T}^{s+k_{n_i}} u)(x)$ i.e. such that $\gamma_i(t) = x$ and
	\begin{equation}
		\mathcal{T}^{t+k_{n_i}} u(x) = \mathcal{T}^{s,t} (\mathcal{T}^{s+k_{n_i}} u)(x) = \mathcal{T}^{s+k_{n_i}} u( \gamma_i(s)) + \int_s^t \Big( L(\tau, \gamma_i(\tau), \dot{\gamma}_i(\tau)) +\alpha_0 \Big) \; d\tau
	\end{equation}
	These curves $\gamma_i$ are minimizing. Thus, by Corollary \ref{MinimizingC1Compactness}, we can assume up to extraction that the curves $\gamma_i$ converge to $\gamma : [s,t] \to M$ in the $C^1$-topology with $\gamma(t) = \gamma_i(t) = x$. Therefore, we get
	\begin{align*}
		v(t,x)  = \lim\limits_i \mathcal{T}^{t+k_{n_i}} u(x) &\geq \liminf_i  \mathcal{T}^{s+k_{n_i}} u( \gamma_i(s)) + \lim\limits_i \int_s^t \Big( L(\tau, \gamma_i(\tau), \dot{\gamma}_i(\tau)) +\alpha_0 \Big) \; d\tau \\
		&= \liminf_i  \mathcal{T}^{s+k_{n_i}} u( \gamma_i(s))  + \int_s^t \Big( L(\tau, \gamma(\tau), \dot{\gamma}(\tau)) +\alpha_0 \Big) \; d\tau 
	\end{align*}
	However, we know from Corollary \ref{Equicontinuity} that the family $(\mathcal{T}^{t} u)_t$ is equicontinuous, which yields
	\begin{align*}
		| \mathcal{T}^{s+k_{n_i}} u( \gamma_i(s)) -  \mathcal{T}^{s+k_{n_i}} u( \gamma(s))| \leq \kappa_1. | u( \gamma_i(s)) -  u( \gamma(s))| \longrightarrow 0 \quad \text{as } n \to \infty
	\end{align*}
	Hence,
	\begin{equation*}
		\liminf_i  \mathcal{T}^{s+k_{n_i}} u( \gamma_i(s)) = \liminf_i  \mathcal{T}^{s+k_{n_i}} u( \gamma(s)) \geq \liminf_n \mathcal{T}^{s+k_n} u( \gamma(s)) =  v(s,\gamma(s))
	\end{equation*}
	and
	\begin{align*}
		v(t,x) & \geq  v(s,\gamma(s)) + \int_s^t \Big( L(\tau, \gamma(\tau), \dot{\gamma}(\tau)) +\alpha_0 \Big) \; d\tau  \geq \mathcal{T}^{s,t}v(s, x )
	\end{align*}
	
	We now establish the inverse inequality. Let $\gamma : [s,t] \to M$ be any curve with $\gamma(t) = x$. We know from the definition of the Lax-Oleinik operator that
	\begin{align*}
		\mathcal{T}^{t+k_n} u(x) \leq \mathcal{T}^{s+k_n} u( \gamma(s)) + \int_s^t \Big( L(\tau, \gamma(\tau), \dot{\gamma}(\tau)) +\alpha_0 \Big) \; d\tau 
	\end{align*}
	and taking the liminf, we get that for each such curve $\gamma$,
	\begin{equation*}
		v(t,x) \leq v(s,\gamma(s)) +  \int_s^t \Big( L(\tau, \gamma(\tau), \dot{\gamma}(\tau)) +\alpha_0 \Big) \; d\tau 
	\end{equation*}
	which gives the desired inequality
	\begin{equation*}
		v(t,x) \leq \mathcal{T}^{s,t}v(s,x)
	\end{equation*}
\end{proof}

\begin{cor} \label{ExistenceKAMF}
	There exists a weak-KAM solution of the Hamilton-Jacobi equation \eqref{HJalpha}.
\end{cor}

\begin{proof}
	Fix a point $x_0$ in $M$ and set $u(t,x) = h^t(x_0, x)$. We know from Proposition \ref{hprop} that $u: (0,+\infty) \times M \to \mathbb{R}$ is a viscosity solution of the Hamilton-Jacobi equation \eqref{HJalpha}. Moreover, we know due to Proposition \ref{hfinite} that the map $v(t,x) = \liminf_n u(t+n,x) = h^{\infty+t}(x_0,x)$ is everywhere finite. Hence, we infer from Proposition \ref{ViscosityLiminf} that map $v(t,x) = \liminf_n u(t+n,x) = h^{\infty+t}(x_0,x)$ is also a viscosity solution of \eqref{HJalpha}. Furthermore, we have
	\begin{equation}
		\mathcal{T}v = v(t+1, \cdot) = \liminf_n u(t+1+n, \cdot) = \liminf_n u(t,\cdot) = v
	\end{equation}
	Hence, $v$ belongs to $\fix (\mathcal{T})$ and it is a weak-KAM solution.
\end{proof}

\begin{rem}
	Note that in the statement of Proposition \ref{ViscosityLiminf} we didn't assume that $v$ is finite as this is a consequence of Proposition \ref{ViscBounded} which is itself a consequence of Corollary \ref{ExistenceKAMF}.
\end{rem}

\subsection{Restriction to the Non-Wandering Set $\Omega(\mathcal{T})$} \label{SectionIsometry}

We explore the implications of the non-expansiveness of the Lax-Oleinik operator $\mathcal{T}$ on its non-wandering and recurrent sets. We start by introducing the definitions of some asymptotic objects of $\mathcal{T}$.\\

To a scalar map $u$ of $\mathcal{C}(M , \mathbb{R})$, we associate its \textit{$\omega$-limit set} $\omega(u):=\omega_{\mathcal{T}}(u)$ under the operator $\mathcal{T}$ as the set of limit points in the $\mathcal{C}(M,\mathbb{R})$ of the sequence $(\mathcal{T}^n(u))_{n \in \mathbb{N}}$. More precisely
\begin{equation}
	\omega(u) = \{ v \in \mathcal{C}(M , \mathbb{R}) \; | \; \exists (k_n)_n \in \mathbb{N}^\mathbb{N} \text{ increasing sequence s.t } \Vert \mathcal{T}^{k_n}u - v \Vert_\infty \to 0 \text{ as } n \to \infty\}
\end{equation}

\begin{defi}
	\begin{enumerate}
		\item The \textit{recurrent set} $\mathcal{R}(\mathcal{T})$ of $\mathcal{T}$ is the set of $\mathcal{T}$-recurrent elements of $\mathcal{C}(M,\mathbb{R})$ i.e. the set of $u \in \mathcal{C}(M,\mathbb{R})$ such that $u$ belongs to $\omega(u)$.
		\item Let $\underline{p} = (p_n)_n$ be an increasing of $\mathbb{N}$. An element $u$ of $\mathcal{R}(\mathcal{T})$ is said \textit{$\underline{p}$-recurrent} if $\lim_n \mathcal{T}^{p_n}u = u$.
		\item The \textit{non-wandering set} $\Omega( \mathcal{T})$ of $\mathcal{T}$ is the set of elements $u$ of $\mathcal{C}(M,\mathbb{R})$ such that for every neighbourhood $U$ of $u$ in $\mathcal{C}(M,\mathbb{R})$, there are infinitely many positive integers $k$ such that $\mathcal{T}^k(U) \cap U \neq \emptyset$.
	\end{enumerate}
\end{defi}

\begin{rem}
	Every element $u$ of $\mathcal{R}(\mathcal{T})$ or $\Omega(\mathcal{T})$ corresponds to a recurrent or a non-wandering (for integer times) viscosity solution $u(t,x)$. By abuse of language, we will often refer to the sets $\mathcal{R}(\mathcal{T})$ and $\Omega(\mathcal{T})$ as the sets of recurrent and non-wandering viscosity solutions, respectively.
\end{rem}

As a consequence of the non-expansiveness of $\mathcal{T}$ shown in Proposition \ref{Contracting}, we get the following properties

\begin{prop} \label{LONW}
	\begin{enumerate}
		\item \label{Minimality} Let $u \in \mathcal{C}(M,\mathbb{R})$. Then, its $\omega$-limit set $\omega(u)$ is compact in $\mathcal{C}(M, \mathbb{R})$, and the restriction of $\mathcal{T}$ to $\omega(u)$ is minimal, i.e. for all $v \in \omega(u)$, $\omega(u)= \overline{\{\mathcal{T}^nv \; | \; n \in \mathbb{N} \}} = \omega(v)$.
		\item \label{NW=R} The non-wandering set $\Omega( \mathcal{T})$ is equal to the recurrent set $\mathcal{R}(\mathcal{T})$.
		\item \label{NW=R=omega} The relation $u \sim v \Leftrightarrow v \in \omega(u)$ is an equivalence relation. If we denote by $\Lambda$ the set of its equivalence classes, then we have
		\begin{equation} \label{NW=R=omegaFormula}
			\Omega( \mathcal{T}) = \mathcal{R}(\mathcal{T}) = \bigsqcup_{u \in \Lambda} \omega(u)
		\end{equation}
		where the union is disjoint.
	\end{enumerate}
\end{prop}

\begin{proof}
	\ref{Minimality}. The set $\omega(u)$ is closed and $\mathcal{T}$-invariant. Let us show that it is bounded in $\mathcal{C}(M,\mathbb{R})$. For all positive integer $n \geq 0$ and for all $x$ in $M$, if we denote by $y$ the point that realizes the infimum in $\mathcal{T}^nu(x)$, then we have
	\begin{equation*}
		| \mathcal{T}^nu(x) | = | u(y) + h^n(y,x) | \leq \Vert v \Vert_\infty + \max ( \Vert m \Vert_\infty ,\Vert M \Vert_\infty ) \leq  \Vert u \Vert_\infty + \kappa_1. 2\diam(M)
	\end{equation*}
	where we used the bound on $m = \inf_n h^n$ and $M = \sup_n h^n$ found in Proposition \ref{hfinite}. Taking limits on $n$, we observe that this bound holds for every element of $\omega(u)$. Moreover, since we know from Corollary \ref{Equicontinuity} that the maps $(u(t,\cdot))_{t\geq 1}$ are equilipschitz, we infer that all the elements of $\omega(u)$ are $\kappa_1$-lipschitz. Therefore, the Arzéla-Ascoli theorem asserts that this set is compact in $\mathcal{C}(M,\mathbb{R})$.\\
	
	We show the minimality property. Let $v \in \omega(u)$ and $(n_k)_{k \in \mathbb{N}}$ be an increasing sequence of integers such that $\mathcal{T}^{n_k}u$ converges to $v \in \omega(u)$. The set $\omega(u)$ is closed and $\mathcal{T}$-invariant. Hence, we have $\overline{\{\mathcal{T}^nv \; | \; n \in \mathbb{N} \}} \subset \omega(u)$.
	
	To prove the inverse inclusion, fix $w \in \omega(u)$. There exists an increasing sequence $(p_k)_{k \in \mathbb{N}}$ such that $\mathcal{T}^{n_k+p_k}u \to w$. The non-expansiveness (\ref{NonExpansivenessFormula}) of the semi-group $\mathcal{T}$ results in
	\begin{align*}
		\Vert \mathcal{T}^{p_k}v - w\Vert_\infty & \leq \Vert \mathcal{T}^{p_k}v - \mathcal{T}^{n_k+p_k}u\Vert_\infty + \Vert \mathcal{T}^{n_k+p_k}u - w\Vert_\infty\\
		& \leq \Vert v - \mathcal{T}^{n_k}u\Vert_\infty + \Vert \mathcal{T}^{n_k+p_k}u - w\Vert_\infty \longrightarrow 0 \quad \text{as } k\to \infty
	\end{align*}
	Hence, $w \in \overline{\{\mathcal{T}^nv \; | \; n \in \mathbb{N} \}}$.
	
	By closedness and $\mathcal{T}$-invariance, we know that $\omega(v) \subset \omega(u)$. Hence, the minimality gives the inverse inclusion and the equality of sets. \\

	\ref{NW=R}. The inclusion $\mathcal{R}(\mathcal{T}) \subset \Omega( \mathcal{T})$ is immediate. We need to prove the inverse inclusion. Let $u \in \Omega( \mathcal{T})$ be a non-wandering element under $\mathcal{T}$. We aim to prove that it is recurrent. For all positive integer $n>0$, consider the set $U_n = \big\{v \in \mathcal{C}(M,\mathbb{R}) \; | \; ||v - u ||_\infty < \frac{1}{n} \big\}$. Using the non-wandering property of $u$, we inductively construct an increasing sequence of positive integers $k_n$ such that $\mathcal{T}^{k_n} U_n \cap U_n \neq \emptyset$. Let $v_n$ be an element of $U_n  \cap \mathcal{T}^{-k_n} U_n \neq \emptyset$. The properties on $U_n$ translate on $v_n$ as follows
	\begin{equation}
		||v_n - u||_\infty < \frac{1}{n} \quad \text{and} \quad ||\mathcal{T}^{k_n}v_n - u||_\infty < \frac{1}{n}
	\end{equation}
	Thus
	\begin{align*}
		||\mathcal{T}^{k_n}u - u||_\infty & \leq ||\mathcal{T}^{k_n}u - \mathcal{T}^{k_n}v_n||_\infty + ||\mathcal{T}^{k_n}v_n - u||_\infty \\
		& \leq ||u - v_n||_\infty + ||\mathcal{T}^{k_n}v_n - u||_\infty < \frac{2}{n} \longrightarrow 0 \quad \text{as } n \to \infty
	\end{align*}
	where we used the non-expansiveness of $\mathcal{T}$ and $\mathcal{T}^{k_n}$ in the second line. We deduce that the sequence $(\mathcal{T}^{k_n}u)_n$ converges uniformly to $u$, meaning that $u$ belongs to $\omega(u)$. This is the definition of recurrence. \\

	\ref{NW=R=omega}. The fact that $\sim$ is an equivalence relation is due to the minimality of $\mathcal{T}$ on $\omega(u)$. And the equalities follow immediately from the two first properties.
\end{proof}

\begin{rem} \label{rect}
	\begin{enumerate}
		\item It follows from \eqref{NW=R=omegaFormula} that the limit points of any viscosity solutions of \eqref{HJalpha} belong to a non-wandering element of $\Omega(\mathcal{T})$.  
		\item As an implication of the minimality on $\omega(u)$, we have that if $\omega(u)$ contains a periodic orbit, then it is equal to the orbit itself.
	\end{enumerate}
\end{rem}

On every $\omega$-limit set $\omega(u)$, the Lax-Oleinik operator $\mathcal{T}$ is non-expanding and surjective in a compact set. This implies that $\mathcal{T}$ is a bijective isometry on $\omega(u)$. The Proposition \ref{Isometry} generalizes this result to the entire non-wandering set $\Omega(\mathcal{T})$ and allows for the definition of its inverse operator within it. We obtain a group $(\mathcal{T}^n)_{n\in \mathbb{Z}}$ of isometries acting on $\Omega(\mathcal{T})$.

\begin{proof}[Proof of Proposition \ref{Isometry}]
	Let $v$ be an element of $\Omega(\mathcal{T})$. By recurrence, we know that there exists an increasing sequence of integers such that $\mathcal{T}^{k_n}(v)$ converges to $v$. Let $v'$ be a limit point in $\omega(v)$ of the sequence $\mathcal{T}^{k_n-1}(v)$. This exists due to the compactness of $\omega(v)$ stated in the Property \ref{Minimality} of Proposition \ref{LONW}. Then, by continuity of the Lax-Oleinik operator $\mathcal{T}$, we get $\mathcal{T}v' = v$. We showed that the restriction of $\mathcal{T}$ to $\Omega(\mathcal{T})$ is onto.
	
	Let us show that it is a bijective isometry on this set. Let $v$ and $w$ be two elements of $\Omega(\mathcal{T})$ and consider a sequences $(v_n,w_n)$ defined inductively as $(v_0,w_0) = (\mathcal{T}v, \mathcal{T}w)$, $(v_1,w_1) = (v, w)$ and for all integer $n \geq 1$, $(v_{n+1}, w_{n+1}) \in \mathcal{T}^{-1} (v_n, w_n)$. Let $k_n$ be an increasing sequence of integers such that $\lim_n k_{n+1}-k_n = + \infty$ and $(v_{k_n}, w_{k_n})$ converge to a point of the compact set $\omega(v) \times \omega(w)$. Then, we have
	\begin{align*}
		\Vert v_{k_{n+1}-k_n} - v_0 \Vert_\infty = \Vert \mathcal{T}^{k_n}v_{k_{n+1}} -\mathcal{T}^{k_n}v_{k_n} \Vert_\infty \leq \Vert v_{k_{n+1}} - v_{k_n} \Vert_\infty \longrightarrow 0 \quad \text{as } n \to +\infty 
	\end{align*}
	giving the limit $\lim_n v_{k_{n+1}-k_n} = v_0 = \mathcal{T}v$. By symmetry, we also have $\lim_n w_{k_{n+1}-k_n} = w_0 = \mathcal{T}w$. Therefore, we obtain
	\begin{align*}
		\Vert \mathcal{T}v - \mathcal{T}w \Vert_\infty &\leq \Vert v - w \Vert_\infty = \Vert v_1 - w_1 \Vert_\infty \\
		&= \Vert \mathcal{T}^{k_{n+1}-k_n-1}v_{k_{n+1}-k_n} - \mathcal{T}^{k_{n+1}-k_n-1}w_{k_{n+1}-k_n} \Vert_\infty \\
		&\leq \Vert v_{k_{n+1}-k_n} -  w_{k_{n+1}-k_n} \Vert_\infty \longrightarrow \Vert \mathcal{T}v - \mathcal{T}w \Vert_\infty \quad \text{as } n\to +\infty
	\end{align*}          
	The double inequality yields the isometric identity
	\begin{equation*}
		\Vert \mathcal{T}v - \mathcal{T}w \Vert_\infty = \Vert v - w \Vert_\infty
	\end{equation*}
	The injectivity, and hence the bijectivity of $\mathcal{T}$ follow immediately.\\
	
	For the general case, it suffices to see that for all times $s<t$, we have $\mathcal{T}^{\lfloor s \rfloor, \lceil t \rceil} = \mathcal{T}^{t, \lceil t \rceil} \circ \mathcal{T}^{s,t} \circ \mathcal{T}^{\lfloor s \rfloor,s}$ where $\lfloor \cdot \rfloor$ and $\lceil \cdot \rceil$ respectively stand for the floor and the ceil maps. Thus, for any $v$ and $w$ in $\Omega(\mathcal{T})$, we get
	\begin{align*}
		\Vert v - w \Vert_\infty &= \Vert \mathcal{T}^{\lfloor s \rfloor, \lceil t \rceil} v - \mathcal{T}^{\lfloor s \rfloor, \lceil t \rceil} w \Vert_\infty \\
		&= \Vert \mathcal{T}^{t, \lceil t \rceil} \circ \mathcal{T}^{s,t} \circ \mathcal{T}^{\lfloor s \rfloor,s}v - \mathcal{T}^{t, \lceil t \rceil} \circ \mathcal{T}^{s,t} \circ \mathcal{T}^{\lfloor s \rfloor,s}w \Vert_\infty \\
		&\leq  \Vert \mathcal{T}^{s,t} \circ \mathcal{T}^{\lfloor s \rfloor,s}v - \mathcal{T}^{s,t} \circ \mathcal{T}^{\lfloor s \rfloor,s}w \Vert_\infty \\
		&\leq  \Vert \mathcal{T}^{\lfloor s \rfloor,s}v -  \mathcal{T}^{\lfloor s \rfloor,s}w \Vert_\infty  \leq \Vert v -  w \Vert_\infty 
	\end{align*}
	We deduce equality everywhere, and the general result follows.\\
	
	If $s<\tau<t$, then we have immediately that $\mathcal{T}^{s,t} = \mathcal{T}^{\tau,t} \circ \mathcal{T}^{s,\tau}$. If for example, $s<t<\tau$, then $\mathcal{T}^{s,\tau} = \mathcal{T}^{t,\tau} \circ \mathcal{T}^{s,t}$ and $\mathcal{T}^{s,t} = (\mathcal{T}^{t,\tau})^{-1} \circ \mathcal{T}^{s,\tau} = \mathcal{T}^{\tau,t} \circ \mathcal{T}^{s,\tau}$. The other cases are symmetric.
\end{proof}

\subsection{Bounded Global Viscosity Solutions}

This subsection is dedicated to proving Theorem \ref{Glob=NW}, which characterizes the non-wandering elements of $\Omega(\mathcal{T})$ as corresponding to global viscosity solutions.

\begin{defi}
	\begin{enumerate}
		\item A \textit{global viscosity solution} is a viscosity solution $u(t,x)$ defined for all times $t \in \mathbb{R}$.
		\item We say that a scalar map $u \in \mathcal{C}(M, \mathbb{R})$ is \textit{global} if it can be associated to a global viscosity solution $u(t,x)$ with $u(0, \cdot) = u$. The set of global maps is denoted by $\mathcal{V}(\mathcal{T})$.
		\item We define the set $\mathcal{B}(\mathcal{T})$ of global maps $u$ in $\mathcal{C}(M, \mathbb{R})$ which can be associated to a bounded global viscosity solution $u: \mathbb{R} \times M \to \mathbb{R}$ with $u(0, \cdot) = u$.
	\end{enumerate}
\end{defi}
	
We aim to prove that the non-wandering set $\Omega(\mathcal{T})$ is equal to the set of global maps $\mathcal{V}(\mathcal{T})$. A first inclusion is given by the following proposition.

\begin{prop} \label{ViscosityGlobalNW}
	Every non-wandering element $v \in \Omega( \mathcal{T})$ is global and  corresponds to a unique global viscosity solution $v(t,x)$ defined by
	\begin{equation}
		v(t,x) = \mathcal{T}^t  v(x) = \lim_n \mathcal{T}^{t+k_n}v(x)
	\end{equation}
	where $k_n$ is a sequence such that $v = \lim_n v(k_n,\cdot)$ and where the limit is uniform on $t$.
\end{prop}

\begin{proof}
	Let us show uniqueness. If $v_1$ and $v_2$ are two global viscosity solutions with initial date $v_1(0,\cdot) = v_2(0,\cdot) = v \in \Omega( \mathcal{T})$. Then for all positive times $t>0$, we have $v_1(t,\cdot) = \mathcal{T}^tv = v_2(0,\cdot)$ and 
	\begin{align*}
		\mathcal{T}^{-t,0} v_1(-t,\cdot) = v_1(0, \cdot) = v = v_2(0, \cdot) = \mathcal{T}^{-t,0} v_2(-t,\cdot)
	\end{align*}
	So that, by applying $\mathcal{T}^t = (\mathcal{T}^{-t,0})^{-1}$, we obtain $v_1 = v_2$.
	
	Now set $v_1(t,x) = \mathcal{T}^t  v(x)$. We have for all $s<t$, and by Proposition \ref{Isometry}, 
	\begin{align*}
		\mathcal{T}^{s,t}v_1(s,x) = \mathcal{T}^{s,t} \circ \mathcal{T}^s  v(x) = \mathcal{T}^t  v(x) = v_1(t,x)
	\end{align*}
	Hence, $v_1$ is a global viscosity solution of \eqref{HJalpha}. Moreover, since $\mathcal{T}^t$ is isometric on $\Omega( \mathcal{T})$, we get the limit
	\begin{align*}
		\Vert \mathcal{T}^{t+k_n}v - \mathcal{T}^{t}v \Vert_\infty = \Vert \mathcal{T}^{k_n}v - v \Vert_\infty \longrightarrow 0 \quad \text{as } n \to +\infty
	\end{align*}
	yielding the second equality with uniform convergence in time.
\end{proof}

The opposite inclusion $\mathcal{V}(\mathcal{T}) \subset \Omega(\mathcal{T})$ is less trivial. Our idea is to use the intermediary set $\mathcal{B}(\mathcal{T})$ of bounded global maps. The key proposition is stated as below.

\begin{prop} \label{Bdd=NW}
	A global viscosity solution of the Hamilton-Jacobi equation \ref{HJalpha} is bounded if and only if it is non-wandering. In other words, the following equality holds
	\begin{equation}
		\Omega(\mathcal{T}) = \mathcal{B}(\mathcal{T})
	\end{equation}
\end{prop}

\begin{proof}
	We show double inclusion. We start with $\Omega(\mathcal{T}) \subset \mathcal{B}(\mathcal{T})$. Let $v$ be in $\Omega(\mathcal{T})$. By Proposition \ref{LONW}, we infer that for any integer $n \in \mathbb{Z}$, $v(n, \cdot)$ belongs to the compact set $\omega(v)$. Hence, for all real time $t$, we have 
	\begin{align*}
		v(t,\cdot) = \mathcal{T}^{t - \lfloor t \rfloor} v(\lfloor t \rfloor,\cdot) \in \mathcal{T}^{t-\lfloor t \rfloor} (\omega(v)) \subset  \mathcal{T}^{[0,1]} (\omega(v)) 
	\end{align*}       
	where the last set is compact in $\mathcal{C}(M, \mathbb{R})$ due to time continuity of the Lax-Oleinik operator $\mathcal{T}^t$. We deduce that $v$ belongs to $\mathcal{B}(\mathcal{T})$.\\
	
	Now let $v$ be in $\mathcal{B}(\mathcal{T})$. The famility $(v(-n, \cdot))_{n\geq 0}$ is bounded by definition of $\mathcal{B}(\mathcal{T})$ and equicontinuous by Corollary \ref{Equicontinuity}. Hence, applying the Arzelà-Ascoli theorem, there exists an increasing sequence of integers $k_n$ such that $\lim_n k_{n+1} - k_n = +\infty$ and $v(-k_n, \cdot)$ converges to a scalar map $v_\alpha$ in $\mathcal{C}(M, \mathbb{R})$. Set $k'_n = k_{n+1} - k_n$. We have
	\begin{align*}
		\Vert \mathcal{T}^{k'_n} v_\alpha - v_\alpha \Vert_\infty &\leq \Vert \mathcal{T}^{k'_n} v_\alpha - v( -k_n,\cdot) \Vert_\infty + \Vert v( -k_n,\cdot) - v_\alpha \Vert_\infty \\
		&\leq \Vert \mathcal{T}^{k'_n} v_\alpha - \mathcal{T}^{k'_n}v( -k_{n+1},\cdot) \Vert_\infty + \Vert v( -k_n,\cdot) - v_\alpha \Vert_\infty \\
		&\leq \Vert v_\alpha - v( -k_{n+1},\cdot) \Vert_\infty + \Vert v( -k_n,\cdot) - v_\alpha \Vert_\infty \longrightarrow 0 \quad \text{as } n \to +\infty
	\end{align*}
	where we used the non-expansiveness of $\mathcal{T}$. Hence, $\lim_n \mathcal{T}^{k'_n} v_\alpha = v_\alpha$ and $v_\alpha$ belongs to $\Omega(\mathcal{T})$. Consequently, we deduce that
	\begin{align*}
		\Vert \mathcal{T}^{k'_n} v - v \Vert_\infty &= \Vert \mathcal{T}^{k'_n + k_n} v(-k_n, \cdot) - \mathcal{T}^{k_n} v(-k_n, \cdot)  \Vert_\infty \\
		& \leq \Vert \mathcal{T}^{k'_n} v(-k_n, \cdot) - v(-k_n, \cdot)  \Vert_\infty \\
		&\leq \Vert \mathcal{T}^{k'_n} v(-k_n, \cdot) - \mathcal{T}^{k'_n} v_\alpha  \Vert_\infty + \Vert \mathcal{T}^{k'_n} v_\alpha  - v_\alpha  \Vert_\infty + \Vert v_\alpha - v(-k_n, \cdot)  \Vert_\infty \\
		&\leq \Vert v(-k_n, \cdot) - v_\alpha  \Vert_\infty + \Vert \mathcal{T}^{k'_n} v_\alpha  - v_\alpha  \Vert_\infty + \Vert v_\alpha - v(-k_n, \cdot)  \Vert_\infty \quad \longrightarrow 0 \quad \text{as } n \to +\infty
	\end{align*}
	Therefore, $\lim_n \mathcal{T}^{k'_n} v = v$ and $v$ belongs to $\Omega (\mathcal{T})$.

\end{proof}

\begin{rem}
	Note that this proposition implies that for a viscosity solution, being recurrent in positive times is equivalent to being recurrent in negative times.
\end{rem}

It remains to show the equality of the sets $\mathcal{V}(\mathcal{T}) = \mathcal{B}(\mathcal{T})$ stated in Proposition \ref{Bdd=NW}.

\begin{proof}[Proof of Proposition \ref{Bdd=NW}]
	We need to show that a global viscosity solution $u(t,x)$ is bounded. Positive time boundedness is given by Corollary \ref{ViscBounded}. For negative time boundedness, we fix a time $t \leq -1$ and a point $x \in M$. Let $y'$ be an arbitrary point of $M$ and consider the point $y \in M$ such that
	\begin{equation*}
		u(0,y') = u(t,y) + h^{t,0}(y,y')
	\end{equation*}		
	Then, using the $\kappa_1$-lipschitz regularity of both $h^{s,t}$ and $u(t,x)$ seen in Proposition \ref{Regularity} and Corollary \ref{Equicontinuity}, we get
	\begin{align*}
		|u(t,x)| &\leq |u(t,x) - u(t,y)| + |u(t,y)| \\
		&\leq |u(t,x) - u(t,y)| + |u(0,y')| + |h^{t,0}(y,y')| \\
		& \leq |u(t,x) - u(t,y)| + |u(0,y')| + |h^{t,0}(y,y') - h^{\lfloor t \rfloor,0}(y,y')| + |h^{\lfloor t \rfloor,0}(y,y')| \\
		& \leq \kappa_1.\diam (M) + \Vert u(0,\cdot) \Vert_\infty + \kappa_1| t - \lfloor t \rfloor | + \Vert M \Vert_\infty \\
		& \leq \Vert u(0,\cdot) \Vert_\infty + \kappa_1. \big( 3\diam (M) +1 \big)
	\end{align*}
	where we used the inequality of Proposition \ref{hfinite} on the map $M(x,y) = \sup\limits_{n \geq 1} h^{0,n}(x,y) = \sup\limits_{n \geq 1} h^{-n,0}(x,y)$.
\end{proof}

\begin{proof}[Proof of Theorem \ref{Glob=NW}]
	Combining Propositions \ref{Bdd=NW} and \ref{Global=Bdd}, we infer the desired equality
	\begin{equation*}
		\Omega(\mathcal{T}) = \mathcal{B}(\mathcal{T}) = \mathcal{V}(\mathcal{T})
	\end{equation*}
\end{proof}

\section{The Uniqueness Theorem in $\Omega(\mathcal{T})$} \label{SectionUniqueness}

This section is dedicated to proving a Uniqueness Theorem \ref{Uniqueness} for non-wandering viscosity solutions on the Mather set. This theorem will play a central role in the various representation formulas presented in the next sections, allowing to determine non-wandering viscosity solutions based on their restrictions to the Mather set.\\

This theorem was initially established by A. Fathi for weak-KAM solutions $\fix(\mathcal{T})$ in the autonomous framework (see \cite{fathi2008weak}) and by M. Zavidovique for the discrete case (see \cite{zavidovique2023discretecontinuousweakkam}). Subsequently, P. Bernard and J.-M. Roquejoffre extended it to non-wandering viscosity solutions in \cite{MR2041603}. We offer an alternative proof of their version, which is a more classical approach adapted from A. Fathi's original proof.

\subsection{Calibration on the Mather Set $\mathcal{M}$}

\begin{prop} \label{CalibNW}
	Let $x$ be an element of the Mather set $\mathcal{M}_0$ with lift $\tilde{x}$ in $\tilde{\mathcal{M}}_0$ and let $\gamma : \mathbb{R} \to M$ be the projection on $M$ of the Lagrangian flow at $\tilde{x}$ i.e. $\gamma(t) = \pi \circ \phi_L^t(\tilde{x})$. Then every recurrent viscosity solution $v$ in $\Omega(\mathcal{T})$ is calibrated by the curve $\gamma$.
\end{prop}

\begin{proof}
	Let $v$ be an element $\Omega(\mathcal{T})$. Let $k_n$ be a sequence of integers such that $v(k_n, \cdot)$ converges to $v$ in $\mathcal{C}(M, \mathbb{R})$. We know from the definition of viscosity solutions that $\mathcal{T}^tv = v(t, \cdot)$. Thus, for all real time $t$ and all point $\tilde{y}$ of $TM$, the definition of the operator $\mathcal{T}$ gives
	\begin{equation} \label{Domination1}
		v\big(t+k_n,\pi \circ \phi_L^{t,t+k_n}(\tilde{y})\big) - v\big(t,\pi \circ \tilde{y}\big) \leq A_L(\pi \circ \phi_L^{t,\tau}(\tilde{y}))  = \int_{t}^{t+k_n} \Big(  L(\tau, \phi_L^{t,\tau}(\tilde{y})) +\alpha_0 \Big) \; d\tau
	\end{equation}
	Let $\mu$ be a minimizing measure on $\mathbb{T}^1 \times TM$ that has $\tilde{x}$ in its support $\supp(\mu)$. We keep the same notation $\mu$ for its time-one periodic lift $\mu$ to $\mathbb{R} \times TM$. We integrate (\ref{Domination1}) in $(t,\tilde{y}) \in [0, 1]\times TM$ with respect to the lift $\mu$.
	\begin{multline} \label{DomIntegrated}
		\int_0^{1} \int_{TM}v\big(t+k_n,\pi \circ \phi_L^{t,t+k_n}(\tilde{y})\big) \; d\mu - \int_0^{1} \int_{TM} v(t,\pi ( \tilde{y})) \; d\mu \\
		\leq \int_0^{1} \int_{TM} \int_{t}^{t+k_n} \Big( L(\tau, \phi_L^{t,\tau}(\tilde{y})) +\alpha_0 \Big) \; d\tau \; d\mu
	\end{multline}
	
	We computate of the right-hand side while taking in consideration the time periodicity of the Lagrangian $L$ and the $\Phi_L^\tau$-invariance of $\mu$ where $\Phi_L^{k_n}$ has been defined in \eqref{InvarianceFlow}
	\begin{equation} \label{DominationRHS}
		\begin{split}
			\int_0^{1} \int_{TM} \int_{t}^{t+k_n} \Big( L(\tau, \phi_L^{t,\tau}(\tilde{y})) +\alpha_0 \Big) \; d\tau \; d\mu &= \int_0^{1} \int_{TM} \int_0^{k_n} \Big( L(t+ \tau, \phi_L^{t,t+ \tau}(\tilde{y})) +\alpha_0 \Big) \; d\tau \; d\mu \\
			&= \int_0^{k_n} \int_0^{1} \int_{TM} \Big( L(t+ \tau, \phi_L^{t,t+\tau}(\tilde{y})) +\alpha_0 \Big) \; d\mu \; d\tau \\
			&= \int_0^{k_n} \left( \int_0^{1} \int_{TM} L(t,\tilde{y}) \; d\mu \; +\alpha_0 \right) \; d\tau =0
		\end{split}
	\end{equation}

	We shift our focus to the left hand side of (\ref{DomIntegrated}). The $\Phi_L^{k_n}$-invariance of $\mu$ results in 
	\begin{align*}
		\int_0^{1} \int_{TM}v\big(t+k_n,\pi \circ \phi_L^{t,t+k_n}(\tilde{y})\big) \; d\mu &= \int_{k_n}^{k_n+1} \int_{TM}v(t,\pi (\tilde{y})) \; d\mu \\
		&= \int_0^1 \int_{TM}v(t+k_n,\pi (\tilde{y})) \; d\mu 
	\end{align*}	
	Thus 
	\begin{multline} \label{DominationLHS1}
		\int_0^{1} \int_{TM}v\big(t+k_n,\pi \circ \phi_L^{t,t+k_n}(\tilde{y})\big) \; d\mu - \int_0^{1} \int_{TM} v(t,\pi ( \tilde{y})) \; d\mu \\
		= \int_0^{1} \int_{TM}  v(t+k_n,\pi(\tilde{y})) \; d\mu - \int_0^{1} \int_{TM} v(t,\pi ( \tilde{y})) \; d\mu \\
		= \int_{0}^{1} \int_{TM} v(k_n+t,\pi(\tilde{y})) - v(t,\pi ( \tilde{y})) \; d\mu \\
	\end{multline}
	However, we know from Proposition \ref{ViscosityGlobalNW} that the restrictions $v(t+k_n,x)$ uniformly converge to $v$ on $\mathbb{R} \times M$. Hence, we deduce that 
	\begin{equation} \label{DominationLHS2}
		\lim_n \int_{0}^{1} \int_{TM} v(k_n+t,\pi(\tilde{y})) - v(t,\pi ( \tilde{y})) \; d\mu =0
	\end{equation}
	
	Now, Gathering \eqref{DomIntegrated}, \eqref{DominationRHS}, \eqref{DominationLHS1} and \eqref{DominationLHS2}, and if we define the defect of calibration $\delta_{k_n}(t,\tilde{y})$ by
	\begin{equation}
		\delta_{k_n}(t,\tilde{y}) = \alpha_0.k_n + \int_{t}^{t+k_n} L(\tau, \phi_L^{t,\tau}(\tilde{y})) \; d\tau - [v\big(t,\pi \circ \tilde{y}\big) - v\big(t,\pi \circ \tilde{y}\big)] \geq 0
	\end{equation}
	we get 
	\begin{align*}
		\lim_n \int_0^1 \int_{TM} \delta_{k_n}(t,\tilde{y}) \; d\mu =0
	\end{align*}
	where the integrand is non-negative and increasing in $n$. This implies that for $\mu$-almost all $(t,\tilde{y})$ in $\supp(\mu)$, we have for all $s>0$,
	\begin{equation*}
		\delta_{s}(t,\tilde{y}) = \alpha_0.s + \int_{t}^{t+s} L(\tau, \phi_L^{t,\tau}(\tilde{y})) \; d\tau - [v\big(t,\pi \circ \tilde{y}\big) - v\big(t,\pi \circ \tilde{y}\big)] =0
	\end{equation*} 
	And by continuity of $L$, $v$ and hence $\delta_{s}$, the equality extends to $\supp(\mu)$. Since $\mu$ is invariant by the Lagrangian flow $\phi_L$ and $\tilde{x}$ belongs to its support, we infer that the graph of the curve $(t,\gamma(t))$ also belongs to $\supp(\mu)$ so that for all negative integer $m \leq 0$, $(m, \gamma(m)$ also belongs to $\supp(\mu)$ and the associated curve are calibrated by all non-wandering viscosity solutions in $\Omega
( \mathcal{T})$. Moreover, by Proposition \ref{ViscosityGlobalNW}, we have $v(m +k_n, \cdot)$ converges to $v(m, \cdot)$ and in particular, $v(m, \cdot)$ belongs to $\Omega(\mathcal{T})$ and it calibrates by $\gamma(t+m)$. Therefore, $\gamma$ is calibrated by $v$ on $\mathbb{R}$.
\end{proof}

\begin{rem} \label{MatherMinimizing}
	We infer from Remark \ref{CalibEL} that all curves in the Mather set $\mathcal{M}$ are minimizing.
\end{rem}

\subsection{The Uniqueness Theorem}

\begin{lem} \label{AlphaMather}
	Let $u$ be in $\Omega(\mathcal{T})$ with corresponding global viscosity solution $u(t,x): \mathbb{R}\times M \to \mathbb{R}$. Let $x$ be an element of $M$ and $\gamma : ( - \infty,0] \to M$ with $\gamma(0)=x$ be a $u$-calibrated curve. Then there exists an increasing sequence of integers $k_n$ such that $\gamma(-k_n)$ converges to a point $x_\alpha$ of $\mathcal{M}_0$.
\end{lem}

\begin{proof}
	Using Riesz representation theorem, we define for all positive time $t>0$ a Borel probability measure $\mu_t : \mathcal{C}(\mathbb{T}^1 \times TM, \mathbb{R}) \to \mathbb{R}$ by
	\begin{equation*}
		\mu_t(\theta) = \frac{1}{t} \int_{-t}^0 \theta(s, \gamma(s), \dot{\gamma}(s)) \;ds
	\end{equation*}
	We know from remark \ref{CalibEL} that $\gamma$ is minimizing. Hence, the A priori compactness of minimizing curves implies that $(\gamma, \dot{\gamma})$ belongs to a compact subset $K$ of $TM$, so that $\text{supp}(\mu_t) \subset K$ is compact. Consequently, we can extract an increasing sequence of integers $k_n$ such that $\mu_{k_n}$ weak-$\ast$ converges to a compactly supported probability measure $\mu$. Additionally, since the curve $\gamma$ is minimizing, it follows the Lagrangian flow $\phi_L$ and we deduce that the measure $\mu$ is $\Phi_L$-invariant. 
	
	Weak-$\ast$ convergence applied to $\theta = L$ gives
	 \begin{align*}
	 	\int_{\mathbb{T}^1 \times TM} L \; d\mu = \lim\limits_n \int_{\mathbb{T}^1 \times TM} L \; d\mu_{k_n} = \lim\limits_n \frac{1}{k_n} \int_{-k_n}^0 L(s, \gamma(s), \dot{\gamma}(s)) \;ds
	 \end{align*}
	 Addtionally, the calibration of $\gamma$ yields
	 \begin{align*}
	 	\frac{1}{k_n} \int_{-k_n}^0 L(s, \gamma(s), \dot{\gamma}(s)) \;ds + \alpha_0= \frac{1}{k_n}  \big[ u(0,x) - u(-k_n, \gamma(-k_n)) \big]
	 \end{align*}
	 However, Theorem \ref{Global=Bdd} and Proposition \ref{Bdd=NW} claim that $u$ is bounded. Thus, we deduce that
	 \begin{align*}
	 	\int_{\mathbb{T}\times TM} L \; d\mu + \alpha_0 =  \lim\limits_n \frac{1}{k_n}  \big[ u(0,x) - u(-k_n, \gamma(-k_n)) \big] = 0
	 \end{align*}
	 and the measure $\mu$ is minimizing. Therefore, by definition of the Mather set $\tilde{\mathcal{M}}$, we get the inclusion
	 \begin{equation*}
	 	\text{supp}(\mu) \subset \tilde{\mathcal{M}}
	 \end{equation*}
	 
	 We now show that $\supp(\mu)$ belongs to the $\alpha$-limit of the curve $\gamma$. Let $(0,\tilde{x}) = (0,x,v)$ be an element of $\text{supp}(\mu)$ and for all $m \geq 1$, let $A_m$ and $B_m$ be the balls of center $(0,\tilde{x})$ and respective radii $1/m$ and $1/(2m)$ in some chart of $\mathbb{T}^1 \times TM$. We consider a continuous bump map $\chi_m : \mathbb{T}^1 \times TM \to [0,1]$ supported on $A_m$ and equal to $1$ on $B_m$. Then, if we denote by $\chi_{A_m}$ the indicator function of the set $A_m$, we have $\chi_{B_m} \leq \chi_m \leq \chi_{A_m} $ and
	\begin{equation*}		
		\mu_{k_n}(B_m) \leq  \mu_{k_n}(\chi_m) \leq \mu_{k_n}(A_m) \quad \text{and} \quad 0<\mu(B_m) \leq  \mu(\chi_m) \leq \mu(A_m)
	\end{equation*}		 
	where the positivity of $\mu(B_m)$ is due to the fact that $B_m$ is a neighbourhood of $(0,\tilde{x}) \in \supp(\mu)$. And by definition of the weak-$\ast$ convergence, we obtain
	\begin{equation} \label{AlphaMatherDem1}
		0 < \mu(\chi_m)  = \lim_n  \mu_{k_n}(\chi_m) \leq \liminf_n \mu_{k_n}(A_m)
	\end{equation}
	
	Moreover, we have
	\begin{align*}
		\mu_{k_n}(A_m) &= \frac{1}{k_n} \int_{-k_n}^0 \chi_{A_m}(s, \gamma(s), \dot{\gamma}(s)) \;ds \\
		&= \frac{1}{k_n} \leb \big( \{ s \in [-k_n,0] \; | \; (s,\gamma(s), \dot{\gamma}(s)) \in A_m \} \big)
	\end{align*}
	where $\leb$ stands for the Lebesgue measure. We consider the real number $t_n \in \mathbb{R}$ defined by
	\begin{equation*}
		t_n := \inf \big\{ \tau \in [0,k_n] \; \big| \; \{ s \in [-k_n,0] \; | \; (s,\gamma(s), \dot{\gamma}(s)) \in A_m \} \subset [-\tau,0] \big\}
	\end{equation*}
	We have $(-t_n,\gamma(-t_n), \dot{\gamma}(-t_n))$ belongs to the closure $\overline{A}_m$ of $A_m$, and
	\begin{equation*}
		\mu_{k_n}(A_m)  \leq \frac{1}{k_n} \leb([-t_n,0]) = \frac{t_n}{k_n}
	\end{equation*}
	Thus, we infer from the inequalities \eqref{AlphaMatherDem1} that
	\begin{equation*}
		0 < \mu(\chi_m) \leq \liminf_n \frac{t_n}{k_n}
	\end{equation*}
	and since $k_n$ diverges to $+\infty$, we deduce that $\lim_n t_n = +\infty$.
	
	For all $m \geq 1$, we constructed a sequence $t_n$ such that 
	\begin{equation} \label{AlphaMatherDem2}
		\lim_n t_n = +\infty \quad \text{and} \quad (-t_n - \lfloor -t_n \rfloor ,\gamma(-t_n), \dot{\gamma}(-t_n)) \in \overline{A}_m \subset \mathbb{T}^1 \times TM
	\end{equation}
	By extracting, we get two increasing sequences $m_i$ and $t_{n_i}$ verifying \eqref{AlphaMatherDem2}. And since $m_i$ is increasing and $A_m$ is of radius $1/m_i$, we obtain 
	\begin{equation*}
		\lim_i \; (-t_{n_i} - \lfloor -t_{n_i} \rfloor ,\gamma(-t_{n_i}), \dot{\gamma}(-t_{n_i})) = (0, \tilde{x})
	\end{equation*}
	and in particular
	\begin{equation*}
		\lim_i \; (-t_{n_i} - \lfloor -t_{n_i} \rfloor ,\gamma(-t_{n_i})) = (0, x) \in \pi(\text{supp}(\mu)) \cap \big( \{0\} \times M \big)  \subset \mathcal{M}_0
	\end{equation*}
	
	In order to obtain integer times instead of $t_{n_i}$, we recall that $(\gamma,\dot{\gamma})$ belongs to a compact set $K$ of $TM$ and it follows the Lagrangian flow $\phi_L$ so that $\phi_L^{- \lfloor -t_{n_i} \rfloor,-t_{n_i}} = \phi_L^{0,-t_{n_i} - \lfloor -t_{n_i} \rfloor}$ restricted to $K$ converges uniformly to the identity and
	\begin{align*}
		\lim_i \; (\gamma(\lfloor -t_{n_i} \rfloor), \dot{\gamma}(\lfloor -t_{n_i} \rfloor)) = \lim_i \; \left(\phi_L^{0,-t_{n_i} - \lfloor -t_{n_i} \rfloor}\right)^{-1} (\gamma(-t_{n_i}), \dot{\gamma}(-t_{n_i})) = id(\tilde{x}) = \tilde{x}
	\end{align*}

\end{proof}

\begin{theo} \label{Uniqueness}
	Let $u$ and $v$ be two non-wandering viscosity solutions in $\Omega(\mathcal{T})$ such that $u_{|\mathcal{M}_0} = v_{|\mathcal{M}_0}$. Then $u=v$ everywhere.
\end{theo}

\begin{proof}
	Let $(t,x)$ be a fixed element of $\mathbb{T}^1 \times M$ and $\gamma : (- \infty, t] \to M$ with $\gamma(t)=x$ be a curve calibrated by $v$ given by proposition \ref{CalibExist}. For all times $s_1 < s_2$, we have
	\begin{equation} \label{UniquenessDem1}
		v(s_2, \gamma(s_2)) - v(s_1, \gamma(s_1)) = \int_{s_1}^{s_2} L(\tau, \gamma(\tau), \dot{\gamma}(\tau)) \; d\tau + \alpha_0.(s_2-s_1)
	\end{equation}
	and from the definition of viscosity solutions 
	\begin{equation*}
		u(s_2, \gamma(s_2)) - u(s_1, \gamma(s_1)) \leq \int_{s_1}^{s_2} L(\tau, \gamma(\tau), \dot{\gamma}(\tau)) \; d\tau + \alpha_0.(s_2-s_1)
	\end{equation*}
	Replacing the right hand side by \eqref{UniquenessDem1}, we get
	\begin{equation*}
		u(s_2, \gamma(s_2)) - u(s_1, \gamma(s_1)) \leq v(s_2, \gamma(s_2)) - v(s_1, \gamma(s_1))
	\end{equation*}
	and more precisely 
	\begin{equation*}
		\big(u - v\big) (s_2, \gamma(s_2))\leq  \big(u-v\big)(s_1, \gamma(s_1))
	\end{equation*}
	Hence $(u-v)(s,\gamma(s))$ is non-increasing in time $s$. Since $u$ and $v$ are bounded by Theorem \ref{Global=Bdd} and Proposition \ref{Bdd=NW}, it follows that $\big(u - v\big) (s, \gamma(s))$ has a finite limit $l$ at $- \infty$ that we will determine.
	
	Now apply Lemma \ref{AlphaMather} to the curve $\gamma : \big(- \infty, \lfloor t \rfloor \big] \to M$ to get a an increasing sequence of integers $k_n$ and an element $x_\alpha$ of $\mathcal{M}_0$ such that $\gamma(- k_n)$ converges to $x_\alpha$. We are interested in computing
	\begin{equation*}
		l = \lim\limits_{s \to -\infty} u(s,\gamma(s)) - v(s,\gamma(s)) = \lim\limits_n u(-k_n,\gamma(-k_n)) - v(-k_n,\gamma(-k_n))
	\end{equation*}
	The equicontinuity of the families $\big( u(t,\cdot) \big)_t$ and $\big( v(t,\cdot) \big)_t$ given by Corollary \ref{Equicontinuity} allow to replace $\gamma(-k_n)$ by its limit $x_\alpha$ to get
	\begin{equation*}
		l = \lim\limits_n u(-k_n,x_\alpha) - v(-k_n,x_\alpha)
	\end{equation*}
	Set $x_\alpha(t) = \pi \circ \phi_L^{-k_n,t}(\tilde{x}_\alpha) $ where $\tilde{x}_\alpha \in \tilde{\mathcal{M}}_0$ is the lift of $x_\alpha$. Proposition \ref{CalibNW} implies that $x_\alpha(t)$ is calibrated by $u$ and $v$, and
	\begin{equation*}
		u(0, x_\alpha(0)) - u(-k_n,x_\alpha) = \int_{-k_n}^0 L(s,x_\alpha(s) \dot{x_\alpha}(s)) \; ds + \alpha_0.k_n = v(0, x_\alpha(0)) - v(-k_n,x_\alpha)
	\end{equation*}
	and since $x_\alpha(0)$ belongs to $\mathcal{M}_0$ and $u_{|\mathcal{M}_0} = v_{|\mathcal{M}_0}$, we obtain
	\begin{equation*}
		u(-k_n,x_\alpha) - v(-k_n,x_\alpha) = u(0, x_\alpha) - v(0, x_\alpha) =0
	\end{equation*}
	
	We proved that $(u-v)(s,\gamma(s))$ is non-increasing in time $s$ and has a null limit at $-\infty$. Therefore, $(u-v)(s,\gamma(s))$ is non-negative and
	\begin{equation*}
		u(t,x) = u(t, \gamma(t)) \leq  v(t, \gamma(t)) = v(t,x)
	\end{equation*}
	The symmetry between $u$ and $v$ gives the inverse inequality and we conclude that $u(t,x) = v(t,x)$ for all $(t,x)$ in $\mathbb{T}^1 \times M$.
\end{proof}

\section{Representation Formulas for the Non-Wandering Set $\Omega(\mathcal{T})$} \label{SectionRepresentation}

This is the central section of the article, divided into three subsections. The first subsection introduces the Peierls barrier, a crucial element for developing more general barriers. The second subsection presents a first, non-canonical representation formula, which is easily adaptable to explicit examples. The final subsection is dedicated to the introduction of the generalized Peierls barrier $\underline{k}$ mentioned in the introduction, and to the proof of the main representation Theorem \ref{kRepresentationTheorem} of this article.

\subsection{Peierls Barriers}

The Peierls barrier, introduced by Mather in \cite{MR1275203}, provides various viscosity solutions to the Hamilton-Jacobi equation. These solutions are crucial for describing all weak-KAM solutions (See \cite{CIS}). In this section, we introduce a generalized Peierls barriers, which will enable us to describe all non-wandering viscosity solutions through representation formulas.

\begin{defi} \label{pPeierlsDefi}
	\begin{enumerate}
		\item For any increasing sequence $\underline{p} = (p_n)_{n \geq 0}$ in $\mathbb{N}$, we define the \textit{$\underline{p}$-Peierls Barrier} $h^{\underline{p}}: M \times M \to \mathbb{R}$ by
		\begin{equation}
			h^{\underline{p}}(x,y) = \liminf_n h^{p_n}(x,y)
		\end{equation}
		with the corresponding time-dependant $\underline{p}$-barrier
		\begin{equation}
			h^{\underline{p}+t}(x,y)= h^{\underline{p}}(t,x,y) = \liminf_n h^{t+p_n}(x,y)
		\end{equation}
		where $h^t$ is the potential explicited in (\ref{Potential}).
		
		More generally, for all two times $s$ and $t$, we define
		\begin{equation} \label{PeierlsGeneral}
			h^{s,\underline{p}+t}(x,y) = \liminf_{n \to \infty} h^{s,p_n+t}(x,y)
		\end{equation}
		
		\item The \textit{Peierls Barrier} $h^\infty : M \times M \to \mathbb{R}$ is the $\underline{p}$-barrier for $p_n = n$ defined by
		\begin{equation}
			h^\infty(x,y) = \liminf_{n \to \infty} h^n(x,y)
		\end{equation}
	\end{enumerate}
\end{defi}

\begin{prop} \label{pPeierlsProp}
	\begin{enumerate}
		\item \label{PeierlsFiniteness}(Finiteness) For all $(t,x,y) \in \mathbb{R} \times M \times M$, the Peierls barrier $h^{\underline{p}}(t,x,y)$ is finite.
		\item \label{PeierlsRegularity} (Regularity) The Peierls barrier $h^{\underline{p}}$ is $\kappa_\varepsilon$-Lipschitz for all $\varepsilon >0$ with $\kappa_\varepsilon$ being the Lipschitz constant introduced in Proposition \ref{Regularity}.
		\item \label{pvisc} (Viscosity Solution) For all $x \in M$, $h^{\underline{p}}(\cdot,x,\cdot)$ is a viscosity solution of the Hamilton-Jacobi equation (\ref{HJalpha}).
		\item \label{PeierlsLiminfProp} (Liminf Property) For all points $x$ and $y$ in $M$ and all sequences of points $(x_n)_n$ and $(y_n)_n$ in $M$ respectively converging to $x$ and $y$, we have
		\begin{equation} \label{PeierlsLiminf}
			h^{\underline{p}}(x,y) = \liminf_n h^{p_n}(x_n,y_n)
		\end{equation}
	\end{enumerate}
\end{prop}

\begin{proof}
	\ref{PeierlsFiniteness}. Direct consequence of Proposition \ref{hfinite}.

	\ref{PeierlsRegularity}. Direct consequence of Proposition \ref{Regularity}.
	
	\ref{pvisc}. For all $x \in M$, we know from Proposition \ref{hprop} that $h^{p_n}(\cdot,x,\cdot)$ is a viscosity solution. Hence, we deduce from Proposition \ref{ViscosityLiminf} that $h^{\underline{p}}(\cdot,x,\cdot) = \liminf_n h^{p_n}(\cdot,x,\cdot)$ is also a viscosity solution.
	
	\ref{PeierlsLiminfProp}. We have from the regularity of $h$ that for all integers $n$ and $k \geq 1$
	\begin{align*}
		|h^k(x,y) - h^k(x_n,y_n)| \leq \kappa_1. \big( d(x,x_n) + d(y,y_n) \big)
	\end{align*}
	Hence, setting $k = p_n$ and taking the liminf on $n$, we get 
	\begin{align*}
		h^{\underline{p}}(x,y) = \liminf_n h^{p_n}(x,y) = \liminf_n h^{p_n}(x_n,y_n)
	\end{align*}
\end{proof}

We also give properties specific to the original Peierls barrier $h^\infty$.

\begin{prop} \label{PeierlsProp}
	\begin{enumerate}
		\item \label{PeierlsWeakKAM} (Weak-KAM Solution) For all $x \in M$, $h^\infty(x,\cdot)$ is a weak-KAM solution of the translated Hamilton-Jacobi equation (\ref{HJalpha}).
		\item \label{PeierlsTriangIneg} (Triangular Inequality) For all point $x$, $y$ and $z$ in $M$, we have the triangular inequality
		\begin{equation}\label{TriangInegPeierls}
			h^\infty(x,z) \leq h^\infty(x,y) + h^\infty(y,z)
		\end{equation}
		\item \label{PeierlsTriangIneg2} And for all times $s< \tau < t$, we have the triangular inequalities
		\begin{equation}\label{TriangInegPeierls2}
			\begin{split}
				h^{s,\infty+t}(x,z) &\leq h^{s,\infty+\tau}(x,y) + h^{\tau,t}(y,z)\\
				h^{s,\infty+t}(x,z) &\leq h^{s,\tau}(x,y) + h^{\tau,\infty + t}(y,z)
			\end{split}
		\end{equation}
		
		\item (Non-negativity) \label{PeierlsPositivity} For every point $x$ in $M$, we have $h^\infty(x,x) \geq 0$.
	\end{enumerate}
\end{prop}

\begin{proof}
	\ref{PeierlsWeakKAM}. It is the explicited weak-KAM solution in the proof of Proposition \ref{ExistenceKAMF} on existence of weak-KAM solutions.\\
	
	\ref{PeierlsTriangIneg}. Fix three points $x$, $y$ and $z$ in $M$. Let $k_n$ and $q_n$ be two increasing sequences of integers such that $h^\infty(x,y) = \lim_n h^{k_n}(x,y)$ and $h^\infty(y,z) = \lim_n h^{q_n}(y,z)$. Then, we have
	\begin{align*}
		h^\infty(x,z) &=  \liminf_{n \to \infty} h^{n}(x,y) \leq \liminf_n h^{k_n + q_n}(x,y) \\
		&\leq \liminf_n h^{k_n}(x,y) + h^{q_n}(y,z) \\
		&= \lim_n h^{k_n}(x,y) + h^{q_n}(y,z) \\
		&= h^\infty(x,y) + h^\infty(y,z)
	\end{align*}
	where we used the triangular inequality (\ref{TriangIneg}) in the second line.\\
	
	\ref{PeierlsTriangIneg2}. As the previous point, we fix three points $x$, $y$ and $z$ in $M$ and three times $s$ and $\tau <t$. Let $k_n$ be an increasing sequences of integers such that $h^{s,\infty+\tau}(x,y) = \lim_n h^{s,k_n+\tau}(x,y)$. Then we have,
	\begin{align*}
		h^{s,\infty+t}(x,z) &=  \liminf_{n \to \infty} h^{s,n+t}(x,y) \leq \liminf_n h^{s,k_n +t}(x,y) \\
		&\leq \liminf_n h^{s,k_n +\tau}(x,y) + h^{k_n +\tau,k_n +t}(y,z) \\
		&=h^{s,\infty +\tau}(x,y) + h^{\tau,t}(y,z) 
	\end{align*}
	The second inequality is analogous.\\
	
	\ref{PeierlsPositivity}. Consider a weak-KAM solution $u$. Then, for all point $x$ in M and for all integer $n \geq 0$, we get from the definition of $u$ that
	\begin{align*}
		0 = u(x) - u(x) = u(n,x) - u(0,x) \leq  h^n(x,x)
	\end{align*}
	Taking the liminf on $n$, we get the inequality $h^\infty(x,x) \geq 0$.
\end{proof}

\begin{defi}
	We define the \textit{Peierls set} $\mathcal{A}_0$ in $M$ as follows
	\begin{equation}
		\mathcal{A}_0 = \{ x \in M \; | \; h^\infty(x,x) = 0 \}
	\end{equation}
\end{defi}

\begin{prop} \label{MatherPeierls}
	For all $x \in \mathcal{M}_0$, we have $h^\infty(x,x) = 0$. In other words, we have the inclusion of sets $\mathcal{M}_0 \subset \mathcal{A}_0$
\end{prop}

\begin{proof}
	Let $x$ be in $\mathcal{M}^R_0$ with lift $\tilde{x}$ in $\tilde{\mathcal{M}}$ and set $x(t) = \pi \circ \phi_L^t(\tilde{x})$. Let $k_n$ be an increasing sequence of integers such that $x(-k_n)$ converge to $x$. From Proposition \ref{CalibNW}, we have that the curve $x(t) = \pi \circ \phi_L^t(x,v)$ is calibrated by any weak-KAM solution $u$ and
	\begin{equation*}
		 h^{k_n}(x(-k_n),x) =  u(x) - u(-k_n,x(-k_n))  =  u(x) - u(x(-k_n))  \longrightarrow 0 \quad \text{as } n \to \infty
	\end{equation*}
	where we use the continuity of $u$. Thus, by the liminf Property \eqref{PeierlsLiminf}, and by the non-negativity Property \ref{PeierlsPositivity} of Proposition \ref{PeierlsProp}, we obtain
	\begin{equation*}
		0 \leq h^\infty(x,x) \leq h^{\underline{k}}(x,x) = \liminf_n h^{k_n}(x(-k_n),x) = 0
	\end{equation*}
	We proved that $h^\infty(x,x)=0$ on $\mathcal{M}^R_0$. Additionally, it was stated in Proposition \ref{MatherNonempty} that this set is dense in the Mather set $\mathcal{M}_0$. Therefore, we deduce by continuity that $h^\infty(x,x)=0$ on $\mathcal{M}_0$.
\end{proof}

\subsection{Representation Formula on $\mathcal{M}_0^R$}

In this subsection, we establish a non-canonical representation formula. To do this, we first introduce a non-canonical barrier $\underline{h}$, defined on the recurrent Mather set $\mathcal{M}_0^R$, utilizing various $\underline{p}$-Peierls barriers associated with different sequences.

\subsubsection{The Barrier $\underline{h}$} \label{Sectionhbar}

We start by introducing a new barrier $\underline{h}$ based on the $\underline{p}$-Peierls barriers. For that purpose, we consider for every point $x$ of the recurrent Mather set $\mathcal{M}_0^R$ with lift $\tilde{x}$ in $\tilde{\mathcal{M}}_0^R$ an increasing sequence $\underline{p}^x = (p^x_n)_{n \geq 0}$ of $\mathbb{N}$ such that $\tilde{x}$ is $-\underline{p}^x$-recurrent by the Lagrangian flow $\phi_L$.

\begin{defi}
	\begin{enumerate}
		\item We define the barrier $\underline{h} : \mathcal{M}_0^R \times M \to \mathbb{R}$ by
		\begin{equation}
			\underline{h}(x,y) = h^{\underline{p}^x}(x,y) = \liminf_n h^{p^x_n}(x,y)
		\end{equation}		 
		with the corresponding time-depending barrier
		\begin{equation}
			\underline{h}^t(x,y)= \underline{h}(t,x,y) = h^{\underline{p}^x+t}(x,y)
		\end{equation}
	\end{enumerate}
\end{defi}

\begin{rem}
	\begin{enumerate}
		\item Note that, in general, this barrier $\underline{h}(x,y)$ is not continuous on $x$.
		\item The definition depends heavily on the choice of the sequences $\underline{p}^x$, which is why $\underline{h}$ is not canonical. This issue will be addressed in the next subsection about representation formulas on $\mathcal{M}_0$.
	\end{enumerate}
\end{rem}

\begin{prop} \label{hpProp}
	\begin{enumerate}
		\item \label{hpVisc} For all $x \in \mathcal{M}_0^R$, the map $\underline{h}(\cdot,x,\cdot)$ is a viscosity solution of the Hamilton-Jacobi equation (\ref{HJalpha}).
		\item \label{hpRegularity} For all $x \in \mathcal{M}_0^R$, the map $\underline{h}(\cdot, x ,\cdot) : \mathbb{R} \times M\to \mathbb{R}$ is $\kappa_\varepsilon$-Lipschitz for all $\varepsilon >0$ with $\kappa_\varepsilon$ being the Lipschitz constant introduced in Proposition \ref{Regularity}.
		\item \label{hpnullMather} For all $x$ in $\mathcal{M}_0^R$, we have $\underline{h}(x,x) = h^{\underline{p}^x}(x,x) = 0$.
		\item \label{hpRecurrence} For all $x \in \mathcal{M}^R_0$, the viscosity solution $\underline{h}(\cdot,x,\cdot)$ is $\underline{p}^x$-recurrent and 
		\begin{equation} \label{hpRecurrenceFormula}
			\Vert \underline{h}^{p^x_n}(x,\cdot) - \underline{h}(x,\cdot) \Vert_\infty \leq 2\kappa_1. d(x(-p^x_n),x)
		\end{equation}
	\end{enumerate}
\end{prop}
\begin{proof}
	\ref{hpVisc} and \ref{hpRegularity}. Since $\underline{h}(x,\cdot) = h^{\underline{p}^x}(x,\cdot) = \liminf_n h^{p^x_n}(x,\cdot)$, these properties follow immediately from Proposition \ref{pPeierlsProp}.\\

	\ref{hpnullMather}. Consider the point $\tilde{x} \in \tilde{\mathcal{M}}_0$ that projects to $x$ on $M$ and set $x(t)$ to be the curve $x(t) = \pi \circ \phi_L^t(\tilde{x})$. We have by definition of the sequence $p_n^x$ that $x(-p_n^x)$ converges to $x$. Then, we get
	\begin{align*}
		h^{\underline{p}^x}(x,x) = \liminf_n h^{p_n^x}(x,x) = \liminf_n h^{p_n^x}(x(-p_n^x),x)
	\end{align*}
	where we used the liminf property \eqref{PeierlsLiminf} of the Peierls barrier and the fact that $x(t)$ is minimizing seen in Remark \ref{MatherMinimizing}. Let $u$ be any weak-KAM solution. We know from proposition \ref{CalibNW} that $x(t)$ is calibrated by $u$, then by continuity and $1$-time periodicity of $u$, we get
	\begin{align*}
		h^{p_n^x}(x(-p_n^x),x) =  u(x) - u(-p_n^x,x(-p_n^x)) =  u(x) - u(x(-p_n^x)) \longrightarrow 0 \quad \text{as } n \to \infty
	\end{align*}
	Therefore, we deduce that 
	\begin{align*}
		h^{\underline{p}^x}(x,x) = \lim_n h^{p_n^x}(x(-p_n^x),x) = 0
	\end{align*}
	
	\ref{hpRecurrence}. The fact that $h^{\underline{p}}(y, \cdot)$ is recurrent follows from the fact that it is a global viscosity solution. This is due to Proposition \ref{hfinite} and the recurrence follows from Theorem \ref{Glob=NW}. To demonstrate $\underline{p}$-recurrence, we will need to prove the identity \eqref{hpRecurrenceFormula}. 
	
	Let $y$ be a fixed point of $\mathcal{M}_0^R$. Fix an integers $n$ and let $(n_i = n_i(n))_{i \geq 0}$ be an increasing sequence of integers depending on $n$ such that $h^{\underline{p}^y+p^y_n}(y,x) = \lim_i h^{p^y_{n_i}+p^y_n}(y,x)$. Now fix an integer $i \geq 0$ and let $k_0=k_0(n,i)$ be an integer such that for all $k \geq k_0$, we get $p^y_k > p^y_{n_i} + p^y_n$. For such $k \geq k_0$, we have
	\begin{equation}\label{hRecIneg2}
		h^{p^y_k}(y,x) \leq h^{p^y_k - p^y_{n_i} -p^y_n}(y,y) + h^{p^y_{n_i} +p^y_n}(y,x)
	\end{equation} 
	Let $\tilde{y}$ be the lift of  point of $y$ in the Mather set $\tilde{\mathcal{M}_0}$ and consider the curve $y(t) = \pi \circ \phi_L^t( \tilde{y})$. We know from the regularity Proposition \ref{Regularity} on the potential $h$ that for all integer $q \geq 1$
	\begin{align*}
		| h^q (y(-p^y_k), y(-p^y_{n_i}-p^y_n)) - h^q(y,y)| \leq \kappa_1. [d(y(-p^y_k),y) + d(y(-p^y_{n_i}-p^y_n),y)] 
	\end{align*}
	So that for $q = p^y_k - p^y_{n_i}-p^y_n$, we get	
	\begin{align} \label{hRecIneg3}
		| h^{p^y_k - p^y_{n_i}-p^y_n} ((y(-p^y_k), y(-p^y_{n_i}-p^y_n)) - h^{p^y_k - p^y_{n_i}-p^y_n}(y,y)| \leq \kappa_1. [d(y(-p^y_k),y) + d(y(-p^y_{n_i}-p^y_n),y)] 
	\end{align}

	Let $v$ be a weak-KAM solution. We know from Proposition \ref{CalibNW} that $y(t)$ is calibrated by $v$ and from Proposition \ref{Regularity} that $v= \mathcal{T}^1v$ is $\kappa_1$-Lipschitz, hence
	\begin{equation} \label{hRecIneg3.5}
		|h^{p^y_k - p^y_{n_i}-p^y_n} (y(-p^y_k), y(-p^y_{n_i}-p^y_n))| = |v(y(-p^y_{n_i}-p^y_n)) - v(y(-p^y_k)) | \leq \kappa_1.d( y(-p^y_{n_i}-p^y_n), y(-p^y_k))
	\end{equation}
	Hence, we deduce from (\ref{hRecIneg3}) that
	\begin{align*}
		h^{p^y_k - p^y_{n_i}-p^y_n}(y,y) &\leq h^{p^y_k - p^y_{n_i}-p^y_n} ((y(-p^y_k), y(-p^y_{n_i}-p^y_n))  + \kappa_1. [ d(y(-p^y_k),y) + d(y(-p^y_{n_i}-p^y_n),y)]  \\
		& \leq \kappa_1.[ d( y(-p^y_{n_i}-p^y_n), y(-p^y_k)) +   d(y(-p^y_k),y) + d(y(-p^y_{n_i}-p^y_n),y)]
	\end{align*}
	and from (\ref{hRecIneg2}) that
	\begin{align*}
		h^{p^y_k}(y,x) \leq \kappa_1.[ d( y(-p^y_{n_i}-p^y_n), y(-p^y_k)) +   d(y(-p^y_k),y) + d(y(-p^y_{n_i}-p^y_n),y)] + h^{p^y_{n_i} +p^y_n}(y,x)\\
	\end{align*}
	
	Let $k_i$ be an increasing sequence such that $k_i \geq k_0(n,i)$, i.e $p^y_{k_i} > p^y_{n_i} + p^y_n$, we obtain
	\begin{align*}
		\lim_{i} d( y(-p^y_{n_i}-p^y_n), y(-p^y_{k_i})) = d(y(-p^y_n),y)
	\end{align*}
	Then, taking the liminf on $i$ yields
	\begin{equation} \label{hRecIneg4}
		\begin{split}
			h^{\underline{p}^y}(y,x) &\leq \kappa_1.\lim_{i} \; [ d( y(-p^y_{n_i}-p^y_n), y(-p^y_{k_i})) +   d(y(-p^y_{k_i}),y) + d(y(-p^y_{n_i}-p^y_n),y)] + \lim_i h^{p^y_{n_i} +p^y_n}(y,x) \\
			& = 2\kappa_1.d(y(-p^y_n),y) + h^{\underline{p}^y + p^y_n}(y,x)
		\end{split}
	\end{equation}
	
	For the inverse inequality. Fix two integers $n$ and $k \geq 0$. The triangular inequality \eqref{TriangIneg} gives
	\begin{equation*}
		h^{p^y_n+ p^y_k}(y,x) \leq h^{p^y_n}(y,y) + h^{p^y_k}(y,x)
	\end{equation*}
	Moreover, we know from the Lipschitz regularity of the potential $h$ that
	\begin{equation*}
		|h^{p^y_n}(y,y) - h^{p^y_n}(y(-p^y_n),y)| \leq \kappa_1.d(y(-p^y_n),y)
	\end{equation*}
	Bounding $h^{p^y_n}(y(-p^y_n),y)$ as in (\ref{hRecIneg3.5}), we get
	\begin{align*}
		h^{p^y_n+ p^y_k}(y,x) & \leq h^{p^y_n}(y(-p^y_n),y) + \kappa_1.d(y(-p^y_n),y) + h^{p^y_k}(y,x) \\
		& \leq 2\kappa_1.d(y(-p^y_n),y) + h^{p^y_k}(y,x)
	\end{align*}
	Then, we Take the liminf on $k$ to obtain the desired inequality 
	\begin{equation}\label{hRecIneg1}
		h^{\underline{p}^y+ p^y_n}(y,x) \leq 2\kappa_1.d(y(-p^y_n),y)  + h^{\underline{p}^y}(y,x)
	\end{equation}
	
	Gathering the inequalities (\ref{hRecIneg1}) and (\ref{hRecIneg4}) leads to
	\begin{equation} \label{hRecIneg5}
		\Vert \underline{h}^{p^y_n}(y,\cdot) - \underline{h}(y,\cdot) \Vert_\infty \leq 2\kappa_1.d(y(-p^y_n),y) \longrightarrow 0 \quad \text{as } n\to +\infty
	\end{equation}
	Therefore, $\underline{h}(y,\cdot)$ is $\underline{p}^y$-recurrent.
\end{proof}

\subsubsection{The Representation Formula}

Recall from Definition \ref{DominationDef} the notion of domination. We will work with $\underline{h}$-dominated maps on $\mathcal{M}_0^R$, i.e maps of $\dom(\mathcal{M}_0^R, \underline{h})$.

\begin{theo} \label{RepresentationTheorem}
	We have the following bijection
	\begin{equation} \label{RepresentationBijection}
		\begin{matrix}
			\Psi_{\underline{h}}: \dom(\mathcal{M}_0^R, \underline{h}) & \longrightarrow & \Omega(\mathcal{T}) \\
			\psi & \longmapsto & \inf\limits_{y \in \mathcal{M}_0^R} \{ \psi(y) + \underline{h}(y, \cdot ) \}
		\end{matrix}
	\end{equation}
	with its inverse being the restriction map
	\begin{equation} \label{InverseMaph}
		\begin{matrix}
			\Phi_{\underline{h}} : \Omega(\mathcal{T}) & \longrightarrow & \dom(\mathcal{M}_0^R, \underline{h}) \\
			v & \longmapsto & v_{| \mathcal{M}_0^R}
		\end{matrix}
	\end{equation}       
\end{theo}

\begin{proof}
	Note that the maps of the form $\inf\limits_{y \in \mathcal{M}_0^R} \{ \psi(y) + \underline{h}(y, \cdot ) \}$ are viscosity solutions due to Proposition \ref{ViscosityInf}. They are globally defined so that they belong to $\Omega(\mathcal{T})$ by Theorem \ref{Glob=NW}. This justifies the well-definition of the map $\Psi_{\underline{h}}$. We need to prove that $\Phi_{\underline{h}}$ is well-defined. Let $v$ be an element of $\Omega(\mathcal{T})$. We show that $v_{| \mathcal{M}_0^R}$ is $\underline{h}$-dominated. Let $x$ and $y$ be two elements of $\mathcal{M}_0^R$ and $\tilde{x}$ the lift of $x$ to $\tilde{\mathcal{M}}_0^R$. Consider the curve $x(t) = \pi \circ \phi_L^t(\tilde{x})$. We know from Proposition \ref{CalibNW} that $x(t)$ is calibrated by $v$ and by any weak-KAM solution $u$. Then for all negative time $t$ we have
	\begin{equation*}
		v(x) - v(t,x(t)) = h^{t,0}(x(t),x) = u(x) - u(t,x(t))
	\end{equation*}
	Since $x(-p^x_n)$ converges to $x$, we have for $t= -p^x_n$
	\begin{equation*}
		\lim_n v(x) - v(-p^x_n, x(-p^x_n)) = \lim_n u(x) - u(-p^x_n,x(-p^x_n))  =  \lim_n u(x) - u(x(-p^x_n)) =0
	\end{equation*}
	We now use the definition of the Lax-Oleinik operator $\mathcal{T}$ to deduce that
	\begin{equation*}
		v(y) - v(x) = \lim_n v(0,y) - v(-p^x_n, x(-p^x_n)) \leq \liminf_n h^{p^x_n}(x(-p^x_n),y) = \liminf_n h^{p^x_n}(x,y) = h^{\underline{p}^x}(x,y) = \underline{h}(x,y)
	\end{equation*}
	where we used the liminf property (\ref{PeierlsLiminf}) of the Barrier $h^{\underline{p}^x}$. We obtained the desired $\underline{h}$-domination which justifies the well-definition of the map $\Phi_{\underline{h}}$.  \\
	
	Let us show that the map $\Phi_{\underline{h}}$ is a left inverse of the map $\Psi_{\underline{h}}$. Take $v = \Psi_{\underline{h}}(\psi)= \inf\limits_{y \in \mathcal{M}_0^R} \{ \psi(y) + \underline{h}(y, \cdot ) \} \in \Omega(\mathcal{T})$. Then, for all $x$ and $y$ in $\mathcal{M}_0^R$, the domination condition gives the inequality
	\begin{equation} \label{hInvInégalité2}
		\psi(x) + \underline{h}(x,x) = \psi(x) \leq \psi(y) + \underline{h}(y, x )
	\end{equation}
	where we recall from Property \ref{hpnullMather} of Proposition \ref{hpProp} that $ \underline{h}(x,x) = h^{\underline{p}^x}(x,x) =0$. We obtain for all $x \in \mathcal{M}_0^R$
	\begin{equation*}
		v(x) = \Psi_{\underline{h}}(\psi)(x)= \inf\limits_{y \in \mathcal{M}_0^R} \{ \psi(y) + \underline{h}(y, x ) \} = \psi(x) + \underline{h}(x,x) = \psi(x)
	\end{equation*}
	In other words, $\Phi_{\underline{h}} \circ \Psi_{\underline{h}}(\psi) = \psi$.\\
	
	We now show that $\Phi_{\underline{h}}$ is the right inverse of the map $\Psi_{\underline{h}}$. Let $v \in \Omega(\mathcal{T})$ and consider $w = \Psi_{\underline{h}} \circ \Phi_{\underline{h}} (v) \in \Omega(\mathcal{T})$. We need to prove the these two maps $v$ and $w$ are equal. By the uniqueness Theorem \ref{Uniqueness}, it suffices to prove that they coincide on the Mather set $\mathcal{M}_0$. Let $x$ be an element of $\mathcal{M}_0^R$. We infer from the $\underline{h}$-domination of $v$ on $\mathcal{M}_0^R$ and from \eqref{hInvInégalité2} that
	\begin{align*}
		w(x) = \Psi_{\underline{h}} ( v_{|\mathcal{M}_0^R} )(x) =  \inf_{y \in \mathcal{M}_0^R} \{ v(y) + \underline{h}(y,x) \} = v(x) + \underline{h}(x,x) = v(x)
	\end{align*}
	Thus, $w_{|\mathcal{M}_0^R} = v_{|\mathcal{M}_0^R}$ and by continuity of viscosity solutions and density of $\mathcal{M}_0^R$ in $\mathcal{M}_0$, we deduce that $w_{|\mathcal{M}_0} = v_{|\mathcal{M}_0}$. We have shown that $v$ and $w$ are two elements of $\Omega( \mathcal{T})$ which coincide on the Mather set $\mathcal{M}_0$, then we get from the uniqueness Theorem \ref{Uniqueness} the equality $v= w =  \Psi_{\underline{h}} \circ \Phi_{\underline{h}} (v)$. This concludes the proof of the theorem.
\end{proof}

\begin{rem} \label{RemarkRepresentationTheorem}
	\begin{enumerate}
		\item As stated in the beginning of this subsection, the proof shows that $\mathcal{M}_0^R$ can be replaced by any of its dense subsets.
		
		\item Note that for all $\underline{h}$-dominated map $\psi \in  \dom(\mathcal{M}_0^R, \underline{h})$, we have $\psi = \Phi_{\underline{h}} \circ \Psi_{\underline{h}} (\psi) = \Psi_{\underline{h}} (\psi)_{| \mathcal{M}_0^R}$ where $\Psi_{\underline{h}} (\psi) \in \Omega(\mathcal{T})$ is continuous. This implies that all elements $\psi$ of $\dom(\mathcal{M}_0^R, \underline{h})$ are continuous. The reason of that will be explained by the fact that the domination by the barrier $\underline{h}$ implies the domination by another continuous barrier $\underline{k}$ introduced in the next subsection.
	\end{enumerate}
\end{rem}

Following the last remark, it is possible to restrict the set of points $y$ over which we take the infimum to a subset of $\mathcal{M}_0$, which may not be dense. This reduction is carried out in the next subsection.

\subsection{Representation Formula on $\mathcal{M}_0$}

This subsection introduces a new generalized barrier $\underline{k}$, defined on the entire Mather set $\mathcal{M}_0$. This barrier is independent of the choice of sequences $\underline{p}^x$ and satisfies the triangular inequality. These properties make $\underline{k}$ suitable for the more general representation formula presented in Theorem \ref{kRepresentationTheorem}, which we prove in this subsection.

\subsubsection{The Generalized Peierls Barrier $\underline{k}$}

We begin by introducing a generalized Peierls barrier $\underline{k}$, derived from the barrier $\underline{h}$, by enforcing the triangular inequality. This, as we will see, ensures the continuity of $\underline{k}$ and results in a canonical barrier that is independent of the choice of sequences $\underline{p}^x$ (Corollary \ref{kMaxMinCorollary}).

\begin{defi} \label{kDefi}
	\begin{enumerate}
		\item We define the \textit{generalized Peierls Barrier} $\underline{k} : \mathcal{M}_0^R \times M \to \mathbb{R}$ by
		\begin{equation}
			\underline{k}(x,y) = \inf \left\{ \sum_{i=0}^{N-1} \underline{h}(x_i, x_{i+1}) \; | \; x_0 = x, \; x_N = y, \; x_i \in \mathcal{M}_0^R, \; N \geq 1 \right\}
		\end{equation}
		with time dependence 
		\begin{equation}
			\underline{k}(t,x,y) = \underline{k}^t(x,y) = \inf \left\{ \sum_{i=0}^{N-2} \underline{h}(x_i, x_{i+1}) + \underline{h}^t(x_{N-1}, y) \; | \; x_0 = x, \; x_N = y, \; x_i \in \mathcal{M}_0^R, \; N \geq 1 \right\}
		\end{equation}
		\item We define the map $\underline{d} : \mathcal{M}_0^R \times \mathcal{M}_0^R \to \mathbb{R}$ by
		\begin{equation}
			\underline{d}(x,y) = \underline{k}(x,y) + \underline{k}(y,x)
		\end{equation}
	\end{enumerate}
\end{defi}

\begin{prop} \label{kProp}
	We have
	\begin{enumerate}
		\item (\textit{Viscosity Solution}) \label{kVisc} For all $x \in \mathcal{M}_0^R$, the map $\underline{k}(\cdot,x,\cdot)$ is a viscosity solution of the Hamilton-Jacobi equation \eqref{HJalpha}.
		\item \label{k>p} For all $(t,x,y) \in \mathbb{R} \times \mathcal{M}_0^R \times M$, we have 
		\begin{equation} \label{k>pFormula}
			h^{\infty+t}(x,y) \leq \underline{k}^t(x,y) \leq  \underline{h}^t(x,y)
		\end{equation}
		\item \label{kDiag=0} For all $x \in \mathcal{M}_0^R$, $\underline{k}(x,x) =0$.
		\item (\textit{Triagular Inequality}) \label{kTriang} For all times $t \in \mathbb{R}$ and all points $x,y \in \mathcal{M}_0^R$ and $z \in M$, we have the triangular inequality
		\begin{equation} \label{kTriangularInegFormula}
			\underline{k}^t(x,z) \leq \underline{k}(x,y) + \underline{k}^t(y,z)
		\end{equation}
		\item (\textit{Regularity}) \label{kReg} The barrier $\underline{k}^t$ is $\kappa_\varepsilon$-Lipschitz on $\mathcal{M}_0^R \times M$ for all $\varepsilon >0$ with $\kappa_\varepsilon$ being the Lipschitz constant introduced in Proposition \ref{Regularity}.
	\end{enumerate}
\end{prop}

\begin{proof}
	\ref{kVisc}. Let $x$ be a point of $\mathcal{M}_0^R$. We have
	\begin{equation*}
		\underline{k}^t(x,y) = \inf \left\{ \sum_{i=0}^{N-2} \underline{h}(x_i, x_{i+1}) + \underline{h}^t(x_{N-1}, y) \; | \; x_0 = x, \; x_N = y, \; x_i \in \mathcal{M}_0^R, \; N \geq 1 \right\}
	\end{equation*}
	By Proposition \ref{hpProp}, this infimum is taken over viscosity solutions. Hence, we infer from Proposition \ref{ViscosityInf} that $\underline{k}(\cdot,x,\cdot)$ is also a viscosity solution. Moreover, following the regularity Property \ref{hpRegularity} of Proposition \ref{hpProp}, we deduce that this solution is $\kappa_\varepsilon$-Lipschitz for all $\varepsilon >0$.\\
	
	\ref{k>p}. The second inequality follows immediately from the definition of $\underline{k}$. We prove the first inequality. Let $(x,y)$ be an element of $\mathcal{M}_0^R \times M$ and let $(x_i)_{0 \leq x_i \leq N-1}$ be a sequence of elements of $\mathcal{M}_0^R$ with $x_0 = x$ and set $x_N = y$. we have
	\begin{align*}
		\sum_{i=0}^{N-2} \underline{h}(x_i, x_{i+1}) + \underline{h}^{t}(x_{N-1},y) &= \sum_{i=0}^{N-2} h^{\underline{p}^{x_i}}(x_i, x_{i+1}) + h^{\underline{p}^{x_{N-1}}+t}(x_{N-1},y) \\
		&= \sum_{i=0}^{N-2} \liminf_{n_i} h^{p_{n_i}^{x_i}}(x_i, x_{i+1}) + \liminf_{n_{N-1}} h^{p_{n_{N-1}}^{x_{N-1}}+t}(x_{N-1}, y) \\
		&= \liminf_{n_1,..,n_{N-1}} \sum_{i=0}^{N-2} h^{p_{n_i}^{x_i}}(x_i, x_{i+1}) + h^{p_{n_{N-1}}^{x_{N-1}}+t}(x_{N-1}, y)\\
		&\geq \liminf_{n_1,..,n_{N-1}} h^{\sum_{i=0}^{N-1} p_{n_i}^{x_i}+t} (x,y) \\
		&\geq \liminf_n h^{n+t}(x,y) = h^{\infty+t}(x,y)
	\end{align*}
	Taking the infimum on such sequences, we deduce the inequality $ \underline{k} \geq h^\infty$.\\
	
	\ref{kDiag=0}. Let $x$ be a point of $\mathcal{M}_0^R$. By the previous property, we have
	\begin{equation*}
		h^\infty(x,x) \leq \underline{k}(x,x) \leq \underline{h}(x,x)
	\end{equation*}
	Moreover, we know from Proposition \ref{MatherPeierls} and Property \ref{hpnullMather} of Proposition \ref{hpProp} that $h^\infty(x,x) = \underline{h}(x,x) =0$. Hence, $\underline{k}(x,x) = 0$.\\
	
	\ref{kTriang}. Let $x$ and $y$ be points of $\mathcal{M}_0^R$ and let $z$ be a point of $M$. Consider two sequences of points $(x_i)_{0 \leq i \leq N}$ with $(x_0,x_N) = (x,y)$ and $(y_j)_{0 \leq j \leq N'}$ with $(y_0,y_N') = (y,z)$. Note that $x_N = y_0 = y$. Hence, concatenating them into a third sequence $(z_i)_{0 \leq k \leq N+N'}$ and using the definition of the barrier $\underline{k}$, we obtain
	\begin{align*}
		\underline{k}(x,z) \leq \sum_{i=0}^{N+N'-2} \underline{h}(z_i, z_{i+1}) + \underline{h}^t(z_{N+N'-1}, z_{N+N'}) = \sum_{i=0}^{N-1} \underline{h}(x_i, x_{i+1}) + \sum_{j=0}^{N'-2} \underline{h}(y_j, y_{j+1}) + \underline{h}^t(y_{N'-1}, y_{N'})
	\end{align*}
	Taking the infimum over such sequences yields the desired triangular inequality.\\
	
	\ref{kReg}. Let $(t,x,y)$ and $(t',x',y')$ be two elements of $\mathbb{R} \times \mathcal{M}_0^R \times M$. We have
	\begin{equation} \label{kRegDem1}
		|\underline{k}^t(x,y) - \underline{k}^{t'}(x',y')| \leq |\underline{k}^t(x,y) - \underline{k}^t(x',y)| + |\underline{k}^t(x',y) - \underline{k}^{t'}(x',y')|
	\end{equation}
	We have seen in the proof of the first point that 
	\begin{equation} \label{kRegDem2}
		|\underline{k}^t(x',y) - \underline{k}^{t'}(x',y')| \leq \kappa_\varepsilon.[d(y,y') + |t-t'|]
	\end{equation}
	We need to bound the other term of the rand-hand side of \eqref{kRegDem1}. We know from the triangular inequality that
	\begin{equation}
		\underline{k}^t(x,y) \leq  \underline{k}(x,x')  +  \underline{k}^t(x',y) \quad \text{and} \quad \underline{k}^t(x',y) \leq \underline{k}(x',x) + \underline{k}^t(x,y)
	\end{equation}
	Thus, we get
	\begin{align*}
		|\underline{k}^t(x,y) - \underline{k}^t(x',y)| \leq \max \big\{ |\underline{k}(x,x')|, |\underline{k}(x',x)| \big\}
	\end{align*}
	Moreover, we proved that $\underline{k}(x,x) = \underline{k}(x',x') = 0$, which yields
	\begin{align*}
		|\underline{k}(x,x')| = |\underline{k}(x,x') - \underline{k}(x,x) | \leq \kappa_\varepsilon.d(x,x')
	\end{align*}
	and similarly $|\underline{k}(x',x)| \leq \kappa_\varepsilon.d(x,x')$. This leads to the bounding
	\begin{equation} \label{kRegDem3}
		|\underline{k}^t(x,y) - \underline{k}^t(x',y)| \leq \kappa_\varepsilon.d(x,x')
	\end{equation}
	Gathering \eqref{kRegDem1},\eqref{kRegDem2} and \eqref{kRegDem3}, we conclude that
	\begin{align*}
		|\underline{k}^t(x,y) - \underline{k}^{t'}(x',y')| \leq \kappa_\varepsilon.[ d(x,x') + d(y,y') + |t-t'|]
	\end{align*}
	which is the desired Lipschitz inequality.
\end{proof}

\begin{cor}
	For all time $t \in \mathbb{R}$, the barrier $\underline{k}^t$ extends in a unique way to the set $\mathcal{M}_0 \times M$. The extended barrier $\underline{k}: \mathbb{R} \times \mathcal{M}_0 \times M \to \mathbb{R}$ possesses all the properties featured in Proposition \ref{kProp}.
\end{cor}

\begin{proof}
	The Property \ref{kReg} of Proposition \ref{kProp} implies that the map $\underline{k}^t$ is uniformly continuous on the dense subset $\mathcal{M}_0^R \times M$ of the compact set $\mathcal{M}_0 \times M$. Hence, it extends uniquely to $\mathcal{M}_0 \times M$. 
	
	All the properties extensions are straightforward except for the viscosity solutions. We prove that for all $x \in \mathcal{M}_0$, the map $\underline{k}(\cdot,x,\cdot)$ is a viscosity solution of the Hamilton-Jacobi equation \eqref{HJ}. Let $x$ be in $\mathcal{M}_0$ and let $x_n$ be a sequence of $\mathcal{M}_0^R$ that converges to $x$. For all times $s < t$, the non-expensiveness of the Lax-Oleinik semi-group stated in Proposition \ref{Contracting} leads to
	\begin{align*}
		\Vert \mathcal{T}^{s,t} \underline{k}^s(x_n, \cdot) - \mathcal{T}^{s,t} \underline{k}^s(x, \cdot) \Vert_\infty \leq \Vert \underline{k}^s(x_n, \cdot) - \underline{k}^s(x, \cdot) \Vert_\infty \leq \kappa_1.d(x_n,x) \longrightarrow 0 \quad \text{as } n\to +\infty
	\end{align*}
	Then $\lim_n \mathcal{T}^{s,t} \underline{k}^s(x_n, \cdot) = \mathcal{T}^{s,t} \underline{k}^s(x, \cdot)$. 
	
	Moreover, since $\underline{k}(\cdot, x_n, \cdot)$ is a viscosity solution, we have $\mathcal{T}^{s,t} \underline{k}^s(x_n, \cdot) = \underline{k}^t(x_n, \cdot)$ with
	\begin{align*}
		\Vert \underline{k}^t(x_n, \cdot) - \underline{k}^t(x, \cdot) \Vert_\infty \leq \kappa_1.d(x_n,x) \longrightarrow 0 \quad \text{as } n\to +\infty
	\end{align*}
	Therefore
	\begin{align*}
		\mathcal{T}^{s,t} \underline{k}^s(x, \cdot) = \mathcal{T}^{s,t} \underline{k}^s(x_n, \cdot) = \lim_n \underline{k}^t(x_n, \cdot) = \underline{k}^t(x, \cdot)
	\end{align*}
	and $\underline{k}(\cdot, x, \cdot)$ is a viscosity solution.
\end{proof}

\subsubsection{The Generalized Representation Formula}

Before diving into the proof of the main result of this paper, let us give some representation formulas that follow directly from Theorem \ref{RepresentationTheorem}.

\begin{prop} \label{PropDominationk}
	\begin{enumerate}
		\item We have equality $\dom(\mathcal{M}_0^R, \underline{k}) = \dom(\mathcal{M}_0^R, \underline{h})$.
		\item The bijection $\Psi_{\underline{h}}$ expressed in \eqref{RepresentationBijection} is equal to
		\begin{equation} \label{RepresentationBijectionBis}
			\begin{matrix}
				\dom(\mathcal{M}_0^R, \underline{k}) & \longrightarrow & \Omega(\mathcal{T}) \\
				\psi & \longmapsto & \inf\limits_{y \in \mathcal{M}_0^R} \{ \psi(y) + \underline{k}(y, \cdot ) \}
			\end{matrix}
		\end{equation}
		\item Extending this formula by continuity to the Mather set $\mathcal{M}_0$, we get the bijection
		\begin{equation} \label{RepresentationBijectionTris}
			\begin{matrix}
				\dom(\mathcal{M}_0, \underline{k}) & \longrightarrow & \Omega(\mathcal{T}) \\
				\psi & \longmapsto & \inf\limits_{y \in \mathcal{M}_0} \{ \psi(y) + \underline{k}(y, \cdot ) \}
			\end{matrix}
		\end{equation}
	\end{enumerate}
\end{prop}

\begin{rem} \label{RemarkDominationk}
		We can note that the equality $\dom(\mathcal{M}_0^R, \underline{k}) = \dom(\mathcal{M}_0^R, \underline{h})$ explains the continuity of $\underline{h}$-dominated maps deduced from Theorem \ref{RepresentationTheorem}.
\end{rem}

\begin{proof}
	\textit{Domination.} The inclusion $\dom(\mathcal{M}_0^R, \underline{k}) \subset  \dom(\mathcal{M}_0^R, \underline{h})$ comes from the inequality $\underline{k} \leq \underline{h}$. Let us prove the inverse inclusion. Fix a map $\psi \in \dom(\mathcal{M}_0^R, \underline{h})$. For all sequence $(x_i)_{0\leq i \leq N}$ in $\mathcal{M}_0^R$, we have
	\begin{equation*}
		\psi(x_{i+1}) - \psi(x_i) \leq \underline{h}(x_i,x_{i+1})
	\end{equation*}
	Thus, summing on $i$ yields
	\begin{equation*}
		\psi(x_N) - \psi(x_0) \leq \sum_{i=0}^{N-1} \underline{h}(x_i,x_{i+1})
	\end{equation*}
	Then, taking the infimum on such sequences linking $x_0 = x$ to $x_N = y$ in $\mathcal{M}_0^R$, we obtain the domination inequality
	\begin{equation*}
		\psi(y) - \psi(x) \leq \underline{k}(x,y)
	\end{equation*}
	which shows that $\psi$ belongs to $\dom(\mathcal{M}_0^R, \underline{k})$.\\
	
	\textit{Formula.} Denote by $\widetilde{\Psi}_{\underline{h}}$ the map expressed in \eqref{RepresentationBijectionBis}. Let $\psi$ be an element of $\dom(\mathcal{M}_0^R, \underline{k}) = \dom(\mathcal{M}_0^R, \underline{h})$ and consider $v = \Psi_{\underline{h}}(\psi)$ and $\tilde{v}= \widetilde{\Psi}_{\underline{h}}(\psi)$. As in the proof of Theorem \ref{RepresentationTheorem}, we have $v_{|\mathcal{M}_0^R} = \tilde{v}_{|\mathcal{M}_0^R} = \psi$. Hence, by density of $\mathcal{M}_0^R$ in the Mather set $\mathcal{M}_0$ and by continuity of the maps $v$ and $\tilde{v}$, we deduce that $v_{|\mathcal{M}_0} = \tilde{v}_{|\mathcal{M}_0}$. Since $v$ and $\tilde{v}$ belong to the non-wandering set $\Omega(\mathcal{T})$, we conclude using the Uniqueness Theorem \ref{Uniqueness} that $\Psi_{\underline{h}}(\psi) = v = \tilde{v}= \widetilde{\Psi}_{\underline{h}}(\psi)$.
\end{proof}

We now establish the representation formula for the non-wandering set $\Omega(\mathcal{T})$ using the generalized Peierls barrier $\underline{k}$. To reduce the set over which the infimum is taken, we introduce an equivalence relation defined by the map $\underline{d}$ which is a pseudometric due to the following proposition.

\begin{prop} \label{kPseudometric}
	The map $\underline{d}$ is a pseudometric on $\mathcal{M}_0$.
\end{prop}

\begin{proof}
	The symmetry $\underline{d}(x,y) = \underline{d}(y,x)$ is immediate. Moreover, we know from the Property \ref{kDiag=0} of Proposition \ref{kProp} that for all $x \in \mathcal{M}_0$, $\underline{d}(x,x) = 2.\underline{k}(x,x) = 0$.
	
	\textit{Non-Negativity.} For any elements $x$ and $y$ of $\mathcal{M}_0$, the Triangular inequality \eqref{kTriangularInegFormula} yields
	\begin{align*}
		\underline{d}(x,y)&= \underline{k}(x,y) + \underline{k}(y,x) \geq \underline{k}(x,x)= 0
	\end{align*}
	
	\textit{Triangular Inequality.} Fix three elements $x$, $y$ and $z$ in $\mathcal{M}_0$. Applying the triangular inequality \eqref{kTriangularInegFormula} twice, we get
	\begin{align*}
		\underline{d}(x,z) &= \underline{k}(x,z) + \underline{k}(z,x) \\
		&\leq \underline{k}(x,y) + \underline{k}(y,z) + \underline{k}(z,y) + \underline{k}(y,x) \\
		&\leq \underline{d}(x,y) + \underline{d}(y,z) 
	\end{align*}
\end{proof}

\begin{defi} \label{kStaticClassesDefi}
	\begin{enumerate}
		\item We define the equivalence relation $\sim$ on $\mathcal{M}_0$ by
		\begin{equation}
			x \sim y \Longleftrightarrow \underline{d}(x,y) = 0
		\end{equation}

		\item The \textit{generalized static classes} are the equivalence classes of the equivalence relation $\sim$. We denote by $\underline{\mathbb{M}}$ the set of generalized static classes. We represent every element of $\underline{\mathbb{M}}$ by an element of $\mathcal{M}_0$ so that we have the inclusion $\underline{\mathbb{M}} \subset \mathcal{M}_0$.
	\end{enumerate}
\end{defi}

Recall from Definition \ref{DominationDef} the notion of domination. We will work with the set $\dom(\underline{\mathbb{M}}, \underline{k})$ of $\underline{k}$-dominated maps on the set of generalized static classes $\underline{\mathbb{M}}$. We can now prove the main result of this paper.

\begin{proof}[Proof of Theorem \ref{kRepresentationTheorem}]
	We saw in Proposition \ref{PropDominationk} that $\dom(\mathcal{M}_0^R, \underline{k}) = \dom(\mathcal{M}_0^R, \underline{h})$. And we know from Theorem \ref{RepresentationTheorem} that all the elements of $\Omega(\mathcal{T})$ are $\underline{h}$-dominated on $\mathcal{M}^R_0$. Hence, they are $\underline{k}$-dominated on $\mathcal{M}_0^R$ and by continuity on the Mather set $\mathcal{M}_0$. In particular, $\underline{k}$-domination holds in $\underline{\mathbb{M}} \subset \mathcal{M}_0$. Following the proof of Theorem \ref{RepresentationTheorem}, this allows to prove that the maps $\Psi_{\underline{k}}$ and $\Phi_{\underline{k}}$ are well-defined and that $\Phi_{\underline{k}} \circ \Psi_{\underline{k}} = Id_{\dom(\underline{\mathbb{M}}, \underline{k})}$. However, the identity $\Psi_{\underline{k}} \circ \Phi_{\underline{k}} = Id_{\Omega(\mathcal{T})}$ requires to show that the Uniqueness Theorem \ref{Uniqueness} holds on the set of generalized static classes $\underline{\mathbb{M}}$.
	
	Let us show this. Let $v \in \Omega(\mathcal{T})$ and consider $w = \Psi_{\underline{k}} \circ \Phi_{\underline{k}} (v) \in \Omega(\mathcal{T})$. We need to prove that $v$ and $w$ are equal. By the uniqueness Theorem \ref{Uniqueness}, it suffices to prove that they coincide on the Mather set $\mathcal{M}_0$. Let $x$ be an element of $\underline{\mathbb{M}} \subset \mathcal{M}_0$. Then due to the $\underline{k}$-domination of $v_{| \underline{\mathbb{M}}}$, and similarly to \eqref{hInvInégalité2}, we have
	\begin{align*}
		w(x) = \Psi_{\underline{k}} \circ \Phi_{\underline{k}} (v)(x) = \inf_{y \in \underline{\mathbb{M}}} \{ v(y) + \underline{k}(y,x) \} = v(x) + \underline{k}(x,x) = v(x)
	\end{align*}
	Let $x$ be a any element of $\mathcal{M}_0$ and $y \in \underline{\mathbb{M}}$ such that $x \sim y$. Then, due to the $\underline{k}$-domination of $v_{| \mathcal{M}_0}$, we have the inequalities
	\begin{align*}
		v(x) - v(y) \leq \underline{k}(y,x) \quad \text{and} \quad  v(y) - v(x) \leq \underline{k}(x,y)
	\end{align*}
	Taking the sum, we obtain
	\begin{align*}
		0 = [v(x) - v(y)] + [v(y) - v(x)] \leq \underline{k}(y,x) + \underline{k}(x,y) = \underline{d}(y,x) = 0
	\end{align*}
	which implies equality in the two inequalities and
	\begin{align*}
		v(x) - v(y) = \underline{k}(y,x) \quad \text{and} \quad  v(y) - v(x) = \underline{k}(x,y)
	\end{align*}
	Moreover, for all $z \in \underline{\mathbb{M}}$, the $\underline{k}$-domination writes
	\begin{align*}
		v(x) \leq v(z) + \underline{k}(z,x)
	\end{align*}
	Therefore, we obtain
	\begin{equation} \label{kRepDem1}
		w(x) = \inf_{z \in \underline{\mathbb{M}}} \{ v(z) + \underline{k}(z,x) \} = v(y) + \underline{k}(y,x) = v(x)
	\end{equation}
	The maps $v$ and $w$ are two elements of $\Omega( \mathcal{T})$ which coincide on the Mather set $\mathcal{M}_0$. Then the Uniqueness Theorem \ref{Uniqueness} results in the equality $v= w =  \Psi_{\underline{k}} \circ \Phi_{\underline{k}} (v)$.
\end{proof}

\begin{rem}
	It was noted in Remark \ref{RemarkRepresentationTheorem} that $\mathcal{M}^R_0$ can be replaced by any dense subset $\mathcal{M}'_0$. It is still possible to consider the generalized static classes $\underline{\mathbb{M}}'$ inside $\mathcal{M}'_0$. Taking countable dense sets allows to avoid the use of the axiom of choice. 
\end{rem}

We conclude this section by providing an application of this representation formula which shows that the generalized Peierls barrier $\underline{k}$ does not depend on the choice of the sequences $\underline{p}^x$ used to define the barrier $\underline{h}$.

\begin{cor} \label{kMaxMinCorollary}
	For all points $x_0 \in \mathcal{M}_0$ and $x \in M$, we have the formulas
	\begin{equation} \label{kMaxMinOmega}
		\begin{split}
			\max_{v \in \Omega(\mathcal{T})} \{v(x) - v(x_0) \} &= \underline{k}(x_0,x)  \\
			\min_{v \in \Omega(\mathcal{T})} \{v(x) - v(x_0) \} &= \inf_{\substack{y \in \underline{\mathbb{M}} \\ y \nsim x_0}} \{ - \underline{k}(y,x_0) + \underline{k}(y,x) \}
		\end{split}
	\end{equation}
\end{cor}

\begin{proof}
	Let us prove the first identity. We will use the representation formula \eqref{RepresentationBijectionTris}. We consider $\psi^+ : \mathcal{M}_0 \to \mathbb{R}$ defined by $\psi^+(x) = \underline{k}(x_0,x)$. We know from the Property \ref{kDiag=0} of Proposition \ref{kProp} that $\psi^+(x_0) = \underline{k}(x_0,x_0)=0$, and from the triangular inequality \eqref{kTriangularInegFormula} that for all $y \in \mathcal{M}_0$,
	\begin{align*}
		\underline{k}(x_0,x) \leq \underline{k}(x_0,y) + \underline{k}(y,x)
	\end{align*}
	which yields
	\begin{align*}
		\psi^+(x) - \psi^+(y) = \underline{k}(x_0,x) - \underline{k}(x_0,y) \leq  \underline{k}(x,y)
	\end{align*}
	Then, $\psi^+$ belongs to $\dom (\mathcal{M}_0, \underline{k})$. We consider the map $v^+ = \Psi_{\underline{k}}(\psi^+) \in \Omega(\mathcal{T})$. We have $v^+(x_0) = \psi^+(x_0) = 0$. Hence, the triangular inequality results in
	\begin{equation*}
		v^+(x) = \inf_{y \in \mathcal{M}_0} \{ \underline{k}(x_0,y) + \underline{k}(y,x) \} = \underline{k}(x_0,x)
	\end{equation*}
	We showed that $v^+ = \Psi_{\underline{k}}(\psi^+) = \underline{k}(x_0, \cdot)$.

	Additionally, for all $\psi \in \dom (\mathcal{M}_0, \underline{k})$ such that $\psi(x_0) = 0$, the domination condition yields for all $x \in \mathcal{M}_0$
	\begin{align*}
		\psi(x) \leq \inf_{y \in \mathcal{M}_0} \{ \psi(y) + \underline{k}(y,x) \} \leq \psi(x_0) + \underline{k}(x_0,x) = \underline{k}(x_0,x) = \psi^+(x)
	\end{align*}
	Hence, for all $v = \Psi_{\underline{k}}(\psi)$, we have
	\begin{align*}
		v(x) = \inf_{y \in \mathcal{M}_0} \{ \psi(y) + \underline{k}(y,x) \} \leq \inf_{y \in \mathcal{M}_0} \{ \psi^+(y) + \underline{k}(y,x) \} = v^+(x) = \underline{k}(x_0,x)
	\end{align*}
	Since all elements $v \in \Omega(\mathcal{T})$ with $v(x_0) = 0$ are of the form $\Psi_{\underline{k}}(\psi)$ with $\psi \in \dom (\mathcal{M}_0, \underline{k})$ and $\psi(x_0) = 0$, we deduce the first identity of \eqref{kMaxMinOmega}.\\
	
	We now prove the second formula. For this formula, we use the representation formula of Theorem \ref{kRepresentationTheorem}. Consider $\psi^- : \underline{\mathbb{M}} \to \mathbb{R}$ defined by $\psi^-(x) = -\underline{k}(x,x_0)$. We have by triangular inequality that
	\begin{equation*}
		\psi^-(x) - \psi^-(y) = \underline{k}(y,x_0) - \underline{k}(x,x_0) \leq \underline{k}(y,x)
	\end{equation*}
	Thus, $\psi^-$ belongs to $\dom (\underline{\mathbb{M}}, \underline{k})$. We consider the map $v^- = \Psi_{\underline{k}}(\psi^-) \in \Omega(\mathcal{T})$ and let $x_1$ be in $\underline{\mathbb{M}}$ such that $x_0 \sim x_1$. As proved by identity \eqref{kRepDem1}, we have
	\begin{align*}
		v^-(x_0) = \psi^-(x_1) + \underline{k}(x_1,x_0) =-\underline{k}(x_1,x_0) + \underline{k}(x_1,x_0) = 0
	\end{align*}

	We claim that for all $x \in M$,
	\begin{equation*}
		\underline{k}(x_1,x) = \underline{k}(x_1,x_0) + \underline{k}(x_0,x)
	\end{equation*}
	Indeed, we apply the triangular inequality \eqref{kTriangularInegFormula} twice to get
	\begin{align*}
		\underline{k}(x_1,x) &\leq \underline{k}(x_1,x_0) + \underline{k}(x_0,x) \\
		&\leq \underline{k}(x_1,x_0) + \underline{k}(x_0,x_1) + \underline{k}(x_1,x) \\
		&= \underline{d}(x_0,x_1) + \underline{k}(x_1,x) = \underline{k}(x_1,x)
	\end{align*}
	where we used the definition $\underline{d}(x_0,x_1)=0$ of $x_0 \sim x_1$. Hence, we deduce the equality everywhere and the claim is proved. 
	
	Consequently, we get for all $y \in \underline{\mathbb{M}}$
	\begin{align*}
		\psi^-(y) + \underline{k}(y,x) = -\underline{k}(y,x_0)+ \underline{k}(y,x) \leq \underline{k}(x_0,x) = -\underline{k}(x_1,x_0) + \underline{k}(x_1,x) = \psi^-(x_1) +\underline{k}(x_1,x)
	\end{align*}
	and
	\begin{align*}
		v^-(x) = \inf_{y \in \underline{\mathbb{M}}} \{ \psi^-(y) +   \underline{k}(y,x) \} = \inf_{\substack{y \in \underline{\mathbb{M}} \\ y \nsim x_0}} \{ - \underline{k}(y,x_0) + \underline{k}(y,x) \}
	\end{align*}
	Let $\psi$ be an element of $\dom (\underline{\mathbb{M}}, \underline{k})$ such that $\psi(x_0) = 0$ and $v = \Psi_{\underline{k}}(\psi)$. The domination condition yields
	\begin{align*}
		\psi(x) \geq \sup_{y \in \underline{\mathbb{M}}} \{ \psi(y) - \underline{k}(x,y) \} \geq \psi(x_0) - \underline{k}(x,x_0) = \psi^{-}(x)
	\end{align*}
	Therefore, $v \geq v^-$ and we obtain the desired identity.
\end{proof}

\section{Applications} \label{SectionApplication}

We explore several applications of the representation formula for different examples of Tonelli Hamiltonians.\\

First, we examine autonomous Tonelli Hamiltonians, for which we will prove Fathi's Convergence Theorem \ref{FathiTh}. Next, we will treat the case of Tonelli Hamiltonians with $N$-periodic or $\underline{p}$-recurrent Mather sets.\\

Additionally, the representation formula can be used to identify specific subsets of the non-wandering set $\Omega(\mathcal{T})$, such as $\fix(\mathcal{T})$ and $\per_n(\mathcal{T})$. The representation of weak-KAM solutions of $\fix(\mathcal{T})$ in the non-autonomous case was established by G. Contreras, R. Iturriaga, and H. S\'anchez-Morgado in \cite{CIS}.\\

It should be noted that these applications are not direct results of the theorems but rather adaptations of their proofs.

\subsection{The Autonomous Case}

We consider the Autonomous framework. Let $H : T^*M \to \mathbb{R}$ be a Tonelli Hamiltonian with corresponding Lagrangian $L : TM \to \mathbb{R}$. In this case, the studied objects $\phi_L^{s,t}$, $\mathcal{T}^{s,t}$, $h^{s,t}$,etc.. verify $\phi_L^{s,t} = \phi_L^{t-s}$, $\mathcal{T}^{s,t} = \mathcal{T}^{t-s}$, $h^{s,t} = h^{t-s}$,etc..\\

Note that the minimizing measures used to define the Mather set in \eqref{MatherDefFormula} are defined on $TM$ instead of $\mathbb{T}^1 \times TM$ so that the Mather set $\tilde{\mathcal{M}} =\tilde{\mathcal{M}}_0$ is included in $TM$.\\

Let us first prove the following weaker version of Fathi's Theorem \ref{FathiTh} following ideas of P.Bernard and J-M.Roquejoffre used in \cite{MR2041603}.

\begin{theo} \label{AutonomousRepresentationTheorem}
	For an autonomous Tonelli Hamiltonian $H : T^*M \to \mathbb{R}$, we have $\Omega(\mathcal{T}) = \fix ( \mathcal{T})$ and the representation formula is given by the following bijection
	\begin{equation} \label{AutonomousRepresentationBijection}
		\begin{matrix}
			\Psi_{1}: \dom(\underline{\mathbb{M}}, h^\infty ) & \longrightarrow & \Omega(\mathcal{T}) \\
			\psi & \longmapsto & \inf\limits_{y \in \underline{\mathbb{M}}} \{ \psi(y) + h^\infty(y, \cdot ) \}
		\end{matrix}
	\end{equation}
	with its inverse being the restriction map 
	\begin{equation} \label{AutonomousInverseMap}
		\begin{matrix}
			\Phi_1 : \Omega(\mathcal{T}) & \longrightarrow & \dom(\underline{\mathbb{M}}, h^\infty ) \\
			v & \longmapsto & v_{| \underline{\mathbb{M}}}
		\end{matrix}
	\end{equation}  
\end{theo}

\begin{proof}
	Let $v$ in $\Omega (\mathcal{T})$. Fix a point $x_0 \in \mathcal{M}$ with lift $\tilde{x}_0$ in $\tilde{\mathcal{M}}$, set $x(t) = \pi \circ \phi_L^t(\tilde{x})$ and consider the weak-KAM solution $u(x) = h^\infty(x_0,x)$. We will show that the map $f(t,x) = v(t,x) - u(t,x)$ is constant on $\mathbb{R} \times \{x_0(\tau); \tau \in \mathbb{R}\} \subset \mathbb{R} \times \mathcal{M}$. We know from Proposition \ref{CalibNW} that the curve $x_0(t)$ is calibrated by both $u$ and $v$, and we have
	\begin{equation*}
		u(t,x_0(t)) - u(0,x_0) = h^t(x_0,x_0(t)) = v(t,x_0(t)) - v(0,x_0)
	\end{equation*}
	Hence, we get
	\begin{equation*}
		f(t,x_0(t)) = v(t,x_0(t)) - u(t,x_0(t)) = v(0,x_0) - u(0,x_0) = f(0,x_0)
	\end{equation*}
	and $f(t,x_0(t))$ is constant on time. Moreover, we know from the regularity on calibrated curves Proposition \ref{CalibRegularity} that for all $y \in \mathcal{M}$ with lift $\tilde{y} \in \tilde{\mathcal{M}}$, we have
	\begin{equation*}
		d_xf(t,y) = d_xv(t,y) - d_xu(t,y) = \partial_vL (\tilde{y}) - \partial_vL (\tilde{y}) =0
	\end{equation*}
	Thus, since the curve $x_0(t)$ is $C^1$ regular, we can integrate along it and get for any time $s \in \mathbb{R}$
	\begin{equation*}
		f(t,x_0(s)) = f(t,x_0) + \int_0^s d_xf(t,x_0(\tau)).\dot{x}_0(\tau) \; d\tau = f(t,x_0)
	\end{equation*}
	and we deduce that that map $f(t,x)$ is constant on the set $\mathbb{R} \times \{x_0(\tau); \tau \in \mathbb{R}\}$. In particular, using the fact that $h^{\infty-n}(x_0,x_0) = h^\infty(x_0,x_0) =0$ established in Proposition \ref{MatherPeierls}, we obtain
	\begin{align*}
		v(-n,x_0) &= v(-n, x_0) - u(-n,x_0) = f(-n,x_0) = f(-n,x_0(-n))\\
		& = f(0,x_0) = v(0, x_0) - u(0,x_0) = v(0, x_0) = v(x_0)
	\end{align*}
	which stand for any $x_0 \in \mathcal{M}$. Consequently, and by definition of viscosity solutions, we get for all $(x,y)$ in $\mathcal{M} \times M$
	\begin{align*}
		v(y) - v(x) = v(y) - v(-n,x) \leq h^n(x,y)
	\end{align*}
	and taking the liminf on $n$, we obtain the domination
	\begin{align*}
		v(y) - v(x) \leq h^\infty(x,y)
	\end{align*}
	The identity \eqref{kMaxMinOmega} implies that for any $(x,y)$ in $\mathcal{M} \times M$, 
	\begin{align*}
		\underline{k}(x,y) = \max_{v \in \Omega(\mathcal{T})} \{v(y) - v(x) \} \leq h^\infty(x,y)
	\end{align*}
	and the Property \ref{k>p} of Proposition \ref{kProp} gives the inverse inequality, which results in the equality $\underline{k} = h^\infty$. Therefore, the bijection established in Theorem \ref{kRepresentationTheorem} translates into \eqref{AutonomousRepresentationBijection}. And for all $v = \Psi_1(\psi)$ for some $\psi \in \dom(\underline{\mathbb{M}}, h^\infty)$, Proposition \ref{ViscosityInf} gives
	\begin{align*}
		\mathcal{T}v(x)= v(1,x) &= \inf\limits_{y \in \underline{\mathbb{M}}} \{ \psi(y) + h^{\infty+1}(y, \cdot ) \} \\
		&= \inf\limits_{y \in \underline{\mathbb{M}}} \{ \psi(y) + h^\infty(y, \cdot ) \}  = v(x)
	\end{align*}
	which justifies that $v$ belongs to $\fix (\mathcal{T})$ and hence that $\Omega(\mathcal{T}) = \fix ( \mathcal{T})$.
\end{proof}

\begin{rem}
	\begin{enumerate}
		\item Note from the proof that the key point to obtain weak-KAM solutions is to establish the domination by $h^\infty$. We will see in the next examples that these dominations by different barriers can detect the periodicity of the $\underline{p}$-recurrence of viscosity solutions for some sequence $p$.
		\item This result is a weaker version of the theorem of A.Fathi \cite{MR1650261}, which more generally shows that for all times $t$, $\underline{k}^t(x,,y) = h^\infty(x,y) = \lim_{s \to +\infty} h^s(x,y)$. This indicates that the weak-KAM solutions $v \in \Omega(\mathcal{T})$ are independent of the time $t$ and that they are viscosity solutions of the stationary Hamilton-Jacobi equation
		\begin{equation} \label{HJAutonomous}
			H(x,d_xu) = \alpha_0
		\end{equation} 
	\end{enumerate}	
\end{rem}

\begin{proof}[Proof of Corollary \ref{FathiTh}]
	Fix a time $\tau >0$ and let $\tilde{H}(t,x,p) = H(t+\tau,x,p)$. We add a tilde in the notation of all the objects $\tilde{L}$, $\tilde{\mathcal{T}}$, $\tilde{\alpha}_0$ associated to $\tilde{H}$.
	
	The potential $\tilde{h}_0$ associated to $\tilde{H}$ verifies for all points $x$ and $y$ in $M$ and all times $s<t$
	\begin{align*}
		\tilde{h}_0^{s,t}(x,y) &= \inf \left\{ \int_s^t L(\zeta + \tau, \tilde{\gamma}(\zeta), \dot{\tilde{\gamma}}(\zeta)) \; d\zeta \; \left| \;
		\begin{matrix}
			\tilde{\gamma} : & [s,t] \to M \\
			& s \mapsto x \\
			& t \mapsto y
		\end{matrix} \right.	\right\} \\
		&= \inf \left\{ \int_{s+\tau}^{t+\tau} L(\zeta, \gamma(\zeta), \dot{\gamma}(\zeta)) \; d\zeta \; \left| \;
		\begin{matrix}
			\gamma : & [s+\tau,t+\tau] \to M \\
			& s \mapsto x \\
			& t \mapsto y
		\end{matrix} \right.	\right\} \\
		&=h_0^{s+\tau,t+\tau}(x,y)
	\end{align*}
	and in particular, its Lax-Oleinik operator $\tilde{\mathcal{T}}^t_0$ verifies $\tilde{\mathcal{T}}^t_0 = \mathcal{T}^{\tau,t+\tau}_0$.
	
	Moreover, using the characterization of the \Mane critical value mentioned in Remark \ref{ManeCharacterization}, we deduce that $\tilde{\alpha}_0 = \alpha_0$. Hence, the Peierls barriers verify $\tilde{h}^{\infty} = h^{\tau,\infty+ \tau}$, the full Lax-Oleinik operator verifies $\tilde{\mathcal{T}}^t = \mathcal{T}^{\tau,t+\tau}$, and the non-wandering set of $\tilde{\mathcal{T}}$ is given by $\Omega(\tilde{\mathcal{T}}) = \Omega(\mathcal{T}^{\tau,1+\tau})$.\\

	Recall from Proposition \ref{Isometry} the definition of $\Omega_\tau(\mathcal{T}) = \mathcal{T}^\tau \Omega(\mathcal{T})$. We claim that $\Omega(\tilde{\mathcal{T}}) = \Omega_\tau(\mathcal{T})$. Indeed, since by Proposition \ref{Isometry} the map $\mathcal{T}^\tau$ is invertible in the two sets, we get
	\begin{align*}
		\Omega_\tau(\mathcal{T}) = \mathcal{T}^\tau \Omega(\mathcal{T}) &= \{\mathcal{T}^\tau u \in \mathcal{C}(M, \mathbb{R}) \; | \; u \textit{ is a limit point of } \mathcal{T}^nu = \mathcal{T}^{\tau,n}\mathcal{T}^\tau u\} \\
		&=\{\mathcal{T}^\tau u \in \mathcal{C}(M, \mathbb{R}) \; | \; \mathcal{T}^\tau u \textit{ is a limit point of } \mathcal{T}^\tau \mathcal{T}^{\tau,n}\mathcal{T}^\tau u = \mathcal{T}^{\tau,n+ \tau}\mathcal{T}^\tau u  \} \\
		&=\{v \in \mathcal{C}(M, \mathbb{R}) \; | \; u \textit{ is a limit point of }  \mathcal{T}^{\tau,n+ \tau} v \} \\
		&=\Omega(\mathcal{T}^{\tau,1+\tau}) = \Omega(\tilde{\mathcal{T}})  
	\end{align*}
	
	Therefore, applying the representation formula \ref{AutonomousRepresentationBijection} once for $\Omega(\tilde{\mathcal{T}})$ on $\mathcal{M}_0(\tilde{L})$ and once for $\Omega_\tau(\mathcal{T})$ on $\mathcal{M}_0$, we get the following formulas
	\begin{equation} \label{AutonomousRepresentationBijectiont1}
		\begin{matrix}
			\tilde{\Psi}_{1}: \dom(\mathcal{M}_0(\tilde{L}), h^{\tau,\infty+\tau} ) & \longrightarrow & \Omega_\tau(\mathcal{T}) \\
			\psi & \longmapsto & \inf\limits_{y \in \mathcal{M}_0(\tilde{L})} \{ \psi(y) + h^{\tau,\infty+\tau}(y, \cdot ) \}
		\end{matrix}
	\end{equation}
	and 
	\begin{equation} \label{AutonomousRepresentationBijectiont2}
		\begin{matrix}
			\Psi^\tau_{1}: \dom(\mathcal{M}_0, h^\infty ) & \longrightarrow & \Omega_\tau(\mathcal{T}) \\
			\psi & \longmapsto & \inf\limits_{y \in \mathcal{M}_0} \{ \psi(y) + h^{\infty+\tau}(y, \cdot ) \}
		\end{matrix}
	\end{equation}
	Since we are working in the autonomous case, we have $\tilde{H} = H$ and $\mathcal{M}_0(\tilde{L}) = \mathcal{M}_0$. Hence, Corollary \ref{kMaxMinCorollary} yields for all $x$ in $\mathcal{M}_0$ and $y$ in $M$
	\begin{equation}
		h^\infty(x,y) = h^{\tau,\infty+\tau}(x,y) = \sup_{v \in \Omega_\tau(\mathcal{T})} \{v(y) - v(x)\} = h^{\infty+\tau}(x,y) 
	\end{equation}
	We deduce that all the elements of $\Omega(\mathcal{T})$ are constant in time and hence are solutions of the stationary Hamilton-Jacobi equation \eqref{HJAutonomous}. In particular, $\Omega(\mathcal{T}) = \Omega_\tau(\mathcal{T})$. And Theorem \ref{AutonomousRepresentationTheorem}asserts that $\Omega(\mathcal{T}) = \fix ( \mathcal{T})$, we obtain that $\fix(\mathcal{T}) = \bigcap_{t >0} \fix(\mathcal{T}^t)$.\\
	
	Now let $u$ be a scalar map in $\mathcal{C}(M, \mathbb{R})$. We have from Proposition \ref{LONW} that $\omega(u) \subset \Omega(\mathcal{T}) = \fix(\mathcal{T})$, and by the minimality property, we deduce that $\omega(u)$ is a singleton so that $\mathcal{T}^nu$ converges to a weak-KAM solution $v$. Hence, for all time $t >0$, $\mathcal{T}^{n+t}u$ converges to $v(t,\cdot) = v$. Therefore, we deduce that $\mathcal{T}^nu$ converges to the weak-KAM solution $v$ of the Hamilton-jacobi equation \eqref{HJAutonomous}.
\end{proof}

\subsection{Periodic Viscosity Solutions}

In this subsection, we present a representation formula for periodic viscosity solutions for a fixed integer period $n \geq 1$. This formula generalizes the representation formula for weak-KAM solutions i.e one-time periodic solutions, as established in \cite{CIS}.\\

To derive such a formula, we need to identify the appropriate domination barrier. This is provided by the $n$-barrier $h^{n\infty}$, which was first introduced by A. Fathi and J. N. Mather in \cite{MR1792479}.

\begin{defi} \label{StaticClassesDefi}
	\begin{enumerate}
		\item We define the \textit{$n$-Peierls Barrier} $h^{n \infty} : M \times M \to \mathbb{R}$ as the $\underline{p}$-Peierls barrier for $p_k = n.k$, i.e
		\begin{equation}
			h^{n \infty}(x,y) = \liminf_{p \to \infty} h^{np}(x,y)
		\end{equation}

		\item We define the subset $\mathcal{M}_n$ of the Mather set $\mathcal{M}_0$ by
		\begin{equation}
			\mathcal{M}_n := \{ x \in \mathcal{M}_0 \; | \; h^{n\infty}(x,x) =0 \}
		\end{equation}
		\item We define the map $d_n : \mathcal{M}_n \times \mathcal{M}_n \to \mathbb{R}$ by
		\begin{equation}
			d_n(x,y) = h^{n\infty}(x,y) + h^{n\infty}(y,x)
		\end{equation}
		
		\item We define the equivalence relation $\sim_n$ on $\mathcal{M}_n$ by
		\begin{equation}
			x \sim_n y \Longleftrightarrow d_n(x,y) = 0
		\end{equation}

		\item The \textit{$n$-static classes} are the equivalence classes of the equivalence relation $\sim_n$. We denote by $\mathbb{M}_n$ the set of $n$-static classes. We assume that every class of $\mathbb{M}_n$ is represented by an element of $\mathcal{M}_n$ so that we have the inclusion $\mathbb{M}_n \subset \mathcal{M}_n \subset \mathcal{M}_0$.
		
		\item We say that a map $\psi : X \to \mathbb{R}$ is $n$-dominated on a set $X$ if it is $h^{n\infty}$-dominated on $X$.
	\end{enumerate}
\end{defi}

\begin{rem}
	Analogously to Proposition \ref{kPseudometric}, and due to the triangular inequality \eqref{TriangInegPeierls} satisfied by the $n$-Peierls barrier $h^{n\infty}$, the map $d_n$ is a pseudometric on $\mathcal{M}_0$, which justifies that $\sim_n$ is an equivalence relation.
\end{rem}

We will work on the set $\dom(\mathbb{M}_n, h^{n\infty})$ of $n$-dominated maps on the set of $n$-static classes $\mathbb{M}_n$ to prove Theorem \ref{PeriodicRepresentationTheorem}.

\begin{proof}[Proof of Theorem \ref{PeriodicRepresentationTheorem}]
	We first show that the maps $\Psi_n$ and $\Phi_n$ are well defined. Let $\psi$ be an element of $\dom(\mathbb{M}_n, h^{n\infty})$ and set $v = \Psi_n(\psi)$. We know from Property \ref{pvisc} of Proposition \ref{pPeierlsProp} that $h^{n\infty}(\cdot, y,\cdot)$ is a viscosity solution for all $y \in \mathbb{M}_n$. Hence, we infer from Proposition \ref{ViscosityInf} that $v$ is a viscosity solution and that
	\begin{align*}
		\mathcal{T}^n v(x) &= \inf\limits_{y \in \mathbb{M}_n} \{ \psi(y) + \mathcal{T}^n h^{n\infty}(y, \cdot )(x) \} \\
		&= \inf\limits_{y \in \mathbb{M}_n} \{ \psi(y) + h^{n\infty+n}(y,x) \} \\
		&= \inf\limits_{y \in \mathbb{M}_n} \{ \psi(y) + h^{n\infty}(y,x) \} = v(x)
	\end{align*}
	Thus, $v \in \per_n(\mathcal{T})$ and $\Psi_n$ is well defined.
	
	Now let $v$ be an element of $\per_n(\mathcal{T})$ and $x$ and $y$ be two points of $\mathbb{M}_n$. We have by definition of $v$ that for any integer $k$, 
	\begin{align*}
		v(y) - v(x) = v(nk,y)-v(x) \leq h^{nk}(x,y)
	\end{align*}
	and taking the liminf on $k$ leads to 
	\begin{align*}
		v(y) - v(x) \leq h^{n\infty}(x,y)
	\end{align*}
	We showed that $\Phi_n(v) = v_{|\mathbb{M}_n}$ is $n$-dominated, which justifies the well-definition of the map $\Phi_n$.\\
	
	The proof of the identity $\Psi_n \circ \Phi_n = Id_{\dom(\mathbb{M}_n, h^{n\infty})}$ is analogous what was done in the proof of Theorem \ref{RepresentationTheorem}. We show that $\Phi_n \circ \Psi_n = Id_{\per_n(\mathcal{T})}$. Let $v$ be an element of $\per_n(\mathcal{T})$ and set $w = \Phi_n \circ \Psi_n (v) \in \per_n(\mathcal{T})$. We aim to prove that $w = v$. By the domination condition and following the proof of Theorem \ref{kRepresentationTheorem}, we get successively that $v_{|\mathbb{M}_n} = w_{|\mathbb{M}_n}$ and $v_{|\mathcal{M}_n} = w_{|\mathcal{M}_n}$. Let $x$ be an element of the Mather set $\mathcal{M}_0$ with lift $\tilde{x}$ in $\tilde{\mathcal{M}}_0$ and set $x(t) = \pi \circ \phi_L^t(\tilde{x})$. By compactness of $M$, there exist an increasing sequence $k_i$ of integers such that $\lim_i k_{i+1} - k_i=+\infty$ and $x(-nk_i)$ converges to a point $x_\alpha$ belonging to the closed set $\mathcal{M}_0$. We set $x_\alpha(t) = \pi \circ \phi_L^t(\tilde{x}_\alpha)$ with $\tilde{x}_\alpha$ is the lift of $x_\alpha$ to $\tilde{\mathcal{M}}_0$. By Proposition \ref{CalibNW}, the curve $x_\alpha(t)$ is calibrated by the $n$-periodic viscosity solutions $v$ and $w$. Hence, using the non-negativity \ref{PeierlsPositivity} of Proposition \ref{PeierlsProp} and the liminf Properties \eqref{PeierlsLiminf}, we get
	\begin{align*}
		0 \leq h^{n\infty}(x_\alpha,x_\alpha) &\leq \lim_i h^{n(k_{i+1}-k_i)}(x(-nk_{i+1}),x(-nk_i)) \\
		&= \lim_i v(-nk_i,x(-nk_i)) - v(-nk_{i+1},x(-nk_{i+1})) \\
		&= \lim_i v(0,x(-nk_i)) - v(0,x(-nk_{i+1})) \\
		&= v(x_\alpha) - v(x_\alpha) = 0
	\end{align*}
	Thus, $x_\alpha$ belongs to $\mathcal{M}_n$. Therefore, we obtain that $v(x_\alpha) = w(x_\alpha)$, and by calibration 
	\begin{align*}
		w(x) & = w(-nk_i,x(-nk_i)) + h^{nk_i}(x(-nk_i),x) \\
		&= w(x(-nk_i)) + h^{nk_i}(x(-nk_i),x) \\
		&= \lim_i w(x(-nk_i)) + h^{nk_i}(x(-nk_i),x) \\
		&= \lim_i w(x_\alpha) + h^{nk_i}(x(-nk_i),x) \\
		&= \lim_i v(x_\alpha) + h^{nk_i}(x(-nk_i),x) \\
		&= \lim_i v(x(-nk_i)) + h^{nk_i}(x(-nk_i),x) = v(x)
	\end{align*}
	which yields $v_{|\mathcal{M}_0} = w_{|\mathcal{M}_0}$ and by the Uniquenesse Theorem \ref{Uniqueness}, $v= w = \Phi_n \circ \Psi_n (v)$.
\end{proof}

\begin{rem} \label{PeriodicRemark}
	\begin{enumerate}
		\item We can derive from this proof and that of Theorem \ref{Uniqueness} the following uniqueness theorem for periodic viscosity solutions
		\begin{prop} \label{UniquenessPeriodic}
			For any viscosity solutions $v \in \per_n(\mathcal{T})$ and $w \in \Omega (\mathcal{T})$, if $v_{|\mathcal{M}_n} = w_{|\mathcal{M}_n}$, then $v=w$.
		\end{prop}
		
		\item The generalization presented here can be considered a direct application of the representation of weak-KAM solutions demonstrated in \cite{CIS}, as $n$-periodic solutions can be viewed as weak-KAM solutions of the Hamiltonian $H_n = nH(nt,x,p)$.
	\end{enumerate}
\end{rem}

This theorem allows us to more easily describe the non-wandering set of systems in which every element of the Mather set is periodic with a uniform integer period, as stated in Corollary \ref{PeriodicMatherOmega}.

\begin{proof}[Proof of Corollary \ref{PeriodicMatherOmega}]
	\ref{PeriodicMatherOmega}. It suffices to prove that for any $v \in \Omega(\mathcal{T})$, $v_{|\mathcal{M}_0}$ is $N$-dominated on $\mathcal{M}_0$. Let $x$ and $y$ be two points of $\mathcal{M}_0$. Let $\tilde{y}$ be the lift of $y$ in $\tilde{\mathcal{M}}_0$ and set $y(t) = \pi \circ \phi^t_L(\tilde{y})$. If $u$ is a weak-KAM solution, we know from Proposition \ref{CalibNW} that $y(t)$ is calibrated by both $u$ and $v$. Hence, for all integer $k \geq 0$, we have 
	\begin{equation*}
		\begin{split}
			v(kN,y) -v(0,y) &= v(kN,y(kN)) -v(0,y) = h^{kN}(y,y(kN)) \\
			&= u(kN,y(kN)) -u(0,y) = u(0,y) - u(0,y) = 0
		\end{split}
	\end{equation*}
	where we used the $N$-periodicity of $y(t)$ and the $1$-time periodicity of the weak-KAM solution $u$. Thus, by definition of viscosity solutions, we get
	\begin{align*}
		v(y) - v(x) = v(kN,y) - v(x) \leq h^{kN}(x,y)
	\end{align*}
	and taking the liminf on $k$, we conclude that
	\begin{align*}
		v(y) - v(x) \leq  h^{N\infty}(x,y)
	\end{align*}
	and in particular, $v_{|\mathbb{M}_N}$ belongs to $\dom (\mathbb{M}_N , h^{N\infty})$. Therefore, $v = \Psi_N( v_{|\mathbb{M}_N} )$ belongs to $\per_N(\mathcal{T})$.  and we conclude that $\Omega( \mathcal{T}) = \per_N (\mathcal{T})$.
\end{proof}

\subsection{$\underline{p}$-Recurrent Viscosity Solutions}

In this section, we fix a sequence $\underline{p}$ of increasing positive integers and we discuss to which extend it is possible to represent $\underline{p}$-recurrent viscosity solutions.

\subsubsection{\underline{p}-Domination}

We define the natural domination notion associated to $\underline{p}$-recurrence.

\begin{defi} \label{pdom}
	a map $\psi : X \to \mathbb{R}$ is said \textit{$\underline{p}$-dominated} in a set $X \subset M$ if it is $h^{\underline{p}}$-dominated.
\end{defi}

Then we have the following

\begin{prop} \label{RecurrentpDomination}
	Let $v \in \Omega(\mathcal{T})$ be a $\underline{p}$-recurrent viscosity solution. Then $v$ is $\underline{p}$-dominated in $M$.
\end{prop}

\begin{proof}
	By definition of viscosity solution, we have for any $x$ and $y$ in $M$
	\begin{align*}
		v(y) - v(x) = \lim_n v(p_n,y) - v(x) \leq \liminf_n h^{p_n}(y,x) = h^{\underline{p}}(x,y)
	\end{align*}
\end{proof}

\begin{rem}(Failure of Reprenting $\underline{p}$-recurrent viscosity solutions)
	\begin{enumerate}
		\item (Failure of Uniqueness) Unlike what was observed for periodic viscosity solutions in Remark \ref{PeriodicRemark}, there is no uniqueness theorem in the set
		\begin{equation}
			\mathcal{M}_{\underline{p}} := \{ x \in \mathcal{M}_0 \; | \; h^{\underline{p}}(x,x) = 0 \}
		\end{equation}
		Hence, if one hopes to get a representation formula of $\underline{p}$-recurrent viscosity solutions, it must be on the entire Mather set $\mathcal{M}_0$.
		\item (Failure of $\underline{p}$-recurrence for $h^{\underline{p}}$) If we define the injective map
		\begin{equation}
			\begin{matrix}
			\tilde{\Psi}_{\underline{p}} \dom(\mathcal{M}_0, h^{\underline{p}}) & \longrightarrow & \Omega(\mathcal{T}) \\
			\psi & \longmapsto & \inf\limits_{y \in \mathcal{M}_0} \{ \psi(y) + h^{\underline{p}}(y, \cdot ) \}
		\end{matrix}
		\end{equation}
		Then, by the $\underline{p}$-domination Proposition \ref{RecurrentpDomination}, we get that all $\underline{p}$-recurrent maps belong to the image $\tilde{\Psi}_{\underline{p}}(\dom(\mathcal{M}_0, h^{\underline{p}}))$. However, the $\underline{p}$-recurrence of $h^{\underline{p}}(x,\cdot)$ does not generally hold if $x$ is not $\underline{p}$-recurrent under the projected Lagrangian flow. Even if this condition were satisfied, there is no guarantee that the infimum $v$ of $\underline{p}$-recurrent viscosity solutions $v_n$ would be $\underline{p}$-recurrent, unless there is uniform $\underline{p}$-recurrence for $v_n$.
	\end{enumerate}	
\end{rem}

Nevertheless, a representation formula using the $\underline{p}$-barrier $h^{\underline{p}}$ is possible when it coincides with the general barrier $\underline{h}$ defined in Subsection \ref{Sectionhbar}. In the next subsections, we will consider cases where we can choose $\underline{h} = \underline{h}^p$.

\subsubsection{The Representation Formula for Mather Sets with $\underline{p}$-Recurrent Elements}

In this section, we assume that there exists a sequence $\underline{p}$ of increasing positive integers for which the Mather set $\tilde{\mathcal{M}}_0$ verifies
\begin{equation}
	\forall \tilde{x} \in \tilde{\mathcal{M}}_0, \quad \lim_n \phi_L^{-p_n}( \tilde{x} ) = \tilde{x}
\end{equation}

\begin{prop} \label{RecurrenthpProp}
	In this case, for all $x$ and $y$ in $\mathcal{M}_0$ and $z$ in $M$
	\begin{equation} \label{RecurrentTriangIneg}
		h^{\underline{p}}(x,z) \leq h^{\underline{p}} (x,y) + h^{\underline{p}}(y,z)
	\end{equation}
	and the following equality holds
	\begin{equation} \label{Recurrenthp=k}
		\underline{k} = h^{\underline{p}}_{| \mathcal{M}_0 \times M}
	\end{equation}
\end{prop}

\begin{proof}
	We choose $\underline{h}(x,y) = h^{\underline{p}}(x,y)$. Fix three $x$, $y$ in $\mathcal{M}_0$ and $z$ in $M$. The triangular inequality \eqref{TriangIneg} on the potential $h$ gives for all integer $n \geq 0$,
	\begin{equation*}
		h^{\underline{p} + p_n} (x,z) \leq h^{\underline{p}} (x,y) + h^{p_n}(y,z)
	\end{equation*}
	We know from Property \ref{hpRecurrence} of Proposition \ref{hpProp} that the viscosity solution $\underline{h}(x,\cdot) = h^{\underline{p}}(x,\cdot)$ is $\underline{p}$-recurrent. Thus, taking the liminf on $n$ in the preceding inequality yields
	\begin{align*}
		h^{\underline{p}}(x,z) = \lim_n h^{\underline{p} + p_n} (x,z) & \leq h^{\underline{p}} (x,y) + \liminf_n h^{p_n}(y,z) = h^{\underline{p}} (x,y) + h^{\underline{p}}(y,z)
	\end{align*}
	
	The triangular inequality results in $\underline{h} \leq \underline{k}$ on $\mathcal{M}_0^R \times M= \mathcal{M}_0 \times M$. And since the inverse inequality is given by \eqref{k>pFormula}, we deduce the equality $ h^{\underline{p}} = \underline{h} = \underline{k}$ on $\mathcal{M}_0 \times M$.
\end{proof}

Following this proposition, we can apply Theorem \ref{kRepresentationTheorem} to obtain

\begin{theo} \label{RecurrentRepresentationTheorem}
	In this case, the representation formula is given by the following bijection
	\begin{equation} \label{RecurrentRepresentationBijection}
		\begin{matrix}
			\Psi_{\underline{p}}: \dom(\underline{\mathbb{M}}, h^{\underline{p}}) & \longrightarrow & \Omega(\mathcal{T}) \\
			\psi & \longmapsto & \inf\limits_{y \in \underline{\mathbb{M}}} \{ \psi(y) + h^{\underline{p}}(y, \cdot ) \}
		\end{matrix}
	\end{equation}
	with its inverse being the restriction map
	\begin{equation} \label{RecurrentRepresentationBijection}
		\begin{matrix}
			\Phi_{\underline{p}}: \Omega(\mathcal{T}) & \longrightarrow &  \dom(\underline{\mathbb{M}}, h^{\underline{p}}) \\
			v & \longmapsto & v_{|\underline{\mathbb{M}}}
		\end{matrix}
	\end{equation}    
\end{theo}

\begin{rem}
	\begin{enumerate}
		\item Note that, even if every recurrent viscosity solutions $v \in \Omega(\mathcal{T})$ are expressed as the infimum of $\underline{p}$-recurrent ones, it does not imply that $v$ is itself $\underline{p}$-recurrent. This is due to the non-uniform $\underline{p}$-recurrent of the viscosity solutions $h^{\underline{p}}(x,\cdot)$ for $x \in \mathcal{M}_0$.
	
		Finding a sequence $\underline{q}$ such that $v$ is $\underline{q}$-recurrent remains highly non-trivial, even in this case. 
		
		\item It is possible to verify through the proof of Theorem \ref{kRepresentationTheorem} that this formula remains valid if we assume that for all $\tilde{x} \in \tilde{\mathcal{M}}_0$, either $\lim_n \phi_L^{+ p_n}(\tilde{x}) = \tilde{x}$ or $\lim_n \phi_L^{- p_n}(\tilde{x}) = \tilde{x}$. Positive or negative time recurrence does not matter for $h^{\underline{p}}$ to satisfy the triangular inequality. Hence, $d_{\underline{p}}(x,y) = h^{\underline{p}}(x,y) + h^{\underline{p}}(y,x)$ is a pseudometric on $\mathcal{M}_0$, making it possible to define its associated equivalence relation $\sim_{\underline{p}}$ and $\underline{p}$-static classes $\mathbb{M}_{\underline{p}}$. Since a uniqueness result on $\mathbb{M}_{\underline{p}}$ holds, the remainder of the proof follows analogously.
		
		\item This Representation Formula is still valid if we only had $\phi_L^1$-recurrence on a dense subset $\mathcal{M}'_0$ of $\mathcal{M}_0$. In this case, we would still have by continuity that $\underline{k} = h^{\underline{p}}$ and a uniqueness theorem on $\mathcal{M}'_0$.
		
		\item Following this last remark, this case includes the situation where the Mather set contains a dense set of periodic curves with integer periods. If $q_n$ is the sequence of these periods, we take $p_n = \prod_{k=0}^n q_k$.
	\end{enumerate}
\end{rem}

\subsubsection{The Representation Formula for Mather Sets with Uniformly $\underline{p}$-Recurrent Elements}

We assume that there exists an increasing sequence of positive integers $\underline{p}$ such that $\phi^{-p_n}_{L|\tilde{\mathcal{M}}_0}$ uniformly converges to the identity.

\begin{prop} \label{UniformlyRecurrenthp}
	In this case, the maps $h^{\underline{p}}(x, \cdot)$ are uniformly $\underline{p}$-recurrent. Their convergence speed is controlled by the convergence of $\phi^{-p_n}_{L|\tilde{\mathcal{M}}_0}$ to the identity. More precisely, if for any $y \in \mathcal{M}_0$ with lift $\tilde{y}$ in $\tilde{\mathcal{M}}_0$ we set the curve $y(t) = \pi \circ \phi_L^t(\tilde{y})$, then
	\begin{equation}
		\Vert \underline{h}^{p_n}(x,\cdot) - \underline{h}(x,\cdot) \Vert_\infty \leq 2\kappa_1.  \sup_{y \in \mathcal{M}_0} d(y(-p_n),y)
	\end{equation}
\end{prop}

\begin{proof}
	This is a direct consequence of identity \eqref{hpRecurrenceFormula}.
\end{proof}

The representation formula results in the following corollary which implies Corollary \ref{UniformlyRecurrentCor}.

\begin{cor}
	All the elements $v$ of $\Omega(\mathcal{T})$ are $\underline{p}$-recurrent and
	\begin{equation}
		\Vert v(p_n,\cdot) - v \Vert_\infty \leq 2\kappa_1.  \sup_{y \in \mathcal{M}_0} d(y(-p_n),y)
	\end{equation}
\end{cor}

\begin{proof}
	Let $v$ be an element of $\Omega( \mathcal{T})$. By the Representation Theorem \ref{RecurrentRepresentationTheorem}, there exists a $\underline{p}$-dominated map $\psi \in \dom (\underline{\mathbb{M}}, h^{\underline{p}})$ such that $v = \Psi_{\underline{p}}(\psi)$. Then, Proposition \ref{ViscosityInf} shows that for all positive integer $n$
	\begin{align*}
		\mathcal{T}^nv(x) & = \inf_{y \in \underline{\mathbb{M}} } \; \big\{ \psi(y) + \mathcal{T}^n h^{\underline{p}}(y,\cdot)(x)  \}  \big\} \\
			&= \inf_{y \in \underline{\mathbb{M}} } \; \big\{ \psi(y) + h^{\underline{p}+n}(y,x)   \big\}
	\end{align*}
	Moreover, we infer from Proposition \ref{UniformlyRecurrenthp} that for all $x \in M$, $y \in \mathcal{M}_0$ and $n \geq 0$
	\begin{align*}
		\big| [\psi(y) + h^{\underline{p} + p_n}(y,x)] - [\psi(y) + h^{\underline{p}}(y,x)] \big|&=  | h^{\underline{p} + p_n}(y,x) -  h^{\underline{p}}(y,x) | \leq 2\kappa_1. \sup_{z \in \mathcal{M}_0} d(z(p_n),z)
	\end{align*}
	where the bound is uniform on $y \in \mathcal{M}_0$. Taking the infimum on $y$ yields
	\begin{align*}
		||v(p_n) - v ||_\infty \leq 2\kappa_1. \sup_{z \in \mathcal{M}_0} d(z(p_n),z) \longrightarrow 0 \quad \text{as } n \to \infty
	\end{align*}
	and $v$ is a $\underline{p}$-recurrent viscosity solution.
\end{proof}

\paragraph{Acknowledgement.} The author is extremely grateful to Marie-Claude Arnaud for her careful reading, and to Jaques Féjoz, Jean-Michel Roquejoffre and Maxime Zavidovique for very informative discussions.

\bibliographystyle{alpha}
\addcontentsline{toc}{section}{References}
\bibliography{Biblio}

\end{document}